%% file: Ex__CMC_foliation_hb.tex
\documentclass{amsart}%
\input{./header__arxiv.tex}
\title[CMC-Foliations for asymptotically hyperbolic manifolds]{Existence and uniqueness of\\constant mean curvature foliations of\\general asymptotically hyperbolic $\boldsymbol3$-manifolds}
 \author[Christopher Nerz]{Christopher Nerz}
 \address{Department of mathematics\\Royal institut of Technology KTH\\Stockholm\\Sweden}
 \email{ncroman@kth.se}
 \date\today
\ifpdf
	\hypersetup{pdftitle		= {Existence and uniqueness of constant mean curvature foliation of general asymptotically hyperbolic 3-manifolds},
							pdfauthor		= {Christopher Nerz},
							pdfsubject	= {Any asymptotically flat Riemannian manifold posses a CMC-foliation},
							pdfkeywords	= {CMC-foliation, existence, uniqueness, constant mean curvature foliation, CMC, CMC-surface, asymptotically flat, CMC-center of mass, center of mass}}
\fi

\gdef\ch{{\operatorname{ch}}}
\gdef\sh{{\operatorname{sh}}}

\makeatletter
\gdef\Oof@switch#1#2#3{{#1}{#3}{#2}}
\NewDocumentCommand\Oof{o}{\IfValueTF{#1}{\Oof@snd{O}[#1]}{\Oof@upper{O}{#1}{#1}}}
\NewDocumentCommand\oof{o}{\IfValueTF{#1}{\Oof@snd{o}[#1]}{\Oof@upper{o}{#1}{#1}}}
\gdef\Oof@precheck#1#2#3{\IfValueTF{#2}{\IfValueTF{#3}{\Oof@snd{#1}{#2}{#3}}{\Oof@lower{#1}{#2}{#3}}}{\Oof@upper{#1}{#2}{#3}}}
\NewDocumentCommand\Oof@upper{mmmt^}{\IfBooleanTF{#4}{\expandafter\Oof@precheck\Oof@switch{#1}{#2}}{\Oof@lower{#1}{#2}{#3}}}
\NewDocumentCommand\Oof@lower{mmmt_}{\IfBooleanTF{#4}{\Oof@precheck{#1}{#2}}{\Oof@snd{#1}{#2}{#3}}}
\NewDocumentCommand\Oof@snd{mmmd()}{\IfValueTF{#4}{\Oof@thd{#1}{#2}{#3}{#4}}{\Oof@thd{#1}{#2}{#3}}}
\NewDocumentCommand\Oof@thd{mmmg}{\mathcal#1\IfValueTF{#2}{^{#2}}\relax\IfValueTF{#3}{_{#3}}{_0}\IfValueTF{#4}{(#4)}\relax}
\makeatother

\begin{document}
\begin{abstract}
In 1996, Huisen-Yau proved that every three-dimensional, asymptotically Schwarzschilden manifold with positive mass is uniquely foliated by stable spheres of constant mean curvature and they defined the \emph{center of mass} using  this \emph{CMC-foliation}. Rigger and Neves-Tian showed in 2004 and 2009/10 analogous existence and uniqueness theorems for three-dimensional, asymptotically \emph{Anti-de Sitter} and asymptotically \emph{hyperbolic} manifolds with positive mass aspect function, respectively. Last year, Cederbaum-Cortier-Sakovich proved that the CMC-foliation characterizes the center of mass in the hyperbolic setting, too. In this article, the existence and the uniqueness of the CMC-foliation are further generalized to the wider class of asymptotically hyperbolic manifolds which do not necessarily have a well-defined mass aspect function, but only a timelike mass vector. Furthermore, we prove that the CMC-foliation also characterizes the center of mass in this more general setting.
\end{abstract}
\maketitle
\let\sc\scalar
\let\oldexp\exp
\undef\exp
\NewDocumentCommand\exp{s}{\IfBooleanTF{#1}\oldexp\expsnd}
\NewDocumentCommand\expsnd{d()}{\IfNoValueTF{#1}\expthd{e^{#1}}}
\NewDocumentCommand\expthd{t-m}{\IfBooleanTF{#1}{\exp({-}#2)}{\exp(#2)}}
\gdef\Eoutric{\outsymbol{\mathcal E}_{\bar{\mathcal R}}}
\gdef\Eoutsc{\outsymbol{\mathcal E}_{\bar{\mathcal S}}}
\gdef\ie{\hbox{i.e.}\nopagebreak\xspace}
\gdef\eg{\hbox{e.g.}\nopagebreak\xspace}
\gdef\etc{\hbox{etc.}\nopagebreak\xspace}
\gdef\resp{\hbox{resp.}\nopagebreak\xspace}
\gdef\thesubsection{\thesection.\alph{subsection}}
\section{Introduction}\bgroup\def\thetheorem{\arabic{theorem}}
Huisken-Yau proved 1996 that manifolds which are asymptotic to the spatial Schwarzschild metric with positive mass possess a foliation by stable constant mean curvature (CMC) hypersurfaces, \cite{huisken_yau_foliation}. They used this foliation as a definition for the center of mass of the manifold and also gave a coordinated version of this center. Since then, this foliation proved to be a suitable tool for the study of asymptotically Euclidean (\ie asymptotically flat Riemannian) manifolds and several generalizations of Huisken-Yau's result were made, \eg by Metzger, Huang, Eich\-mair-Metzger, and the author, \cite{metzger2007foliations,Huang__Foliations_by_Stable_Spheres_with_Constant_Mean_Curvature,metzger_eichmair_2012_unique,nerz2015CMCfoliation,nerz2015GeometricCharac}.
In 2004, Rigger used Huisken-Yau's method---the mean curvature flow---to prove the existence and uniqueness of such a foliation for manifolds asymptotic to the spatial Anti-de Sitter-Schwarzschild solution, \cite{rigger2004foliation}. This result was generalized by Neves-Tian and Chodosh, \cite{NevesTianExistenceCMC_I,NevesTianExistenceCMC_II,chodosh2014large}.

In this article, we generalize the existence and uniqueness to the setting of asymptotically hyperbolic manifolds for which the full non-constant part of the curvature is unknown. In particular, the mass aspect tensor of these manifolds is \emph{not} necessarily well-defined, as it was assumed by all results mentioned.\footnote{As explained in Subsection~\ref{ResultsSoFar} Chodosh did not use the mass aspect tensor to prove existence, \ie convergence of the $\sinh(\rad)^{-3}$-term of the metric, but only its boundedness. However his proof of uniqueness of the foliation needs the strong assumptions that the manifold is a compact deformation of the Anti-de Sitter-Schwarzschild spatial solution.} Our method of proof furthermore generalizes the result to manifolds with past-pointing mass vector. Furthermore, we prove uniqueness of the CMC-leaves in a wide class of surfaces which was previously used by Neves-Tian in the restrictive case of a manifold asymptotic to the spatial Anti-de Sitter-Schwarzschild solution. The combination of our existence and uniqueness proof also implies that the CMC-foliation is stable under perturbation of the metric.

As an additional result, we show that the center of mass defined by Cederbaum-Cortier-Sakovich is in our general setting characterized by the CMC-foliation, too---as Cederbaum-Cortier-Sakovich proved in the setting of a well-defined mass aspect tensor, \cite{cederbaum2015center}. Combined with \cite[Thm~5.1]{cederbaum2015center}, this characterizes also the evolution of the CMC-foliation in time, see Remark~\ref{EvolutionInTime}.

Finally, we prove a new regularity theorem for large, almost umbilic hypersurfaces in the hyperbolic space.

\subsection{The main results}
Let us now state the main results in a simplified version---for the full theorems see Theorems~\ref{Regularity_Theorem}, \ref{Existence_Theorem_full}, \ref{Uniqueness_Theorem_full}, \ref{Regularity_Theorem_full}, \ref{Stability_Theorem_full}, and~\ref{W2pSurfaceRegularity}. The meaning of the assumed decay rates are discussed in Subsection~\ref{DecayRates} and the corresponding formal definitions are given in Definitions~\ref{Ck_asymptotically_hb} and~\ref{Definition_Mass}. In Subsection~\ref{ResultsSoFar}, we compare our results with the previous ones.
\begin{theorem}[{Existence of the CMC-foliation, see Theorem~\ref{Existence_Theorem_full}}]\label{Existence_Theorem}
Let $(\outM,\outg*)$ be a $\Ck^2_{\decay,\scdecay}$-asymptotically hyperbolic manifold with decay rates $\decay>\frac52$ and $\scdecay>3$. If the mass vector is timelike, then $\outM$ (outside of a compact set) is foliated by surfaces $\M<\Hradius>$ of constant mean curvature $\H<\Hradius>\equiv{-}2\:\frac{\cosh(\Hradius)}{\sinh(\Hradius)}$. Furthermore, the hyperbolic coordinate center of $\M<\Hradius>$ converges to the hyperbolic center of mass.
\end{theorem}
Here, we used the definition of hyperbolic coordinate center and hyperbolic center of mass by Cederbaum-Cortier-Sakovich, \cite{cederbaum2015center}, see Definition~\ref{center}.

\begin{theorem}[Uniqueness of the CMC-foliation, see Theorem~\ref{Uniqueness_Theorem_full}]\label{Uniqueness_Theorem}
Let $(\outM,\outg*)$ be a $\Ck^2_{\decay,\scdecay}$-asymptotically hyperbolic manifold with decay rates $\decay\in\interval{\frac52}*3$ and $\scdecay>3$. If the mass vector is timelike, then the CMC-surfaces constructed in  Theorem~\ref{Existence_Theorem} are unique within the class
\[ \mathcal M :=
 \lbrace \sphere\cong\M\hookrightarrow\outM \ \middle|\
  \begin{aligned}
		\rradius_0\le\max_{\M}\rad \le \zeta'\min_{\M}\rad,\quad \H(\M)\equiv{-}2\,\frac{\cosh(\Hradius)}{\sinh(\Hradius)}, \\
		\M\text{ \normalfont has }{-}(4-\eta)\,\sinh(\Hradius)^{{-}2}\text{\normalfont-controlled instability}, \\
		\c^{-1}\exp(2\Hradius)\le\volume{\M}\le\c\,\exp(2\Hradius)\text{ \normalfont or } \eta>2\end{aligned}\ \rbrace,
\]
for every $\c\ge0$, $\eta\in\interval0*4$, and $1\le\zeta'<\min\lbrace\frac{4\decay-5}{2\decay},\frac{\scdecay}3\rbrace$, where $\rradius_0:=\Cof{\rradius_0}[\decay][\scdecay][\c][\eta][\zeta']$ is finite.
\end{theorem}
The definition of \emph{controlled instability} is given in Definition~\ref{ControlledInstability}, but let us note here that it is a weaker assumption than stability.
In Remark~\ref{Compare_uniqueness}, we explain the differences between our uniqueness result and the ones by Neves-Tian.\smallskip

As Rigger and Neves-Tian, we also prove regularity estimates for the CMC-leaves, see below. Note that we prove an analogous theorem (not including the estimates on $\ewjac_i$) for a wide class of CMC-surfaces in $\Ck^2_{\decay}$-asymptotically hyperbolic manifolds with decay rate $\decay>2$ (not only $\decay>\frac52$)---if there exists such a CMC-surface for such a decay rate, see Theorem~\ref{Regularity_Theorem}. We explain in Remark~\ref{Expectation_on_k} why it is justified to expect that the regularity given here is optimal, \ie that for any decay rate between $\decay\in\interval2*3$ the regularity cannot be strengthened without additionally assumptions on $\outg$.
\begin{theorem}[Regularity of the CMC-leaves, see Theorems~\ref{Regularity_Theorem} and \ref{Regularity_Theorem_full}]\label{Regularity_Theorem_short}
Let $p\in\interval1\infty$ be a constant and $(\outM,\outg*)$ be a $\Ck^2_{\decay,\scdecay}$-asymptotically hyperbolic manifold with decay rates $\decay\in\interval*{\frac52+\outve}*3$ and $\scdecay=3+\outve>3$. If the mass vector is timelike, then the leaves $\M<\Hradius>$ of the CMC-foli\-ation are almost umbilic, $\zFundtrf<\Hradius>=\Oof(\exp({-}\decay\Hradius))$, and almost round, $\g<\Hradius>=\sinh(\Hradius)^2(\sphg+\Oof_2(\exp((2-\decay)\Hradius)))$. Furthermore, the three eigenvalues $\ewjac_i$ of the stability operator with smallest absolute value satisfy $\ewjac_i=\frac{6\mass*}{\sinh(\Hradius)^3}+\Oof{\exp({-}(3+\outve)\Hradius)}$. In particular, these surfaces are stable if and only if the mass vector is future-pointing.

The surfaces $\M<\Hradius>$ are graphs over (hyperbolic) geodesic spheres, $\M<\Hradius>=\hgraph\graphf<\Hradius>$, and their graph functions $\graphf<\Hradius>\in\Wkp^{2,p}(\hsphere_\Hradius(\hcenterz<\Hradius>))$, their hyperbolic center $\hcenterz<\Hradius>$, and their metric $\g<\Hradius>$ satisfy $\graphf<\Hradius>=\Oof_2{\exp((2-\decay)\Hradius)}$, $\hcenterz<\Hradius>=\houtcenterz+\Oof(\exp({-}\outve\Hradius))$, and $\g<\Hradius>=\sinh(\Hradius+\graphf<\Hradius>)^2(\sphg+\Oof_2(\exp({-}(2+\outve)\Hradius)))$, respectively, where $\houtcenterz$ denotes the hyperbolic center of mass.
\end{theorem}
In particular, the above theorem gives a quantitative version of the statement \lq the hyperbolic coordinate center of $\M<\Hradius>$ converges to the hyperbolic center of mass\rq\ from Theorem~\ref{Existence_Theorem}.\pagebreak[3]

Our method of proof implies a stability of the surfaces under perturbations of the metric---the precise statement can be found in Theorem~\ref{Stability_Theorem_full}.
\begin{theorem}[Stability of the CMC-foliation, see Theorem~\ref{Stability_Theorem_full} and Corollary~\ref{Stability_Corollary_full}]
Let $(\outM[1],\outg[1])$ and $(\outM[2],\outg[2])$ be two $\Ck^2_{\decay,\scdecay}$-asymptotically hyperbolic Riemannian metrics with decay rates $\decay>\frac52$ and $\scdecay>3$. In balanced coordinates, each CMC-leaf $\M[1]<\Hradius>$ (CMC with respect to $\outg[1]$) is as $\Wkp^{2,p}$-close to the corresponding CMC-leaf $\M[2]<\Hradius>$ (CMC with respect to $\outg[2]$) as the metrics are $\Ck^1_{\decay}$-close to each other.

In particular, if a sequence of uniformly $\Ck^2_{\decay,\scdecay}$-asymptotically hyperbolic Riemannian metrics $\outg[n]$ on $\R^3\setminus\overline{B_1(0)}$ converges to a $\Ck^2_{\decay,\scdecay}$-asymptotically hyperbolic Riemannian metric $\outg$ with respect to the $\Ck^1_{\decay}$-topology and the mass vectors of $\outg[n]$ and $\outg$ are balanced, \ie $\mass[i]=(\mass[i]_0,0,0,0)$ and $\mass=(\mass_0,0,0,0)$, then each sequence of CMC-leaves $\M[i]<\Hradius>$ converge in $\Wkp^{2,p}$ to the CMC-leaf $\M<\Hradius>$ of the limit metric and this convergence is uniform in $\Hradius$.
\end{theorem}
Note that this also implies a characterization of the evolution of the CMC-leafs in time (under the Einstein equations), see Remark~\ref{EvolutionInTime}.\smallskip

Finally, one of the central steps in our argument may be worth noting by itself: it is a regularity theorem for large, (pointwise) nearly umbilic surfaces with (pointwise) nearly constant mean curvature in the hyperbolic space. Although, this theorem seems to be very natural, it is a cruicial step in the proof of all the other theorems. In fact, it contains one of the most central improvements of this article compared to the one by Neves-Tian, \cite{NevesTianExistenceCMC_II}.
\begin{theorem}[Regularity of almost umbilic surfaces in \texorpdfstring{$\hyperbolicspace^3$}{hyperbolic space}, see~Theorem~\ref{W2pSurfaceRegularity}]
Let $\M\hookrightarrow\hyperbolicspace^3$ be a hypersurface in the hyperbolic space $\hyperbolicspace^3=(\R^3,\houtg)$. Assume that the hyperbolic mean curvature of $\M$ is very close to a constant, $\hH={-}2\,\frac{\cosh(\Hradius)}{\sinh(\Hradius)}+\Oof(\exp((2+\outve)\Hradius))$,\vspace{-.2em} $\M$ is pointwise almost umbilic, $\hzFundtrf=\Oof(\exp((1+\outve')\Hradius))$, is large, \ie $\Hradius\gg1$. If $\M$ is also bounded by two geodesic spheres of radius $\frac12\Hradius$ and $\frac32\Hradius$ around some point $p\in\hyperbolicspace^3$, then it is a graph over a geodesic sphere, $\M=\graph\graphf$, this graph function $\graphf$ satisfies $\graphf=\Oof_2(\exp({-}\outve\Hradius))$, and $\M$ is as close to being umbilic as its mean curvature is close to be constant, \ie $\hzFundtrf=\Oof(\exp((2+\outve)\Hradius))$.
\end{theorem}
Note that the results $\graphf=\Oof_2(\exp({-}\outve\Hradius))$ and $\hzFundtrf=\Oof(\exp((2+\outve)\Hradius))$ are not necessarily pointwise true, but only in a Sobolev/Lebesgue sense.

\subsection{Decay rates}\label{DecayRates}
Throughout the paper, we study asymptotically hyperbolic manifolds. This means, there is a given Riemannian manifold $(\outM,\outg*)$ and a coordinate system of $\outM$ (outside of some compact set) given such that $\outg$ in these coordinates is asymptotically equal to the hyperbolic metric $\houtg$. More precisely, we compare the metric $\outg*$ with a reference metric $\refoutg$ (mostly the hyperbolic metric) and assume that 
\[ \outg=\refoutg+\Oof_k{\exp({-}\decay\rad)}, \qquad
	\outsc=\refoutsc+\Oof{\exp(-\scdecay\rad)}, \]
where $\decay\le\scdecay$ and $k$ are called \emph{decay rates} (of the metric and the scalar curvature) and \emph{decay order}, respectively. Here, $\Oof_k(\exp({-}\tau\rad))$ denotes a tensor $E$ for which $\vert E\vert_{\houtg*}$, $\vert\houtlevi*E\vert_{\houtg*}$, $\vert\houtlevi*\houtlevi*E\vert_{\houtg*}$, \dots, $\vert\houtlevi*^{k}E\vert_{\houtg*}$ decay at infinity, \ie for $\rad\to\infty$, at least as fast as $\exp({-}\tau\rad)$. Note that in many cases (\eg in this article), we can weaken the above assumption by only assuming
\[ \outg=\refoutg+\Oof_k{\exp({-}\decay\rad)}, \qquad
	\outsc\ge\refoutsc+\Oof{\exp(-\scdecay\rad)} \]
and still get more or less the same result, see Remark~\ref{AlternativeAssumptions}.

A very important example of an asymptotically hyperbolic manifold is the spatial Anti-de Sitter-Schwarzschild solution $\AdSoutg$ satisfying
\[ \AdSoutg = \houtg + \frac{\frac23\mass*}{\sinh(\rad)} + \Oof_\infty(\exp({-}5\rad)) = \d r^2 + (\sinh(\rad)^2 + \frac{\frac23\mass*}{\sinh(\rad)})\sphg + \Oof_\infty(\exp({-}5\rad)). \]
Furthermore, this metric has the same constant scalar curvature as the hyperbolic space, $\AdSoutsc={-}6$. Here, $\mass*\neq0$ denotes the \emph{mass parameter} of the manifold.\footnote{Note that we choose the mass parameter to be compatible with the definition of mass by Chru{\'s}ciel-Herzlich \cite{chrusciel2003mass,herzlich2015computing} differing from the definition used by Neves-Tian by a factor $2$.} Note that the first term which differs from the hyperbolic space is the \emph{mass term} $\frac{\frac13\mass*}{\sinh(\rad)}\sphg$ which is of order ${-}3$ (its components decay with $\sinh(\rad)^{-1}$ as $\vert\sphg\vert_{\houtg}=\sinh(\rad)^{{-}2}$). In particular, we can calculate the first non-constant term (again of order ${-}3$) of the Ricci curvature 
\begin{equation*}\labeleq{Ass_Ricci}
 \AdSoutric={-}2+\frac{\mass*}{\sinh(\rad)}\sphg*-\frac{2\mass*}{\sinh(\rad)^3}\dr^2+\Oof_\infty(\exp({-}5\rad))
\end{equation*}
without knowing the \lq error-term\rq\ $\Oof_\infty(\exp({-}5\rad))$.

\subsection{Known results}\label{ResultsSoFar}
In \cite{NevesTianExistenceCMC_I}, Neves-Tian generalized the results by Rigger, \cite{rigger2004foliation}, by proving the existence and uniqueness of a CMC-foliation for a more general class of asymptotically hyperbolic manifolds than Rigger used: Neves-Tian assumed $\Ck^2_5$-asymptotic to the Anti-de Sitter-Schwarzschild spatial solution, meaning they assumed the decay rates $\scdecay=\decay=5$ with the spatial anti-de Sitter metric $\refoutg=\AdSoutg$ as reference metric, \ie
\begin{equation*}\labeleq{Ass_NevesTian_I}
	\outg = \AdSoutg + \Oof_2(\exp({-}5\rad))
		= \houtg* + \frac{2\,\mass*}{3\sinh(\rad)}\,\sphg + \Oof_2(\exp({-}5\rad)).
\end{equation*}
In particular, the Ricci curvature in their setting is still of the form \eqref{Ass_Ricci} for $\Oof_0$ instead of $\Oof_\infty$.

In a second paper the following year, Neves-Tian further generalized this result to the decay rate $\decay=4$ and the more general reference metric\footnote{Note that this means for each asymptotically Anti-de Sitter metric, \ie of the form \eqref{Ass_NevesTian_I}, we have $\massaspect=\frac{\mass*}2\id_{\sphere}$ and \emph{not} $\massaspect=\mass*\id_{\sphere}$ as one could also think. These explains the additional factor two.} $\refoutg=\houtg+\frac{4\massaspect}{3\sinh(\rad)^3}$, \cite{NevesTianExistenceCMC_II}---note that the increased the decay order to $k=3$---, \ie they assumed
\begin{equation*}\labeleq{Ass_NevesTian_II}
 \outg = \houtg* + \frac{4\,\massaspect}{3\sinh(\rad)}\,\sphg + \Oof_3(\exp({-}4\rad))
\end{equation*}
and therefore the Ricci curvature is of the form \eqref{Ass_Ricci} when we replace $\mass*$ by $\trmassaspect*$ and $\Oof_\infty$ by $\Oof_1$. This reference metric $\refoutg*$ is characterized by the so called \emph{mass aspect tensor} $\massaspect\in\Omega^{(0,2)}(\sphere)$ which got his name as its characterizes the \emph{mass vector} $\mass=(\mass^\ui)_\ui\in\R^{3,1}$ via
\[ \mass^0 = \frac1{16\pi}\int_{\sphere}\trmassaspect\d\sphmug, \qquad \mass^i = \frac1{16\pi}\int_{\sphere} x^i\,\trmassaspect\d\sphmug, \]
where $x^i$ denote the Euclidean coordinates induces on $\sphere=\sphere_1(0)$.
In these works, Neves-Tian had to additionally assume that the mass $\mass*$ is positive and the mass aspect function $\sphtr\,\massaspect$ is positive, respectively. Both assumptions ensure that the mass vector $\mass=(\mass^\ui)_\ui$ is future-pointing and timelike. The latter seems to be a necessary assumption for the existence of a unique CMC-foliation, see \cite{cederbaum2015center}. However, the assumptions on the mass and mass aspect function are stronger than the pure assumption of a future-pointing and timelike mass vector.

In contrast to Neves-Tian, Chodosh in \cite{chodosh2014large} assumed only \emph{control of the metric up to the mass aspect order} $\decay=3$ and $\outsc\ge{-}6$, \ie
\[ \outg=\houtg+\Oof_2(\exp({-}3\rad)),\qquad
		0\le\exp(\rad)(\outsc+6)\in\Lp^1(\outM), \]
and proved existence of \emph{isoperimetric regions} $\Omega$ of volume $V$---in particular, $\partial*\Omega$ is a stable CMC-surface---for every sufficiently large volume $V$. Furthermore, he proved uniqueness of these isoperimetric regions if $(\outM,\outg*)$ is \emph{identical} to the Anti de-Sitter space outside of some compact set.\smallskip

In this paper, we generalize the existence and uniqueness of the CMC-foliation to the decay rates $\decay>\frac52$ and $\scdecay>3$. In doing so, we reach a decay rate \emph{below} the critical order of the mass aspect tensor. We note that in contrast to the Euclidean setting, we still have to assume higher decay rates than $\decay>\frac32$ and $\scdecay>3$ which are the ones necessary to ensure that the mass vector is well-defined. As we do not assume $\outsc\ge\houtsc\equiv{-}6$, the CMC-leaves will not necessarily correspond to solutions of the isoperimetric problem---as Chodosh proved under his stronger assumptions.

\begin{remark}[Comparing the uniqueness result with the one by Neves-Tian]\label{Compare_uniqueness}
In \cite{NevesTianExistenceCMC_I} and \cite{NevesTianExistenceCMC_II}, Neves-Tian proved uniqueness in the class of stable constant mean curvature surfaces $\M$ satisfying
\[ \max_{\M}\rad \le \frac65\min_{\M}\rad + \text{constant}\quad\text{and}\quad
		\max_{\M}\rad \le \min_{\M}\rad + \text{constant}, \]
respectively. Here, there assumption were asymptotics to the spatial Anti-de Sitter solution, \ie \eqref{Ass_NevesTian_I}, and a well-defined mass aspect tensor, \ie \eqref{Ass_NevesTian_II}, respectively. Our uniqueness result holds within the class of constant mean curvature surfaces $\M$ with controlled instability\footnote{This is a slightly weaker assumption than stability. We have to introduce this weaker version to include the case of a past-pointing mass vector.} and
\[ \max_{\M}\rad  \le \zeta'\min_{\M}\rad, \qquad\zeta'<2-\frac5{2\decay} \xrightarrow{\decay\to3} \frac76, \]
\ie it generalizes Neves-Tian's first uniqueness result from decay rate $\decay=\scdecay=5$ to decay rates $\decay\in\interval{\frac52}*3$, $\scdecay>3$. Note that even for our extreme case $\decay=3$, $\scdecay\ge\frac72$, we do not achieve the factor $\zeta'=\frac65$ which Neves-Tian used (for $\decay\ge5$) but only every $\zeta'<\frac76$. However we expect that if we apply our methods to one of their setting, then the uniqueness should hold for a even larger radius factor $\zeta'>\frac65$.
\end{remark}

\par\medskip\par\egroup\textbf{Acknowledgment.}
The author wishes to thank Carla~Cederbaum for discussions on optimal decay rates and for sharing her knowledge on the hyperbolic center of surfaces and asymptotically hyperbolic spaces. Furthermore, the proof of the very important Theorem~\ref{W2pSurfaceRegularity} would not be possible without the inspiring conversations with several other people and therefore the author owes thanks to Sebastian Heller and Gerhard Huisken for fruitful discussions on regularity of graphs, to Mattias Dahl for very helpful comments on isometries of the hyperbolic space, and to Katharina Radermacher for observing a necessary interpretations of an integral related to the coordinate center of regions in the hyperbolic space. Finally, the author thanks the \emph{Alexander von Humboldt Foundation} for ongoing financial support via the \emph{Feodor Lynen scholarship}.

\section*{Structure of the paper}
In Section~\ref{Assumptions_and_notation}, we explain the notation and definitions used in this article. Although it is quite technical, Section~\ref{Regularity_of_the_hypersurfaces} contains the most important step in the proof of all of the main theorems, namely the regularity theory of CMC-surfaces in asymptotically hyperbolic spaces. In particular, we prove important inequalities on the roundness of these surfaces. The stability of these surfaces is then proven in Section~\ref{Section-Stability}. The existence and uniqueness theorems are proven in Section~\ref{proofmaintheorem}---these proofs are based on the same continuity argument as the existence proofs by Metzger and Neves-Tian in \cite{metzger2007foliations,NevesTianExistenceCMC_II}, but extend this proof structure to a \lq smooth\rq\ argument as it was previously done by the author in \cite{nerz2015CMCfoliation}. Now, the regularity theorem is a corollary of the results in Section~\ref{Regularity_of_the_hypersurfaces}. In the short Section~\ref{Section_Stability}, we show the stability of the CMC-foliation, \ie the stability of the CMC-leaves under pertubations of the metric. Finally, we state and prove the regularity theorem for almost umbilic surfaces in $\hyperbolicspace$ in Appendix~\ref{Section-W2pSurfaceRegularity}.

\section{Assumptions and notation}\label{Assumptions_and_notation}
\begin{notation}[Notations for the most important tensors]
In order to study foliations (near infinity) of three-dimensional Riemannian manifolds by two-dimensional spheres, we have to deal with different manifolds (of different or the same dimension) and different metrics on these manifolds, simultaneously. To distinguish between them, all three-dimensional quantities like the surrounding manifold $(\outM,\outg*)$, its Ricci and scalar curvature $\outric$ and $\outsc$ and all other derived quantities carry a bar, while all two-dimensional quantities like the CMC leaf $(\M,\g*)$, its second fundamental form $\zFund*$, the trace-free part of its second fundamental form $\zFundtrf:=\zFund*-\frac12\,(\tr\zFund*)\g*$, its Ricci, scalar, and mean curvature $\ric$, $\sc$, and $\H:=\tr\zFund$, its outer unit normal $\nu$, and all other derived quantities do not.\end{notation}

Here, we interpret the second fundamental form and the normal vector of a hypersurface as quantities of the surface (and thus as two-dimensional). For example, if $\M<\Hradius>$ is a hypersurface in $\outM$, then $\nu<\Hradius>$ denotes its normal (and \emph{not} $\IndexSymbol{\outsymbol\nu}<\Hradius>$). The same is true for the \lq lapse function\rq\ and the \lq shift vector\rq\ of a hypersurfaces arising as a leaf of a given deformation or foliation. Furthermore, we stress that the sign convention used for the second fundamental form, \ie $\zFund(X,Y)=\outg(\outlevi*_{\!X}Y,\nu)$ for $X,Y\in\X(\M)$, results in the \emph{negative} mean curvature $\boldsymbol{\eukH(\sphere_\rradius)\equiv{-}\frac2\rradius}$ for the Euclidean sphere of radius $\rradius$.

\begin{notation}[Left Indexes and accents of tensors]
If different two-dimensional manifolds or metrics are involved, then the lower left index denotes the mean curvature index $\Hradius$ of the current leaf $\M<\Hradius>$, \ie the leaf with mean curvature $\H<\Hradius>\equiv{-}\frac{\cosh(\Hradius)}{\sinh(\Hradius)}$, or the radius $\rradius$ of a coordinate sphere $\sphere_\rradius(0)$. Furthermore, quantities carry the upper left index $\hyperbolich$, $\euclideane$, and $\sphg*$ if they are calculated with respect to the hyperbolic metric $\houtg*$, the Euclidean metric $\eukoutg$, and the standard metric $\sphg<\Hradius>$ of the Euclidean sphere $\sphere_\Hradius(0)$, correspondingly. We abuse notation and suppress the left indexes, whenever it is clear from the context which manifold and metric we refer to.
\end{notation}

\begin{notation}[Indexes]
We use upper case latin indices $\ii$, $\ij$, $\ik$, and $\il$ for the two-dimensional range $\lbrace2,3\rbrace$, lower case latin indices $\oi$ and $\oj$ for the three-dimensional range $\lbrace 1,2,3\rbrace$, and the greek index $\ui$ for the four-dimensional range $\{0,1,2,3\}$. The Einstein summation convention is used accordingly.\pagebreak[3]\smallskip
\end{notation}

As there are different definitions of \lq asymptotically hyperbolic\rq\ in the literature, we now give the one used in this paper. 
\begin{definition}[\texorpdfstring{$\Ck^2_{\decay,\scdecay}$}{C2}-asymp\-to\-tic\-ally hyperbolic Riemannian manifolds]\label{Ck_asymptotically_hb}
Let $\decay,\scdecay>0$ be constants. A triple $(\outM,\outg,\outx)$ is called (three-dimensional) \emph{$\Ck^2_{\decay,\scdecay}$-asymp\-to\-tic\-ally hyperbolic} Riemannian manifold if $(\outM,\outg)$ is a three-dimensional smooth Riemannian manifold and $\outx:\outM\setminus\overline L\to\R^3$ is a smooth chart of $\outM$ outside a compact set $\overline L\subseteq\outM$ such that there exists a constant $\oc\ge0$ with
\begin{equation*} 
 \vert\outg-\houtg\vert_{\houtg*} \hspace{-.05em}+ \vert\houtlevi*(\outg*-\houtg*)\vert_{\houtg*} \hspace{-.05em}+ \vert\outric-\houtric \vert_{\houtg*} \le \oc\,\exp(-\decay\,\rad), \quad
 \vert \outsc-\houtsc\vert \le \oc\,\exp({-}\scdecay\,\rad),\hspace{-.1em}
 \labeleq{Decay_assumptions_g}\end{equation*}
where $\houtg*=\d r^2+\sinh(\rad)^2\sphg*$ and $\sphg*$ denote the hyperbolic metric and the standard metric of the Euclidean unit sphere $\sphere$, respectively. Here, these quantities are identified with their push-forward along $\outx$. Finally, $(\outM,\outg*,\outx)$ is called $\Ck^2_{\decay}$-asymptotically hyperbolic if it is $\Ck^2_{\decay,\decay}$-asymptotically hyperbolic.
\end{definition}
We often abuse notation and suppress the chart $\outx$.
\begin{remark}[Alternative assumptions]\label{AlternativeAssumptions}
We can weaken the assumption on $\outsc$ by only assuming $\outsc\ge\houtsc+\outc\,\exp({-}\scdecay\rad)$, $(\outsc-\houtsc)\in\Lp^1(\sinh(\rad)\,\houtmug)$ and still get the same results except the following: In Theorem~\ref{Regularity_Theorem_short} (and \ref{Regularity_Theorem_full}), we do not achieve $\ewjac_i=\frac{6\mass*}{\sinh(\Hradius)^3}+\Oof(\exp({-}(3+\outve)\Hradius))$ for the three eigenvalues $\ewjac_i$ of the (negative) stability operator with smallest absolute value, but only $\ewjac_i\ge\frac{6\mass*}{\sinh(\Hradius)^3}-\Oof(\exp({-}(3+\outve)\Hradius))$.

Furthermore, we can weaken the above assumptions by only assuming
\[ \outg=\houtg+\oof_2(\exp({-}\frac52\rad)), \quad
	\outsc=\houtsc+\oof(\exp({-}3\rad)), \quad
	\exp(\rad)(\outsc-\houtsc)\in\Lp^1(\outM\setminus\overline L), \]
\ie by replacing $\exp({-}\decay\rad)$ and $\exp({-}\scdecay\rad)$ by $f(\rad)$ and $g(\rad)$, where $f$ and $g$ are smooth functions on $\interval0\infty$ with
\[ \lim_{r\to\infty}\exp(\frac52r)f(r)=0, \qquad
	 \lim_{r\to\infty}\exp(3r)g(r)=0, \qquad
	 \int_0^\infty \exp(3r)g(r)\d r<\infty. \]
However, we then have to replace every $\exp({-}\outve\rad)$ in the claims by $h(\Hradius)$ and can only allow the radius factor $\zeta'=1$, where $h=\oof(\Hradius^0)$ is a smooth function on $\interval0\infty$ with $h(r)\to0$ as $r\to\infty$.

Finally, we can replace the above pointwise assumptions by Sobolev assumptions (up to the third derivative of $(\outg-\houtg)$ and the first of $(\outsc-\houtsc)$) as they were introduced by Bartnik in the asymptotically Euclidean setting, \cite{bartnik1986mass}.\smallskip\pagebreak[2]
\end{remark}

We use the definition of mass of an asymptotically hyperbolic manifold as it was given by Chru{\'s}ciel-Herzlich and Michel \cite{chrusciel2003mass,michel2011geometric}.
\begin{definition}[Mass of an asymptotically hyperbolic manifold {\cite{chrusciel2003mass,michel2011geometric}}]\label{Definition_Mass}
For a $\Ck^2_{\decay,\scdecay}$-asymp\-to\-tic\-ally hyperbolic manifold $(\outM,\outg,\outx)$ with $\decay\in\interval{\frac32}*3$ and $\scdecay>3$ the \emph{mass vector} $\mass\in\R^{1,3}$ is defined by $\mass:=\lim\limits_{\rradius\to\infty}\mass(\sphere_\rradius)$, \vspace*{-.5em}where $\mass(\sphere_\rradius(0)):=(\mass_\ui(\sphere_\rradius(0)))_\ui$ and
\[ \mass_\ui(\M) := \frac1{\omega_n}\,\int_{\M} (\momentum(\houtlevi*V_\ui,\cdot)-V_\ui\,\houtdiv\momentum)(\hnu) \d\hmug \qquad\forall\,\ui\in\{0,\dots,3\}, \]
where $\hnu$ is the outer unit-normal of $\M\hookrightarrow\outM$, $\momentum:=\houttr\outfg\,\houtg-\outfg$, $\outfg:=\outg-\houtg$, $V_0:=\cosh(\rad)$, and $V_\oi:=\sinh(\rad)\frac{\outx_\oi}\rad$ for $\oi\in\{1,\dots,3\}$. The mass vector of $(\outM,\outg,\outx)$ is called \emph{timelike} if
\[ \vert\mass\vert_{\R^{3,1}}^2 := {-}\mass_0
^2 + \vert(\mass_\oi)_\oi\vert_{\R^3}^2 := {-}\mass_0^2 + \sum_{\oi=1}^3\mass_\oi^2 < 0 \]
and it is called \emph{future-pointing} if $\mass_0>0$ and \emph{past-pointing} if $\mass_0\relax<0$. Finally, $\mass*:={-}\vert\mass\vert_{\R^{3,1}}>0$ and $\mass*:=\vert\mass\vert_{\R^{3,1}}<0$ denote the \emph{total mass} if the mass vector is timelike future-pointing and timelike past-pointing, respectively.
\end{definition}
\begin{remark}[Some remarks on this definition]\label{remarks_on_mass}
We recall some basic facts on this mass for further information on it, we refer to \cite{chrusciel2003mass,michel2011geometric,dahl2015density} and the citations therein.\smallskip
\begin{enumerate}[nosep,label=(\roman{*})]
\item $\mass$ is well-defined if $\decay>\frac32$ and $\scdecay>3$ and we can replace $\sphere_\rradius(0)$ by any surface $\M<\rradius>$ sufficiently close to $\sphere_\rradius(0)$, see Chru{\'s}ciel-Herzlich's and Michel's results, \cite{chrusciel2003mass}, \cite[Sect.~IV.A.2]{michel2011geometric}.
\item Chru{\'s}ciel-Herzlich have proven that $\mass$ behaves under change of coordinates as to be expected, \cite{chrusciel2003mass}. In particular, there exists an isometry $\hiso$ of the hyperbolic space $\hyperbolicspace$ such that
\begin{itemizeequation}
 \mass(\hiso\circ\outx) = (\mass*,0,\dots,0), \labeleq{Balanced_condition}
\end{itemizeequation}
where $\mass(\hiso\circ\outx)$ denotes the mass with respect to the asymptotically hyperbolic coordinates $\hiso\circ\outx$ of $\outM$ and with respect to this coordinates~\eqref{Decay_assumptions_g} is still satisfied---for some constant $\oc'=\Cof{\oc'}[\oc][\decay][\mass]$ instead of $\oc$. Note that this isometry is (up to a rotation) uniquely determined by this property. Furthermore, this proves that $\mass*$ is a geometric object, \ie independent of the chosen coordinates satisfying~\eqref{Decay_assumptions_g} for some $\decay>\frac32$ and $\scdecay>3$.
\item Sakovich-Dahl proved that the mass depends \emph{continuously} on the metric, \ie $\mass[1]$ is as close to $\mass[2]$ as the $\Ck^2_{\decay,\scdecay}$-asymptotically hyperbolic metrics $\outg[1]$ and $\outg[2]$ are close to each other (in a sufficiently strong sense). In particular, if $\outg[n]*\to\outg*$ (in a sufficiently strong sense) then $\mass(\outg[n]*)\to\mass(\outg*)$, \cite{dahl2015density}. Note that Theorem~\ref{Stability_Theorem_full} and Corollary~\ref{Stability_Corollary_full} imply the analogous result for the CMC-foliation.
\item Neves-Tian and Cederbaum-Cortier-Sakovich called coordinates $\outx$ satisfying~\eqref{Balanced_condition} \emph{balanced}, \cite{NevesTianExistenceCMC_II,cederbaum2015center}. Furthermore, Cederbaum-Cortier-Sako\-vich defined a \emph{center of mass} vanashing for balanced coordinates and proved (under their stricter assumptions) that it corresponds to the \emph{hyperbolic center} of the CMC-foliation which they also defined, \cite{cederbaum2015center} and see Definition~\ref{center}. As side result, we prove the same result using our weaker assumptions.
\item\label{Ricci_version_of_mass}%
 Recently, Herzlich proved that this mass vector is equivalent to (what he called) the \emph{Ricci version of the mass vector} being
\begin{itemizeequation}\labeleq{Ricci_Version_Mass}
 8\pi\,\mass^\ui = {-}\lim_{\rradius\to\infty}\int_{\sphere_\rradius(0)}\outsymbol G(\outsymbol X^{(\ui)},\nu<\rradius>) \d\mug,\qquad
	\outsymbol G := \outric-(\frac12\outsc+1)\outg,
\end{itemizeequation}
see \cite{herzlich2015computing}. Here, the four conformal vector fields $X^{(\ui)}$ are defined using the Poincar\'e ball model of the hyperbolic space, \ie a chart $\outy:\outM\setminus\overline K\to B_1(0)$, via 
\begin{itemizeequation}
	X^{(0)}:=\outy^\oi\,\partial[\outy]_\oi,\qquad
	 X^{(\oi)}:=\rad@\outy^2\,e_\oi-2\,\outy^\oi\,\outy^\oj\partial[\outy]_\oj. \end{itemizeequation}
Note that this means 
\begin{itemizeequation}
	X^{(0)}=\sinh(\rad)\frac{\outx^\oi}\rad\partial[\outx]_\oi, \quad
	 X^{(\oi)}=\partial[\outx]_\oi-\sinh(\rad)\frac{\outx^\oi\,\outx^\oj}{\rad^2}\,\partial[\outx]_\oj+\Oof_\infty(\exp({-}\rad)).
\end{itemizeequation}
Furthermore, we can again replace $\sphere_\rradius(0)$ by any surface $\M<\rradius>$ sufficiently close to $\sphere_\rradius(0)$.
	\end{enumerate}
\end{remark}
\begin{definition}[Hyperbolic hawking mass{, \eg \cite{wang2001mass}}]
If $(\M,\g)$ is a closed hypersurface in a three-dimensional Riemannian manifold, then
\[ \HmHaw := \HmHaw(\M) := (\frac{\volume{\M}}{16\pi})^{\frac12}(1-\frac1{8\pi}\int_{\M}\frac{\H^2-4}2\d\mug) \]
is called (\emph{hyperbolic}) \emph{Hawking mass} of $\M$.
\end{definition}
\begin{remark}[Hawking and total mass]
A direct calculation proves
\[ \vert\HmHaw(\M) - \mass^0\vert \le C\,(\exp({-}\outve\rradius)+\textrm d_*(\M,\sphere_\rradius(0))) \]
for a suitable chosen definition of the \lq metric\rq\ $\textrm d_*$, where $\rradius:=\min_{\M}\rad$. This means the hyperbolic Hawking mass approaches the $0^{\text{th}}$-component for surfaces sufficiently close to the coordinate spheres around the coordinate origin. Remark~\ref{remarks_on_mass}.\ref{Ricci_version_of_mass}, the Gau\ss-Bonnet theorem, the Gau\ss\ equation, and Theorem~\ref{Regularity_Theorem_short} proves $\HmHaw(\M<\Hradius>)\to\mass*$ for $\Hradius\to\infty$, where $\M<\Hradius>$ denote the CMC-leaf of mean curvature $\H<\Hradius>\equiv{-}2\:\frac{\cosh(\Hradius)}{\sinh(\Hradius)}$ constructed in Theorem~\ref{Existence_Theorem}.
\end{remark}

\begin{definition}[{Hyperbolic center of a hypersurface and hyperbolic (coordinate) center of mass, \cite[eq.(9), Def.~3.10]{cederbaum2015center}}]\label{center}
Let $\M\hookrightarrow\hyperbolicspace$ be a closed hypersurface in the hyperbolic space. The point $\centerz(\M)\in\hyperbolicspace$ with
\[ I(\centerz(\M)) = \centerz\relax_{\R^{3,1}}(\M) = \frac{\centerz'(\M)}{\sqrt{{-}\vert\centerz'(\M)\vert_{\R^{3,1}}^2}}, \qquad\qquad
		{\centerz_{}'}^\ui := \int_{\M} I^\ui \d\hmug \]
is called (\emph{coordinated}) \emph{hyperbolic center of $\M$}, where
\[ I=(I^\ui)_\ui:\hyperbolicspace\to\R^{3,1}:\outx\mapsto(\cosh(\rad),\,\frac{\outx^1\sinh(\rad)}\rad\,,\,\frac{\outx^2\sinh(\rad)}\rad\,,\,\frac{\outx^3\sinh(\rad)}\rad) \]
is the natural embedding of $\hyperbolicspace$ to the hyperboloid in the Minkowski spacetime.

If $\outx$ is a $\Ck^2_{\decay,\scdecay}$-asymptotically hyperbolic chart of a Riemannian manifold $(\outM,\outg*)$, then $\outcenterz\in\hyperbolicspace$ with
\[ I(\outcenterz) := \frac{\mass}{\sqrt{{-}\vert\mass\vert^2_{\R^{3,1}}}} \]
is called \emph{hyperbolic} (\emph{coordinate}) \emph{center of mass} (of $\outM$ with respect to $\outx$).
\end{definition}
\begin{remark}[The center of mass and the CMC-foliation]
In \cite{cederbaum2015center}, proved that the CMC-foliation $\{\M<\Hradius>\}_\Hradius$ constructed by Neves-Tian in \cite{NevesTianExistenceCMC_II} characterizes the center of mass in the following way: the (coordinate) hyperbolic center $\centerz(\M<\Hradius>)$ of the leaves of the foliation converge (as $\Hradius\to\infty$) to the hyperbolic (coordinate) center of mass of $(\outM,\outg*)$. Their proof for this result can also be applied to the CMC-foliation constructed here, see Theorem~\ref{Existence_Theorem_full}.
\end{remark}

As mentioned, we frequently use foliations. In the following, we characterize them infinitesimally by their lapse functions and their shift vectors.
\begin{definition}[Lapse functions, shift vectors]
Let $\theta>0$ and $\Hradius_0\in\R$ be constants, $I\supseteq\interval{\Hradius_0-\theta\Hradius}{\Hradius_0+\theta\Hradius}$ be an interval, and $(\outM,\outg*)$ be a Riemannian manifold. A smooth map $\Phi:I\times\M\to\outM$ is called \emph{deformation} of the closed hypersurface $\M=\M<\Hradius_0>=\Phi(\Hradius_0,\M)\subseteq\outM$ if $\Phi<\Hradius>(\cdot):=\Phi(\Hradius,\,{\cdot}\,)$ is a diffeomorphism onto its image $\M<\Hradius>:=\Phi<\Hradius>(\M)$ and $\Phi<\Hradius_0>\equiv\id_{\M}$. The decomposition of $\spartial*_\Hradius\Phi$ into its normal and tangential parts can be written as
\[ \partial[\Hradius]@\Phi = {\rnu<\Hradius>}\,{\nu<\Hradius>} + {\rbeta<\Hradius>}, \]
where $\nu<\Hradius>$ is the outer unit normal to $\M<\Hradius>$. The function $\rnu<\Hradius>:\M<\Hradius>\to\R$ is called \emph{lapse function} and the vector field $\rbeta<\Hradius>\in\X(\M<\Hradius>)$ is called \emph{shift} of $\Phi$. If $\Phi$ is a diffeomorphism, then it is called a \emph{foliation}.\pagebreak[3]
\end{definition}

For notation convenience, we use the following abbreviated form for the contraction of two tensor fields.
\begin{definition}[Tensor contraction]
Let $(\M,\g*)$ be a Riemannian manifold. The \emph{traced tensor product} of a $(0,k)$ tensor field $S$ and a $(0,l)$ tensor field $T$ on $(\M,\g*)$ with $k,l>0$ is defined by
\[ (\trzd ST)_{I_1\dots I_{k-1}\!J_1\dots J_{l-1}} := S_{I_1\dots I_{k-1} K}\,\g^{KL}\,T_{LJ_1\dots J_{l-1}}. \]
This definition is independent of the chosen frame.
\end{definition}
Finally, we specify the definitions of Lebesgue and Sobolev norms on compact Riemannian manifolds which we use throughout this article.
\begin{definition}[Lesbesgue and Sobolev norms]
If $(\M,\g*)$ is a two-dimensional, compact Riemannian manifold without boundary, then the \emph{Lebesgue norms} are defined by
\[ \Vert T\Vert_{\Lp^p(\M)} := (\int_{\M} \vert T\vert_{\g*}^p \d\mug)^{\frac1p}\quad\forall\,p\in\interval*1\infty, \qquad \Vert T\Vert_{\Lp^\infty(\M)} := \mathop{\text{ess\,sup}}\limits_{\M}\,\vert T\vert_{\g*}, \]
where $T$ is an arbitrary measurable function (or tensor field) on $\M$. Correspondingly, $\Lp^p(\M)$ is defined to be the set of all measurable functions (or tensor fields) on $\M$ for which the $\Lp^p$-norm is finite. The \emph{Sobolev norms} are defined by
\[ \Vert T\Vert_{\Wkp^{k+1,p}(\M)} := \Vert T\Vert_{\Lp^p(\M)} + \volume{\M}^{\frac12} \Vert\levi*T\Vert_{\Wkp^{k,p}(\M)}, \qquad
		\Vert T\Vert_{\Wkp^{0,p}(\M)} := \Vert T\Vert_{\Lp^p(\M)}, \]
where $k\in\N_{\ge0}$, $p\in\interval*1*\infty$ and $T$ is any measurable function (or tensor field) on $\M$ for which the $k$-th (weak) derivative exists. Correspondingly, $\Wkp^{k,p}(\M)$ is the set of all such functions (or tensors fields) for which the $\Wkp^{k,p}(\M)$-norm is finite. Furthermore, $\Hk^k(\M)$ denotes $\Wkp^{k,2}(\M)$ for any $k\ge1$ and $\Hk(\M):=\Hk^1(\M)$.\pagebreak[3]
\end{definition}
\begin{remark}\label{Scaling_Of_Sobolev}
We directly see that natural Sobolev- and Poincar\'e inequalities of such a manifold $\M$ are of the form
\[ \Vert T\Vert_{\Lp^2(\M)} \le \cSob\Vert T\Vert_{\Wkp^{1,1}(\M)}, \qquad
		\Vert f-\mean f\Vert_{\Lp^2(\M)} \le \c_{\text{Poincar\'e}}\Vert\levi*f\Vert_{\Lp^1(\M)}. \]
Let us explain a basic scaling property: If $(\M<n>,\g<n>):=(\M,n^2\,\g)$ is a sequence of rescaled versions of $\M$ and $f\in\Ck^\infty(\M)$ is a function, then
\begin{align*}
 \Vert f\Vert_{\Lp^p(\M<n>)} ={}& (\frac{\volume{\M<n>}}{\volume{\M}})^{\frac1p}\Vert f\Vert_{\Lp^p(\M)}, &
 \Vert\levi<n>*f\Vert_{\Lp^p(\M<n>)} ={}& (\frac{\volume{\M<n>}}{\volume{\M}})^{\frac1p-\frac12}\Vert\levi*f\Vert_{\Lp^p(\M)} \\
 \Vert f\Vert_{\Wkp^{k,p}(\M<n>)} ={}& (\frac{\volume{\M<n>}}{\volume{\M}})^{\frac1p}\Vert f\Vert_{\Wkp^{k,p}(\M)},
\end{align*}
\ie these norms behave nicely under scaling. In particular, the Sobolev- and Poincar\'e inequalities are only optimal inequalities on a sequence of functions ${}_nf\in\Ck^\infty(\M<n>)$ if $\Vert {}_nf\Vert_{\Lp^p(\M<n>)}$ and $\volume{\M}^{\frac12}\Vert\hlevi<n>{}_nf\Vert_{\Lp^p(\M<n>)}$ have the same scaling behavior (as $n\to\infty$). This will lead to some obstacle in the rest of the article.
\end{remark}

\begin{notation}[Hyperbolic radii]
If $\M$ is a two-dimensional manifold, then the real number $\Aradius$ with $\volume{\M}=4\pi\,\sinh(\Aradius)^2$ is called (\emph{hyperbolic}) \emph{area radius}.

If $\M$ is a closed hypersurface of a three-dimensional Riemannian manifold and $\M$ has constant mean curvature, then the real number $\Hradius$ with $\H\equiv{-}2\:\frac{\cosh(\Hradius)}{\sinh(\Hradius)}$ is called (\emph{hyperbolic}) \emph{mean curvature radius}.

If $\M\hookrightarrow\R^3$ is a closed hypersurface in the Euclidean (or hyperbolic) space, then the real number $\rradius:=\min_{\M}\rad$ and $\Rradius:=\max_{\M}\rad$ are called \emph{minimal} and \emph{maximal coordinate radius}, respectively.
\end{notation}
\begin{remark}
In the largest part of this article, we assume that these radii are compatible, \ie $\Hradius\le\Aradius+C\le\Rradius+C'\le\Hradius+C''$ and $\rradius\le\Rradius\le(1+\zeta)\rradius$ for some constant $C$, $C'$, $C''$ being small (compared to the radii). For the surfaces studied here, this compatibility is proven as first step in the proof of Theorem~\ref{Regularity_Theorem}, where we use an argument developed by Neves-Tian in \cite{NevesTianExistenceCMC_I,NevesTianExistenceCMC_II}.
\end{remark}

Finally, we need a kind of stability assumption on the CMC-surfaces. More precisely, we do not need \emph{stability} of the CMC-surfaces, but only a sufficient control of the \emph{instability} of the surface, as the author already used in \cite{nerz2015GeometricCharac} (a suggestion by Carla Cederbaum). We note that this is not only a weaker assumption, but the only one we can assume if the mass vector is past-pointing.
\begin{definition}[$\alpha$-controlled instability]\label{ControlledInstability}
Let $\alpha\in\R$ and $\c>0$ be constants. A closed hypersurface $(\M,\g*)$ with constant mean curvature in a three-dimen\-sio\-nal Riemannian manifold $(\outM,\outg*)$ has \emph{$\alpha$-controlled instability} if the eigenvalues of the stability operator are bounded from bellow by $\alpha$, \ie
\[ \int\trtr{\levi*f}{\levi*f}\d\mug \ge \int(\trtr\zFund\zFund+\outric*(\nu,\nu)+\alpha)\,(f-\mean f)^2\d\mug \qquad\forall\,f\in\Ck^1(\M), \]
where $\mean f:=\fint f\d\mug := \volume{\M}^{-1}\,\int f\d\mug$.
\end{definition}
We see that a CMC-hypersurface is non-strictly stable if and only if it has $0$-controlled instability and it is strictly stable if and only if it has $\alpha$-controlled instability with $\alpha>0$. Furthermore, it is obvious that any CMC-hypersurface with ${-}\alpha$-controlled instability has always ${-}\beta$-controlled instability if $\beta>\alpha$.\smallskip

\section{Regularity of the hypersurfaces}\label{Regularity_of_the_hypersurfaces}
\NewDocumentCommand\regsphere{sO{\decay}D(){\Hradius}mO{p}}%
{\mathcal{R}^{\hspace{-.1em}#5}_{\hspace{.15em}#2}(#4)\IfBooleanTF{#1}\relax{_{#3}}}%
\NewDocumentCommand\strregsphere{sO{\decay}D<>{\scdecay}D(){\Hradius}mO{p}}%
{\mathcal{R}^{\hspace{-.1em}#6}_{\hspace{.15em}#2,#3}(#5)\IfBooleanTF{#1}\relax{_{#4}}}%
In this section, we prove the central regularity result for CMC-surfaces. In Remark~\ref{Expectation_on_k}, we explain why it is justified to expect that the regularity given here is optimal, \ie for every decay rate between $\decay\in\interval2*3$ the regularity cannot be strengthened without additional assumptions on $\outg$. Note that we prove the regularity theorems for CMC-surfaces in $\Ck^2_{\decay,\scdecay}$-asymptotically hyperbolic spaces up to decay rate $\decay>2$ (and not $\decay>\frac52$)---if there exist such a CMC-surface for this decay rate. Note furthermore, that it would be sufficient if $\H$ is $\Wkp^{1,p}$-close to the constant ${-}2\:\frac{\cosh(\Hradius)}{\sinh(\Hradius)}$ instead of being this constant, \ie if $\Vert\H+2\frac{\cosh(\Hradius)}{\sinh(\Hradius)}\Vert_{\Wkp^{1,p}(\M)}\le\c\exp((\frac2p-\decay)\Hradius)$, then we get the same results.\smallskip

Let us state the main result of this section.
\begin{theorem}[Regularity of a CMC-surface]\label{Regularity_Theorem}
Let $(\outM,\outg*)$ be a $\Ck^2_{\decay}$-asymptotically hyperbolic manifold with decay rate $\decay>2$. For all constants $\eta\in\interval0*4$, $\zeta\in\interval*0{1-\frac2{\decay}}$, $\c>0$, and $p\in\interval1\infty$, there exist constants $\rradius_0=\Cof{\rradius_0}[\decay][\oc][\zeta][\eta][\c]$ and $C=\Cof[\decay][\oc][\zeta][\eta][\c][p]$ with the following property:

If $\M$ is a closed hypersurfaces in $(\outM,\outg*)$ diffeomorphic to $\sphere$ with constant mean curvature $\H\equiv{-}2\,\frac{\cosh(\Hradius)}{\sinh(\Hradius)}$ and ${-}(4-\eta)\sinh(\Hradius)^{{-}2}$-controlled instability which satisfies
\begin{equation*}\labeleq{class-assumption}
 \Rradius \le (1+\zeta)\rradius, \qquad\rradius\ge\rradius_0, \qquad
	\eta>2\text{ or } \c^{-1}\exp(2\Hradius)\le\volume{\M}\le\c\,\exp(2\Hradius),
\end{equation*}
then
\[ \Vert\zFundtrf\Vert_{\Lp^p(\M)} + \Vert\levi\zFundtrf\Vert_{\Lp^p(\M)} \le C\,\exp((\frac2p-\decay')\Hradius), \qquad\decay'=\frac{\decay}{1+\zeta}>2 \]
In this setting, there exist an isometry $\hiso:\hyperbolicspace^3\to\hyperbolicspace^3$ of the hyperbolic space $\hyperbolicspace^3=(\R^3,\houtg)$ and a function $\graphf\in\Wkp^{2,p}(\sphere_\Hradius(0))$ such that
\[ \hiso(\M)=\hgraph\graphf, \qquad\qquad
	 \Vert\graphf\Vert_{\Wkp^{2,p}(\sphere_\Hradius(0))} \le C\,\exp((2+\frac2p-\decay')\Hradius). \]\smallskip\pagebreak[3]
\end{theorem}

Now, let us define the coordinate-invariant class of surfaces, we are working with.
\begin{definition}[Regular spheres]\label{Regular_sphere}
Let $(\outM,\outg*)$ be a three-dimensional Riemannian manifold. A hypersurface $\M\hookrightarrow(\outM,\outg*)$ is called \emph{regular sphere with radius $\Hradius$ and constants $\decay\in\interval0*3$, $p\in\interval2*\infty$, and $\c\ge0$}, in symbols $\M\in\regsphere{\c}$, if $\M$ is a topological sphere with constant mean curvature $\H\equiv{-}2\,\frac{\cosh\Hradius}{\sinh\Hradius}$ satisfying
\begin{equation*}
	\c^{-1}\,\exp(2\Hradius) \le \volume{\M} \le \c\,\exp(2\Hradius), \qquad
	\Vert \vert\outric*+2\hspace{.05em}\outg*\vert_{\outg*} \Vert_{\Lp^p(\M)} \le \c\,\exp((\frac2p-\decay)\,\Hradius),
\end{equation*}
where $\nicefrac2p:=0$ if $p=\infty$. Furthermore, $\M\in\regsphere{\c}$ is \emph{strictly regular}, in symbols $\M\in\strregsphere{\c}$, if additionally
\[ \Vert\outsc\Vert_{\Lp^1(\M)} \le \c\,\exp((2-\scdecay)\Hradius). \]
\end{definition}
\begin{remark}[An alternative assumption on the scalar curvature]\label{AlternativeAssumption_outsc_regsphere}
Note that we can actually replace the assumption $\Vert\outsc\Vert_{\Lp^1(\M)} \le \c\,\exp((2-\scdecay)\Hradius)$ by $\Vert(\outsc+6)^-\Vert_{\Lp^1(\M)} \le \c\,\exp((2-\scdecay)\Hradius)$, where $(\outsc+6)^-:=\min\{0,\outsc+6\}$ denotes the negative part of $\outsc+6$.
\end{remark}

Now, we prove that any regular sphere is in a $\Lp^4$-sense almost umbilic. The proof is almost identical to the one of \cite[Lemma~2.3]{nerz2015GeometricCharac} and familiar with the one of \cite[Prop.~4.3]{NevesTianExistenceCMC_I} (and its analog in \cite{NevesTianExistenceCMC_II}). However, as some technical steps are slightly different, we recall the proof for the readers convenience.
\begin{lemma}[\texorpdfstring{First $\Hk$- and $\Lp^4$}{H1 and Lp-4}-estimates for the second fundamental form]\label{Bootstrap_for_trace_free_second_fundamental_form__Lp4}
Let $(\M,\g*)\in\regsphere{\c}$ be a regular sphere with ${-}(4-\eta)\,\sinh(\sigma)^{{-}2}$-controlled instability within a three-dimensional Riemannian manifold $(\outM,\outg*)$, where $\decay\in\interval1*3$, $p\in\interval*2*\infty$, $\c\ge0$, and $\eta\in\interval0*4$ are constants. There are constants $\Hradius_0=\Cof{\Hradius_0}[\decay][\eta][p][\c]$ and $C=\Cof[\decay][\eta][p][\c]$ such that
\[ \exp(-2\,\Hradius)\,\Vert\zFundtrf\Vert_{\Hk(\M)}^2 + \Vert\zFundtrf\Vert_{\Lp^4(\M)}^4 \le C\,\exp(2(1-\decay)\Hradius) \nopagebreak \]
if $\Hradius>\Hradius_0$.\pagebreak[2]
\end{lemma}
\begin{proof}
As already done in \cite[Prop.~3.3]{metzger2007foliations}, \cite{NevesTianExistenceCMC_I,NevesTianExistenceCMC_II}, \cite[Prop.~2.1]{nerz2015CMCfoliation}, \cite[Lemma~2.3]{nerz2015GeometricCharac}, we first integrate $\trtr{\laplace\zFundtrf}{\zFundtrf}$ and then integrate it by parts. As in \cite[Lemma~2.3]{nerz2015GeometricCharac}, we do not use the Simon's identity for $\laplace\zFund$, but the following (equivalent but shorter) formula
\[ \levi*\div\zFund = \laplace\zFund - \div_1(\outrc_{\cdot\cdot\cdot\nu}) - \trzd\ric\zFund + \rc_{\cdot\ii\ij\cdot}\,\zFund^{\ii\ij}. \]
which can be proven equivalent to the Simon's identity using normal coordinates and the Codazzi equation.\footnote{Actually, this is true for every hypersurface $\M\hookrightarrow(\outM,\outg*)$ in any Riemannian manifolds $(\outM,\outg*)$.}
As we assumed $\dim\M=2$ (and $\dim\outM=3$), we know $\ric=\frac12\sc\g$, $\riem_{\ik\ii\ij\il}=\frac12\sc(\g_{\ii\ik}\g_{\ij\il}-\g_{\ii\il}\g_{\ij\ik})$. Therefore, a integration by parts proves
\begin{align*}
 \Vert\levi*\zFundtrf\Vert_{\Lp^2(\M)}^2
 ={}& \int\trzd{\div\zFund}{\div\zFundtrf}\d\mug - \int\outrc(\levi\zFundtrf,\nu)\d\mug \\
		&	- \frac12\int_{\M}\sc(\trtr\zFundtrf\zFundtrf + (\g_{\ii\ik}\g_{\ij\il}-\g_{\ii\il}\g_{\ij\ik})\,\zFund^{\ij\ik}\,\zFundtrf^{\ii\il})\d\mug \\
 ={}& 2\,\Vert\outric_\nu\Vert_{\Lp^2(\M)}^2 - \int\sc\trtr\zFundtrf\zFundtrf\d\mug,
\end{align*}
where we used $\H\equiv{-}2\,\frac{\cosh(\Hradius)}{\sinh(\Hradius)}$. Thus, the Gau\ss\ equation and the assumptions $\outric*$ imply
\[ \vert\Vert\levi*\zFundtrf\Vert_{\Lp^2(\M)}^2 + \int (\outsc - 2\,\outric*(\nu,\nu) + \frac12\H^2 - \trtr\zFundtrf\zFundtrf)\trtr\zFundtrf\zFundtrf\d\mug \vert \le C\,\exp(2(1-\decay)\,\Hradius). \]
Now, we use the assumptions on $\H$ and $\outric$ to get
\begin{equation*}\labeleq{Bootstrap_for_trace_free_second_fundamental_form__Lp4__1} \vert\Vert\levi*\zFundtrf\Vert_{\Lp^2(\M)}^2 + 2\,\sinh(\Hradius)^{-2}\,\Vert\zFundtrf\Vert_{\Lp^2(\M)}^2 - \Vert\zFundtrf\Vert_{\Lp^4(\M)}^4 \vert \le C_\delta\,\exp(2(1-\decay)\,\Hradius) + \delta\,\Vert\zFundtrf\Vert_{\Lp^4(\M)}^4, \end{equation*}
where $\delta>0$ is arbitrary and $C_\delta$ depends on this $\delta$. Note that we did not use the control of the instability so far.\pagebreak[2]

The ${-}\alpha:={-}(4-\eta)\,\sinh(\Hradius)^{{-}2}$-control of the instability of $\M$ for $f:=\vert\zFundtrf\vert_{\g*}$ means
\[ \Vert\levi*f\Vert_{\Lp^2(\M)}^2 \ge \int(\outric*(\nu,\nu)+\trtr\zFund\zFund-\alpha)(f-\mean f)^2 \d\mug. \]
Inserting the assumption on $\H$, we get
\begin{align*}
 \Vert\levi*f\Vert_{\Lp^2(\M)}^2
 \ge{}&	\int \trtr\zFundtrf\zFundtrf\,f^2\d\mug
				+ \int(\outric*(\nu,\nu)+2+\frac2{\sinh(\Hradius)^2}-\alpha)\,f^2 \d\mug \\
			&	+ \int(\H-\meanH)^2\,f^2 \d\mug
				- 2\,\mean f\int(\outric(\nu,\nu)+\trtr\zFund\zFund-\alpha)f\d\mug \\
			&	+ \mean f^2\int\outric(\nu,\nu)+\trtr\zFund\zFund-\alpha\d\mug \\
 \ge{}&	(1-\delta)\,\Vert\zFundtrf\Vert_{\Lp^4(\M)}^4
				+ (2\,\sinh(\sigma)^{-2}-\alpha)\,\Vert\zFundtrf\Vert_{\Lp^2(\M)}^2 \\
			&	+ C_\delta\,\Vert\H^2-\meanH^2\Vert_{\Lp^2(\M)}^2
				- C\,\exp((1-\decay)\,\Hradius)\,\,\mean f\,\Vert\zFundtrf\Vert_{\Lp^2(\M)} \\
			&	- C\,\exp((2-\decay)\,\Hradius)\,\mean f^2
				- C_\delta\,\exp((2-2\,\decay)\,\Hradius),
\end{align*}
where we used the assumptions on $\outric*(\nu,\nu)$ and $\H$ as well as $f\ge0$ and $\mean f\ge0$ in the second step. Using the lower bound on $\volume{\M}$, the assumption on $\H-\mean\H$, and $\mean f=\fint\vert\zFundtrf\vert_{\g*}\d\mug\le \volume{\M}^{{-}\frac12}\,\Vert\zFundtrf\Vert_{\Lp^2(\M)}$, we get
\begin{align*} \labeleq{Bootstrap_for_trace_free_second_fundamental_form__Lp4__2}
 \Vert\levi*\vert\zFundtrf\vert_{\g*}\Vert_{\Lp^2(\M)}^2
 \ge{}&	(1-\delta)\,\Vert\zFundtrf\Vert_{\Lp^4(\M)}^4
				+ (2\,\sinh(\sigma)^{-2}-\alpha-C\,\exp({-}\decay\,\Hradius))\,\Vert\zFundtrf\Vert_{\Lp^2(\M)}^2 \\
			&	- C_\delta\,\exp((2-2\,\decay)\,\Hradius).
\end{align*}
In \cite[(1.28)]{schoen1975curvature}, Schoen-Simon-Yau proved
\begin{equation*}\labeleq{Regularity_of_the_spheres_Schoen-Simon-Yau}
 (2-\delta)\,\trtr{\levi*\left|\zFundtrf*\right|_{\g*}}{\levi*\left|\zFundtrf*\right|_{\g*}}
	\le \trtr{\levi*\zFundtrf}{\levi*\zFundtrf} + C_\delta\,\vert\outric\vert_{\outg*}^2
\end{equation*}
for minimal surfaces, where we note that Schoen-Simon-Yau did not use their minimality condition to prove the above inequality.\footnote{In fact, they prove this inequality by brilliant algebraic argument in a suitable chosen chart.} For the readers convenience, we repeat their proof of this inequality in our notation in Lemma~\ref{Schoen-Yau75}. 

Now, \eqref{Bootstrap_for_trace_free_second_fundamental_form__Lp4__1}, \eqref{Bootstrap_for_trace_free_second_fundamental_form__Lp4__2}, and \eqref{Regularity_of_the_spheres_Schoen-Simon-Yau} imply
\begin{align*}\MoveEqLeft
 \Vert\levi*\zFundtrf\Vert_{\Lp^2(\M)}^2 + 2\,\sinh(\Hradius)^{-2}\,\Vert\zFundtrf\Vert_{\Lp^2(\M)}^2 \\
	\le{}& (1-\delta)\,\Vert\levi*\zFundtrf\Vert_{\Lp^2(\M)}^2
		- (2\,\sinh(\Hradius)^{{-}2}-\alpha-C\,\exp({-}\decay\,\Hradius))\,\Vert\zFundtrf\Vert_{\Lp^2(\M)}^2 + C_\delta\,\exp(2(1-\decay)\,\Hradius)
\end{align*}
and we get
\begin{equation*}
 \Vert\levi*\zFundtrf\Vert_{\Lp^2(\M)}^2 + (4\,\sinh(\Hradius)^{-2}-\alpha-C\,\exp({-}\decay\,\Hradius))\,\Vert\zFundtrf\Vert_{\Lp^2(\M)}^2 \\
	\le C_\delta\,\exp(2(1-\decay)\,\Hradius).
\end{equation*}
With $4\sinh(\Hradius)^{{-}2}-\alpha\ge\eta\sinh(\Hradius)^{-2}$, this proves the claim for sufficiently large~$\Hradius$, where we keep \eqref{Bootstrap_for_trace_free_second_fundamental_form__Lp4__1} in mind.
\end{proof}
By the above and the Gau\ss\ equation, we have proven a $\Lp^2$-estimate on the Gau\ss\ curvature. As in \cite{nerz2015GeometricCharac}, we use this control to conclude that there exists a \lq good\rq\ conformal parametrization of any regular sphere. This proof is almost identical to the one of \cite[Prop~2.4]{nerz2015GeometricCharac}. However, as some technical details are different, we recall the proof nevertheless for the readers convenience.
\begin{proposition}[Regularity of the spheres, intrinsic version]\label{Intrinsic_Regularity_of_the_spheres}
Let $(\M,\g*)\in\regsphere{\c}$ be a hypersurface with ${-}(4-\eta)\,\sinh(\sigma)^{{-}2}$-controlled instability within a three-dimensional Riemannian manifold $(\outM,\outg*)$, where $\decay\in\interval2*3$, $p\in\interval*2*\infty$, $p^*\le p$, $\c\ge0$, and $\eta\in\interval0*4$ are constants. There are constants $\Hradius_0=\Cof{\Hradius_0}[\eta][p][\c]$ and $C=\Cof[\eta][p][\c][p^*]$ and a conformal parametrization $\varphi:\sphere\to\M$ with corresponding conformal factor $\conformalf\in\Hk^2(\sphere)$, \ie $\varphi^*\g*=\exp(2\,\conformalf)\,\sinh(\Hradius)^2\,\sphg*$, such that
\begin{equation*}
 \Vert\conformalf\Vert_{\Wkp^{2,p^*}(\sphere,\sphg*)} \le C\,\exp((2-\decay)\,\Hradius), \qquad
 \Vert\zFundtrf\Vert_{\Wkp^{1,p^*}(\M)} \le C\,\exp((1+\frac2p-\decay)\,\Hradius) \labeleq{Regularity_of_the_spheres__k}
\end{equation*}
if $\Hradius>\Hradius_0$, where $\sphg*$ denotes the standard metric of the Euclidean unit sphere and $p*:=p$ if $p\in\interval2\infty$ and $p^*<\infty$ arbitrary if $p=\infty$. In particular, the Sobolev inequality holds on $\M$.\pagebreak[2]
\end{proposition}
\begin{proof}
By Lemma~\ref{Bootstrap_for_trace_free_second_fundamental_form__Lp4}, we know $\Vert\sc[r]-2\Vert_{\Lp^2(\M[r])}\le C\,\exp((4-2\,\decay)\,\Hradius)\ll1$ if $\Hradius$ is sufficiently large, where $(\M[r],\g[r]):=\sinh(\Hradius)^{{-}1}\,(\M,\g*)$ and where we used $\decay>2$. Thus, we can use \cite[Thm~A.1]{nerz2015GeometricCharac} for $p'=2$ to conclude that there is a conformal parametrization $\psi$ of $\M$ such that the corresponding conformal factor $\conformalf\in\Hk^2(\M)$ satisfies
$\Vert\conformalf\Vert_{\Hk^2(\M)} \le C\,\exp(\frac{2-\decay}2\,\Hradius)$.
In particular, the Sobolev inequality holds on $\M$, where the Sobolev constant does not depend on $\Hradius$. Thus, $\H\equiv\text{const}$ and the Simon's identity, \ie
\begin{equation*}\labeleq{Simons-identity}
 \laplace\zFund*
	= \div_2\outrc_{\cdot\cdot\cdot\nu} - \levi*\outric_{\nu} + \frac{\H^2}2\zFundtrf + \H\,\trzd\zFundtrf\zFundtrf - \trtr\zFundtrf\zFundtrf\,\zFund* - \trzd{(\tr_{23}\outrc)}\zFundtrf + \outrc_{\cdot I\!J\cdot}\,\zFundtrf^{I\!J},
\end{equation*}
imply $\vert\zFundtrf\vert\le C\,\exp((1-\decay)\,\Hradius)$, \cite{simons1968minimal,schoen1975curvature}. In particular, the Gau\ss\ curvature is positive. Thus, the regularity of the Laplace operator holds on $\M$, where the corresponding constants does not depend on $\Hradius$, see \cite[Cor.~2.3.1.2]{christodoulou1993global}. In particular, \eqref{Simons-identity} proves the second inequality in~\eqref{Regularity_of_the_spheres__k}. Repeating the above argument, we also get the first inequality in~\eqref{Regularity_of_the_spheres__k} for $\Wkp^{2,p^*}$ instead of $\Hk^2$.\pagebreak[3] 
\end{proof}

\begin{remark}[Comparing with Neves-Tian's approach]
Neves-Tian proved an analogous result in their settings. In their later paper \cite{NevesTianExistenceCMC_II} they apply De~Lellis-M\"uller's beautiful theorem that surfaces in the Euclidean space with small tracefree part of the second fundamental form can be written as perturbed coordinate spheres. This was previously also used by Metzger (and later by the author) in the Euclidean setting, \cite{metzger2007foliations,nerz2015CMCfoliation}. To use this theorem in the hyperbolic setting, Neves-Tian applied the conform invariants of $\Vert\zFundtrf\Vert_{\Lp^2(\M)}$ by switching to the Poincar\'e ball model of the hyperbolic space, in which the hyperbolic metric is conformally equivalent to the Euclidean one. Thus, they concluded that $\M$ is close to an \emph{Euclidean} sphere (in this model). To conclude that $\M$ is close to a \emph{hyperbolic} sphere, they used their assumptions on the inner and outer radius on $\M$, \ie that $\max_{\M}\rad\le\min_{\M}\rad+C$. However, their approach heavily relies on their strong decay assumptions (at least $\decay>3$) on $\outg$ which had to satisfy \eqref{Ass_NevesTian_II} and on the strict assumption on the coordinate radii (see above). In particular, it cannot be applied when the decay rate of the error terms is $\exp({-}3\rad)$ or slower.

In their earlier paper \cite{NevesTianExistenceCMC_I}, Neves-Tian used a different approach to show that the CMC-surfaces are graphs over the coordinate sphere and to control their graph function: after deducing some basic inequalities (as we did so far), they proved that the distance $\rad$ from the origin satisfies (in highest order) the partial differential equation $\laplace\rad \approx 1 - \exp(2\rad)$, \ie that $\rad$ is closed to a solution $f\approx\rad-\Hradius$ of
\begin{equation*}\labeleq{const-sc-eq} \sphlaplace f = 1 - \exp(2f) \end{equation*}
and recalled that these functions correspond to metrics $\g:=\exp(2f)\sphg*$ on the Euclidean sphere with constant Gau\ss\ curvature. It is well-known that the set of these metrics is not compact and therefore $\rad$ cannot be controlled by purely analyzing this equation. Instead they proved the control using the assumed asymptotic of the surrounding metric, \ie that $\outg$ is of the form \eqref{Ass_NevesTian_I}. In principle, this is also the approach which we use in the proof of Theorem~\ref{W2pSurfaceRegularity} for the graph function instead of $\rad$. However, we cannot adapt Neves-Tian's proofs as they again depend on their restrictive assumption on the behavior of $\outg*$ near infinity (implying restrictive inequalities on $\outric$ and $\zFundtrf$).\pagebreak[3]
\end{remark}

We will later see that the estimates proven so far are not sufficient for what we need in the later parts of the article. Let us therefore take a step back and compare what we have proven so far with the expectation what should be provable.
\begin{remark}[Expectation of the decay rate of \texorpdfstring{$\zFundtrf$}{the second fundamental form}]\label{Expectation_on_k}
Our approach in Lemma~\ref{Bootstrap_for_trace_free_second_fundamental_form__Lp4} (which is very similar to Neves-Tian's approach) assumes that the error ($\levi*\zFundtrf$) in the first derivative of the second fundamental form decays \emph{one order faster} than the error ($\zFundtrf*$) of the second fundamental form itself. Namely, we have proven
\[ \sinh(\Hradius)^{-1}\Vert\zFundtrf\Vert_{\Wkp^{1,p}(\M)}=\sinh(\Hradius)^{-1}\,\Vert\zFundtrf\Vert_{\Lp^p(\M)}+\Vert\levi*\zFundtrf\Vert_{\Lp^p(\M)} \le C\,\exp((\frac2p-\decay)\Hradius). \nopagebreak\]\nopagebreak
Note the factor $\sinh(\Hradius)^{-1}$ before $\Vert\zFundtrf\Vert_{\Lp^p(\M)}$ which does not appear before $\Vert\levi\zFundtrf\Vert_{\Lp^p(\M)}$. 

This approach led to the optimal decay in the Euclidean setting, however it is \emph{not} optimal in the hyperbolic setting, as it does not consider that the tangential and the \lq radial\rq\ vector fields behave in the hyperbolic setting differently under scaling. To explain the consequences of this, we consider for a moment the coordinate spheres $\hsphere_\Hradius(0)$ in a $\Ck^2_{\decay}$-asymp\-to\-tic\-ally hyperbolic manifold $(\outM,\outg*)$. Here, the metric satisfies
\begin{equation*}
 \Vert\g-\sinh(\rad)^2\sphg\Vert_{\Lp^\infty(\hsphere_\Hradius(0))} \le \Vert\outg-\houtg\Vert_{\Lp^\infty(\hsphere_\Hradius(0))} \le C\,\exp({-}\decay\,\Hradius)
\end{equation*}
and second fundamental form satisfies
\begin{align*}
 \Vert\zFundtrf<\Hradius>\Vert_{\Lp^\infty(\hsphere_\Hradius(0))} \le{}& \Vert\lieD{\nu}(\outg-\houtg)\Vert_{\Lp^\infty(\hsphere_\Hradius(0))} \le C\,\exp({-}\decay\,\Hradius), \\
 \Vert\levi<\Hradius>*\zFundtrf<\Hradius>\Vert_{\Lp^\infty(\hsphere_\Hradius(0))} \le{}& \Vert\outlevi*\lieD{\nu}(\outg-\houtg)\Vert_{\Lp^\infty(\hsphere_\Hradius(0))} + \Vert\lieD{\nu}(\outg-\houtg)\Vert_{\Lp^\infty(\hsphere_\Hradius(0))}^2 \le C\,\exp({-}\decay\,\Hradius).
\end{align*}
In particular, the error of the second fundamental form has the \emph{same} decay rate as the error of its derivative and is therefore \emph{not} of lower decay rate, as it was implicitly assumed in Lemma~\ref{Bootstrap_for_trace_free_second_fundamental_form__Lp4} and in the articles by Neves-Tian. In other words, we should be able to prove an inequality of the form
\[ \Vert\g-\sinh(\rad)^2\sphg\Vert_{\Lp^\infty(\M)} + \Vert\zFundtrf\Vert_{\Lp^\infty(\M)}+\Vert\levi\zFundtrf\Vert_{\Lp^\infty(\M)} \le C\,\exp({-}\decay\Hradius), \]
because this is true for the coordinate spheres, and the \emph{geometric spheres} (the CMC-surfaces) should (at least) be as round as the coordinate ones. Furthermore, we see that we should not get stronger inequalities (on $\zFundtrf$ etc.) than the ones stated above, because they are already as strong as the ones of the surrounding space, \eg if $\zFundtrf$ would decay faster, then $\lieD{\nu}(\outg-\houtg)$ would also decay faster which would lead to a stronger decay of $\outg-\houtg$ than the one we assumed and we cannot hope for such a \lq decay boostrap\rq\ for $\outg-\houtg$.\smallskip

As explained above, we cannot use Sobolev, Poincar\'e, or similar inequalities on $\zFundtrf$ to conclude such a result, because these are only optimal when the derivative ($\levi*\zFundtrf$) decays one order faster than the tensor ($\zFundtrf$) itself, see Remark~\ref{Scaling_Of_Sobolev}.
Thus, instead of adding one additional derivative to $\zFundtrf$---and looking at $\levi*\zFundtrf*$---\nolinebreak, we \emph{reduce} the derivatives by two and look at $\vert\outx-p_1\vert$ being the distance function to $\hsphere_\Hradius(p_1)$ for some $p_1\in\R^3$. More precisely, we look at the graph function $\graphf\in\Wkp^{2,p}(\hsphere_\Hradius(p))$---after proving that it exists. Then, we establish optimal estimates for $\graphf$ implying the estimates on $\zFundtrf$ \etc explained above.
\end{remark}

\begin{proof}[Proof of Theorem~\ref{Regularity_Theorem}]$\,\!$%
\nopagebreak\par\nopagebreak
\emph{Step 1 {\normalfont(}$\M$ is regular sphere{\normalfont)}:}
First, we have to prove that the mean curvature radius $\Hradius$ is bounded from above by the coordinate radius $\Rradius$ or more precisely that $\int\exp(2\Hradius-2\rad)\d\mug\le C$. Here, we use an argument by Neves-Tian: A direct calculation presented in \cite[Prop.~3.4]{NevesTianExistenceCMC_I} shows
\begin{align*}
	\Big|&\laplace\rad - \Big( \underbrace{(4-2\vert X\vert^2)}_{\ge2}\exp({-}2\rad) + (\H+2) + \underbrace{\vphantom{\Big|}(\H+2)(\outg*(\houtlevi\rad,\nu)-1)}_{\ge0} + \underbrace{\vphantom{\Big|}\vert X\vert^2}_{\ge0} \Big)\Big| \\
	&\qquad\le C\,\exp({-}\decay\rad),
\end{align*}
where $X$ denotes the orthogonal projection of $\partial*_r$ to the tangent space of $\M$.
In particular, we have
\[ \laplace\rad\ge \exp({-}2\rad)(2-C\exp((2-\decay)\rad)) - \sinh(\Hradius)^{{-}2}. \]
By integrating this inequality, we get
\begin{equation*}\labeleq{Hradius-meanvalue-rad}
 \sinh(\Hradius)^{{-}2} \ge	(2-C\exp((2-\decay)\rradius))\fint\exp({-}2\rad)\d\mug \ge \fint\exp({-}2\rad)\d\mug
\end{equation*}
being the inequality which we wanted. Note that this implies that $\Hradius$ is large if $\rradius$ is large and that $\fint\exp(2\Hradius-\decay\rad)$ is arbitrary small if $\rradius$ is sufficiently large.

Now, let us assume that $\c^{-1}\exp(2\Hradius)\le\volume{\M}\le\c\exp(2\Hradius)$ is not a~priori true. In particular, we have $\eta>2$ and by the above $\int\exp(2(\Hradius-\rad))\d\mug\le C\sinh(\Aradius)^2$. We start as in \cite[Lemma~4.1]{NevesTianExistenceCMC_I} and use the test functions $\varphi_\oi:=\outx^i\circ\psi^{-1}$, where $\psi:\sphere\to\M$ is a conformal parametrization of $\M$ with $\int\varphi_i\d\mug=0$. These were already used by Huisken-Yau in \cite[Prop.~5.3]{huisken_yau_foliation} and were based on an idea by Christodoulou-Yau, \cite{christodoulou71some}. By the controlled instability assumption, this implies
\begin{equation*}
 {-}\frac{8\pi}3
 =	\int_{\sphere}\outx_\oi\,\sphlaplace\outx_\oi\d\sphmug
 =	\int_{\M}\varphi_\oi\,\laplace\varphi_\oi\d\mug 
 \le \int(\frac{4-\eta}{\sinh(\Hradius)^2}-\trtr\zFund\zFund-\outric(\nu,\nu))\varphi_\oi^2\d\mug
\end{equation*}
for every $\oi\in\{1,2,3\}$, where we have used the conformal invariance of $\laplace f\,\d\mug$. Now, we recall that $(\sum_i\varphi_i^2)\circ\psi=\sum_i\outx_i^2\equiv1$ to get
\begin{align*}
 8\pi
 \ge{}& \int\trtr\zFundtrf\zFundtrf + \frac{\H^2-4}2
					+ 2+\outric(\nu,\nu)
					- \frac{4-\eta}{\sinh(\Hradius)^2} \d\mug \\
 \ge{}& \Vert\zFundtrf\Vert_{\Lp^2(\M)}^2
				+ \int\frac{\eta-2}{\sinh(\Hradius)^2}
				- C\int \exp({-}\decay\rad) \d\mug. \labeleq{Triv_on_L2_Zfundtrf}
\end{align*}
In particular, \eqref{Hradius-meanvalue-rad} gives
\[ (\eta-2)\sinh(\Aradius)^2 \le 8\pi\sinh(\Hradius)^2 + C\,\exp((2-\decay)\rradius)\,\sinh(\Aradius)^2 \]
implying $\Aradius\le\Hradius+C$, \ie $\volume{\M}\le C\exp(2\Hradius)$. On the other hand, the Gau\ss-Bonnet theorem and the Gau\ss\ equation combined with equations \eqref{Hradius-meanvalue-rad} and \eqref{Triv_on_L2_Zfundtrf} give
\begin{align*}
 8\pi
  ={}& \int\sc\d\mug
	= \int(\outsc-2\outric(\nu,\nu)-\trtr\zFundtrf\zFundtrf+\frac{\H^2}2\d\mug) \\
	\le{}& C\,\int \exp({-}\decay\rad) \d\mug
					+ \int\frac2{\sinh(\Hradius)^2}\d\mug
	\le (8\pi+C\,\exp((2-\decay)\rradius))\,\frac{\sinh(\Aradius)^2}{\sinh(\Hradius)^2}
\end{align*}
implying $\Aradius\ge\Hradius+C$, \ie $C^{-1}\exp(2\Hradius)\le\volume{\M}$.

Now, we look again on the general setting. By the above, we have $C^{-1}\,\exp(2\Hradius) \le \volume{\M} \le C\,\exp(2\Hradius)$ and by \eqref{Hradius-meanvalue-rad} we furthermore know $\Hradius\le\Rradius+C\le(1+\zeta)\rradius+C$. In particular, we see that $\M\in\regsphere[\decay']{C}$ for some constant $C=\Cof[\decay][\oc][\zeta][\eta][c]$, $\decay'=\frac{\decay}{1+\zeta}$, and every $p\in\interval*1*\infty$.\smallskip

\emph{Step 2 {\normalfont(}applying everything proven so far and Theorem~\ref{W2pSurfaceRegularity}{\normalfont)}}
By the above, we can apply Lemma~\ref{Bootstrap_for_trace_free_second_fundamental_form__Lp4} and Proposition~\ref{Intrinsic_Regularity_of_the_spheres} for $\decay'$ instead of $\decay$. In particular, the Sobolev inequality holds and therefore the Simon's identity and Lemma~\ref{Bootstrap_for_trace_free_second_fundamental_form__Lp4} imply
\[ \Vert\zFundtrf\Vert_{\Lp^\infty(\M)} \le C\,\exp((1-\decay')\Hradius), \]
see for example \cite[Prop.~D.2]{nerz2015CMCfoliation}. Due to the assumptions on $\outg$, this implies for sufficiently large $\Hradius$ that
\begin{equation*} \Vert\hzFund+\hg\Vert_{\Lp^\infty(\M)} \le C\,\exp((1-\decay')\Hradius), \qquad
 \Vert\hH+2\,\frac{\sinh(\Hradius)}{\cosh(\Hradius)}\Vert_{\Lp^\infty(\M)} \le C\,\exp({-}\decay'\Hradius).
\end{equation*}
Therefore, Theorem~\ref{W2pSurfaceRegularity} implies the existence of an isometry $\hiso$ of the hyperbolic space and of a function $\graphf\in\Wkp^{2,p}(\sphere_\Hradius(0))$ as in the claim of the theorem. Using the regularity of the (weak) Laplace operator on the Simon's identity \eqref{Simons-identity}, we get the claimed estimate on $\levi*\zFundtrf$, too. This proves the claim.
\end{proof}

\begin{remark}[Not controlling the center]\label{Remark_NotControllingTheCenter}
In \cite[Sect.~7]{NevesTianExistenceCMC_I}, Neves-Tian proved a direct control on the hyperbolic center of the CMC surfaces\footnote{although the notation of the hyperbolic center was only later introduced by Cederbaum-Cortier-Sakovich, \cite{cederbaum2015center}}. To do so, they brilliantly use the Kazdan-Warner identity to prove that this center is (in highest order) identical to the center of mass which they assumed to vanish (by using balanced coordinates). Here, they had to explicitly use that the Ricci curvature $\outric$ is (in highest order) characterized by the mass aspect function $\sphtr\,\massaspect$. In particular, they had to exploit the fact that the error term in the metric decays faster than the first non-constant term of the Ricci curvature. 
In our setting, we have that the first non-constant term in the Ricci curvature can be of order $\exp({-}\decay\rad)$ being the same order as the one of the error term $\houtg-\outg$. Therefore, we cannot use a method similar to the one by Neves-Tian. Instead, we apply an a~posteriori ansatz to control the center of the CMC-surfaces, see Section~\ref{proofmaintheorem}.
\end{remark}

\section{Stability of the surfaces}\label{Section-Stability}
In this section, we prove the stability of closed CMC-hyper\-surfaces in three-dimen\-sional, asymptotically hyperbolic Riemannian manifolds---at least for those within our class of surfaces satisfying~\eqref{class-assumption}. Although, there are \emph{conceptually} large differences to the Euclidean setting (see Remark~\ref{LinearizedBoosts}), \emph{formally} the argument is more or less identical with the one used by the author in the asymptotically Euclidean setting, see \cite{nerz2015CMCfoliation,nerz2015GeometricCharac}. We recall the main steps here nevertheless for the readers convenience and to point out the conceptional differences.\smallskip

First, we do a spectral analysis of the stability operator in order to prove that any regular CMC-sphere is stable: We see that the eigenvalues of the stability operator $\jacobiext*$ of a $\Ck^2_{\decay}(\c)_\Hradius$-round sphere $\M$ are of order $\sinh(\Hradius)^{-2}$ except for three eigenvalues of smaller order, where the stability operator of $\M$ is the linearization of the \emph{mean curvature map}. It is characterized by
\[ \jacobiext*f = \laplace f + (\outric*(\nu,\nu) + \trtr\zFund\zFund)\,f \qquad\forall\,f\in\Hk^2(\M), \]
for more details see Proposition~\ref{Stability} or (in a more general context) \cite{barbosa2012stability}. As we prove in Proposition~\ref{Stability}, the corresponding partition of $\Hk^2(\M)$ (respectively $\Lp^2(\M)$) is (asymptotically) given as follows.
\begin{definition}[Canonical partition of \texorpdfstring{$\Lp^2$}{L-2}]\label{Canonical_partition}
Let $\M$ be a $\Ck^2_{\decay}(\c)_\Hradius$-round sphere in $\outM$. Let $\boost g$ be the $\Lp^2(\M)$-orthogonal projection of a function $g\in\Lp^2(\M)$ on the linear span of eigenfunctions of the (negative) Laplacian with eigenvalue $\lambda$ satisfying $\vert\lambda-2\sinh(\Hradius)^{{-2}}\vert\le \frac32\,\sinh(\Hradius)^{-2}$, \ie
\[ \boost g := \sum\lbrace \eflap_i\,\int_{\M} g\,\eflap_i \d\mug \ \middle|\ \frac12 \le \sinh(\Hradius)^2\,\ewlap_i \le \frac72\rbrace\qquad\forall\,g\in\Lp^2(\M), \]
where $\{f_i\}_{i=0}^\infty$ denotes a complete orthonormal system of $\Lp^2(\M)$ by eigenfunctions $\eflap_i$ of the (negative) Laplace operator with corresponding eigenvalue $\ewlap_i$ satisfying $0\le\ewlap_i\le\ewlap_{i+1}$. Finally, $\deform g:=g-\boost g$ denotes the rest of such a function $g\in\Lp^2(\M)$. Elements of $\boostLp^2(\M):=\lbrace\boost f:f\in\Lp^2(\M)\rbrace$ are called \emph{linearized boosts} and those of $\deformLp^2(\M):=\lbrace\deform f:f\in\Lp^2(\M)\rbrace$ are called \emph{deformations}.
\end{definition}
Note that we here included rescalings $(g\equiv \mean g:=\fint g\d\mug$) in the deformations.

\begin{remark}[Linearized boosts]\label{LinearizedBoosts}
It is important to note that the space $\boostLp^2(\M)$ does \emph{not} correspond to translations as it does in the asymptotically Euclidean setting, \cite[Def.~2.6]{nerz2015CMCfoliation}. Instead it corresponds to the \emph{linearization of boosts} of the surrounding space.\footnote{This is why we use the notation $\boost{(\,{\cdot}\,)}$ instead of ${(\,{\cdot}\,)}^{\text t}$ which we used in the Euclidean setting.}

To see this, we recall that the isometries of the hyperbolic space (as hyperboloid in the Minkowski spacetime) are given by the boosts and the rotations. If $\hiso:\interval-\ve\ve\times\R^3\to\R^3$ is a family of hyperbolic isometries, \ie $\hiso(t,\,{\cdot}\,):\R^3\to\R^3$ is an isometry of $\houtg$ for every $t\in\interval-\ve\ve$, and $P(0,\,{\cdot}\,)=\id_{\R^3}$, then $X:=(\partial*_t\hiso)(0)$ denotes its linearization and is a killing vector field. This killing vector field consists of two parts, one rotation vector field (tangential to every sphere $\hsphere_\rradius(0)$ around the origin) and a boost (orthogonal to every sphere $\hsphere_\rradius(0)$ around the origin). Therefore, the lapse function $\rnu<r>:=\houtg(X,\hnu)$ on the sphere $\hsphere_r(0)$ \lq measures\rq\ the \lq boost part\rq\ of the linearization $X$ of $\hiso$. By a direct calculation, we see that it satisfies $\rnu<r>=\boost{\rnu<r>}$. However as boosts are non-linear (in contrast to translations in the Euclidean setting which are linear), $\rnu<r>$ is not independent of~$r$.

Now, we connect the linearized boosts with the stability operator: In the notation used above, $\hiso(t,\hsphere_\rradius(0))$ has constant mean curvature which is independent of~$t$ as $\hiso(t,\,{\cdot}\,)$ is a isometry. In particular, the lapse function must lay within the kernel of the linearization of the mean curvature---being the stability operator. By a dimension argument, we see that $\boostLp^2(\M)$ is the kernel of the stability operator. This explains why these functions are the problematic ones when we study stability of CMC-surfaces in asymptotically hyperbolic manifolds. Furthermore, we see that we have to use the mass as it is what breaks the boost invariance of our metric.
\end{remark}

Now, we can prove the announced stability proposition which is one of the central tools for the proofs of the main theorems. Note that we need here the assumption on the scalar curvature for the first time. A formally identical argument was used in the asymptotically flat setting by the author in \cite{nerz2015CMCfoliation,nerz2015GeometricCharac} and a familiar argument (implicitly) by Huang \cite{Huang__Foliations_by_Stable_Spheres_with_Constant_Mean_Curvature}. We do not give here the proof as it is (formally) identical to the one of \cite[Prop.~2.7]{nerz2015CMCfoliation} if we replace \cite[Prop.~2.4]{nerz2015CMCfoliation}, the Euclidean mean curvature radius $\Hradius$, and the translational part $(\,{\cdot}\,)^t$ in \cite{nerz2015CMCfoliation} by Theorem~\ref{Regularity_Theorem}, the (Euclidean) area radius $\eukAradius:=\sqrt{(4\pi)^{-1}\,\volume{\M}}\approx\sinh(\Hradius)$, and the linearized boost $\boost{(\,{\cdot}\,)}$, respectively.
\begin{proposition}[Stability]\label{Stability}
Let $\decay=:\frac52+\outve\in\interval{\frac52}*3$, $\scdecay\ge3+\outve$, $p\in\interval*2*\infty$, $\c\ge0$, and $\eta\in\interval0*4$ be constants and let $(\M,\g*)\in\strregsphere{\c}$ be a strictly regular sphere with radius $\Hradius$ in a three-di\-men\-sio\-nal Riemannian manifold $(\outM,\outg*)$. If $\M$ has ${-}(4-\eta)\,\sinh(\sigma)^{{-}2}$-controlled instability, then there are constants $\Hradius_0=\Cof{\Hradius_0}[\outve][\c][|\HmHaw|]$ and $C=\Cof[\outve][\c][|\HmHaw|]$ such that
\begin{align*}
 \vert \int_{\M} (\jacobiext*\boost g)\,\boost h\d\mug - \frac{6\,\HmHaw}{\sinh(\Hradius)^3}\,\int_{\M}\boost g\,\boost h\d\mug \vert \le{}& C\,\exp({-}(3+\outve)\Hradius)\,\Vert\boost g\Vert_{\Lp^2(\M)}\,\Vert\boost h\Vert_{\Lp^2(\M)}, \labeleq{Stability__t} \\
 \Vert \deform g\Vert_{\Lp^2(\M)} \le{}& \sinh(\Hradius)^2 \Vert\hspace{.05em}\jacobiext*\deform g\Vert_{\Lp^2(\M)} \labeleq{Stability__d}
\end{align*}
for every $g,h\in\Hk^2(\M)$ if $\Hradius>\Hradius_0$, where $\HmHaw=\HmHaw(\M)$ denotes the hyperbolic Hawking mass of $\M$. Furthermore, the corresponding $\Wkp^{k,p}$-inequalities
\begin{align*}
 \Vert\boost g\Vert_{\Wkp^{3,p}(\M)}	\le{}& (\frac{\sinh(\Hradius)^3}{6\,\vert\HmHaw\vert}+C\,\exp((3-\outve)\Hradius))\,\Vert\jacobiext*g\Vert_{\Lp^p(\M)}, \\
 \Vert\deform g\Vert_{\Wkp^{2,p}(\M)} \le{}& C\,\exp(2\Hradius)\,\Vert\jacobiext*g\Vert_{\Lp^p(\M)}, \\
 \Vert\Hesstrf\,g\Vert_{\Lp^p(\M)}		\le{}& C\,\exp((\frac12-\outve)\Hradius)\,\Vert\jacobiext*g\Vert_{\Lp^p(\M)}
\end{align*}
hold for every function $g\in\Wkp^{2,p}(\M)$ and $p\in\interval*2\infty$ if $\Hradius>\Hradius_0$.

If in the above setting $\efjac\in\Hk^2(\M)$ is an eigenfunction of ${-}\jacobiext*$ with corresponding eigenvalue $\ewjac\in\interval-{\frac32\sinh(\Hradius)^{{-}2}}{\frac32\sinh(\Hradius)^{{-}2}}$, then
\begin{equation*}
	\Vert\,\deform{\efjac}\Vert_{\Hk^2(\M)} \le C\,\exp({-}(\frac12+\outve)\Hradius)\Vert\,\efjac\Vert_{\Hk^2(\M)},\qquad \vert\,\ewjac-\frac{6\,\HmHaw}{\sinh(\Hradius)^3}\vert \le C\,\exp({-}(3+\outve)\,\Hradius). \labeleq{Stability__ewjac} \end{equation*}
\end{proposition}
\begin{remark}
Note that if only the weaker assumption on the scalar curvature $\outsc$ is valid, see Remark~\ref{AlternativeAssumption_outsc_regsphere}, then \eqref{Stability__ewjac} is not true as stated above. Instead we get
\begin{equation*}
	\Vert\,\deform{\efjac}\Vert_{\Hk^2(\M)} \le C\,\exp({-}(\frac12+\outve)\Hradius)\Vert\,\efjac\Vert_{\Hk^2(\M)},\qquad \ewjac\ge\frac{6\,\HmHaw}{\sinh(\Hradius)^3} - C\,\exp({-}(3+\outve)\,\Hradius) \tag{\ref{Stability__ewjac}'}. \end{equation*}
Note furthermore that the Gau\ss-Bonnet theorem, the Gau\ss\ equation, and Proposition~\ref{Intrinsic_Regularity_of_the_spheres} imply
\[ \vert\HmHaw - \frac{\volume{\M}^{\frac12}}{16\pi^{\frac32}}\int(\outric-(\frac\sc2+1)\outg)(\nu,\nu)\d\mug\vert \le C\,\exp({-}\outve\Hradius) \xrightarrow{\Hradius\to\infty} 0. \]
In particular, if $\M$ and $\hiso$ are as in Theorem~\ref{Regularity_Theorem}, then $\vert\HmHaw-\mass^0\vert\le C\,\exp({-}\outve\Hradius)$ and we can therefore replace $\HmHaw$ by $\mass^0$. We will later see that in our setting $\hiso\circ\outx$ are balanced coordinates and thus $\HmHaw$ can be replaced by the total mass $\mass*\in\{\pm\vert\mass\vert_{\R^{3,1}}\}$.
\end{remark}

\section{Existence and uniqueness of the CMC-foliation}\label{existence_of_the_CMC-foliation}\label{proofmaintheorem}
In this section, we prove the existence of the CMC-foliation as well as the uniqueness and stability of the leaves of the foliation. This is done analogously to \cite[Sect.~3]{nerz2015CMCfoliation} which uses the proof structure of \cite{metzger2007foliations} also used in \cite{NevesTianExistenceCMC_II}. In contrast to these works, the proof of the uniqueness presented here has to deal with an additional obstruction due to the fast that the CMC-surfaces are a priori non-concentric (around the center of mass), see Remark~\ref{Remark_NotControllingTheCenter}. \smallskip\pagebreak[3]

Let us start by stating the existence and uniqueness results.\nopagebreak
\begin{theorem}[Existence of the CMC-foliation]\label{Existence_Theorem_full}
Let $(\outM,\outg*)$ be a three-dimensional $\Ck^2_{\decay,\scdecay}$-asymptotically hyperbolic manifold with timelike mass vector and decay rates $\decay>\frac52$ and $\scdecay>3$. There are a finite radius $\Hradius_0$, a compact set $\outsymbol K\subseteq\outM$, and a diffeomorphism $\outPhi:\interval{\Hradius_0}\infty\times\sphere\to\outM\setminus\outsymbol K$ such that $\M<\Hradius>:=\Phi(\Hradius,\M)$ has constant mean curvature $\H<\Hradius>\equiv{-}2\:\frac{\cosh(\Hradius)}{\sinh(\Hradius)}$ and the sequence of the hyperbolic centers of $\M<\Hradius>$ converge to the center of mass of $(\outM,\outg*)$ as $\Hradius\to\infty$.
\end{theorem}
Here, we use the definition of the hyperbolic center of a surface and the center of mass as as defined by Cederbaum-Cortier-Sakovich, \cite[Def.~3.10]{cederbaum2015center}, see Definition~\ref{center}. Note that this result does \emph{not} correspond to the (non-)existence result of the center of mass in the asymptotically flat setting, see \cite{cederbaumnerz2013_examples}.\smallskip

\begin{theorem}[Uniqueness of the CMC-foliation]\label{Uniqueness_Theorem_full}
Let $(\outM,\outg*)$ be a three-dimensional $\Ck^2_{\decay,\scdecay}$-asymptotically hyperbolic manifold with timelike mass vector and decay rates $\decay\in\interval{\frac52}*3$ and $\scdecay>3$. For all constants $\c\ge0$, $\eta\in\interval0*4$, and $0\le\zeta<\min\{\frac{2\decay-5}{2\decay},\frac{\scdecay-3}3\}$ there is a constant $\rradius_0=\Cof{\rradius_0}[\decay][\scdecay][\zeta][\c][\eta]$ with the following property:

If $\M\hookrightarrow\outM$ is a closed hypersurfaces diffeomorphic to $\sphere$ with constant mean curvature $\H\equiv{-}2\,\frac{\cosh(\Hradius)}{\sinh(\Hradius)}$ and ${-}(4-\eta)\sinh(\Hradius)^{{-}2}$-controlled instability satisfying
\[ \Rradius \le (1+\zeta)\rradius, \qquad\rradius\ge\rradius_0, \qquad
	\eta>2\text{ or } \c^{-1}\exp(2\Hradius)\le\volume{\M}\le\c\,\exp(2\Hradius), \nopagebreak\]
then $\M$ is the CMC-surface $\M<\Hradius>$ defined in Theorem~\ref{Existence_Theorem_full}.\pagebreak[3]
\end{theorem}
Furthermore, our proof of this existence result includes the following stability and regularity statements for each of the leaves of the foliation.\nopagebreak
\begin{theorem}[Regularity of the CMC-leaves]\label{Regularity_Theorem_full}
Let $p\in\interval1\infty$ be a constant and $(\outM,\outg*)$ be a three-dimensional $\Ck^2_{\decay,\scdecay}$-asymptotically hyperbolic manifold with timelike mass vector and decay rates $\decay\ge\frac52+\outve\in\interval{\frac52}*3$, $\scdecay\ge3+\outve$. There exist a constant $C=\Cof[\outve][\oc][\mass*]$ and functions $\graphf<\Hradius>\in\Wkp^{2,p}(\hsphere_\Hradius(\centerz<\Hradius>))$ and $\graphf<\Hradius>'\in\Wkp^{2,p}(\hsphere_\Hradius(\outcenterz))$ such that $\M = \hgraph(\graphf<\Hradius>) = \hgraph(\graphf<\Hradius>')$ and
\begin{equation*}
  \Vert\graphf<\Hradius>'\Vert_{\Wkp^{2,p}(\hsphere_\Hradius(\outcenterz))}
	\le C\,\exp((\frac2p-\outve)\Hradius), \quad
 \Vert\zFundtrf\Vert_{\Lp^p(\M)} + \exp({-}2\Hradius)\Vert\graphf<\Hradius>\Vert_{\Wkp^{2,p}(\hsphere_\Hradius(\centerz<\Hradius>))}
	\le C\,\exp((\frac2p-\decay)\Hradius),
\end{equation*}
where $\M<\Hradius>$, $\hsphere_r(p)$, $\centerz<\Hradius>$, $\outcenterz$ denote the CMC-leaves of Theorem~\ref{Existence_Theorem_full}, the geodesic ball of radius $r$ around $p\in\hyperbolicspace$, the hyperbolic centers of $\M<\Hradius>$, and the hyperbolic center of mass $\outcenterz$, respectively.

Furthermore, the $i^{\text{th}}$-smallest eigenvalue $\ewjac<\Hradius>_i$ of the {\normalfont(\hspace{-.05em}}negative{\hspace{.05em}\normalfont)} stability operator of $\M<\Hradius>$ satisfies
\begin{equation*}
 \ewjac<\Hradius>_0 \le \frac{{-}3}{2\sinh(\Hradius)^2}, \quad
 \vert\ewjac<\Hradius>_i-\frac{6\,\mass*}{\sinh(\Hradius)^3} \vert \le C\,\exp(-(3+\outve)\,\Hradius),\quad
 \ewjac<\Hradius>_j \ge \frac 3{2\sinh(\Hradius)^2}
\end{equation*}
where $i\in\{1,2,3\}$, $j>3$, and where $\outve>0$ with $\decay\ge\frac52+\outve$ and $\scdecay\ge3+\outve$.
\end{theorem}

\begin{assumptions}[Assumptions for the following]\label{Assumptions_for_the_following}
Let $(\outM,\outg,\outx)$ be a three-di\-men\-sio\-nal $\Ck^2_{\decay,\scdecay}$-asymp\-to\-tic\-ally hyperbolic Riemannian manifold with \emph{positive} mass $\mass$ and decay rates $\decay>\frac52$ and $\scdecay>3$ and set $\outve:=\min\{\decay-\frac52,\scdecay-3,\frac12\}$. Without loss of generality, the balancing condition holds, \ie
\begin{equation*} \mass = (\totalmass,0,0,0), \labeleq{balanced_assumption}\end{equation*}
see Remark~\ref{remarks_on_mass}. Let $\AdSoutg$ be the Anti-de Sitter metric of mass $\mass=(\mass*,0,0,0)$, in particular
\[ \AdSoutg := \d r^2 + \sinh(\rad)^2\,(1+\frac{\mass+{\normalfont \textrm{Err}}}{3\,\sinh(\rad)^3})\,\sphg*, \]
where $\text{\normalfont Err}\in\Ck^\infty(\outM)$ satisfies
\[ \vert \houtlevi*^{(k)}\,{\normalfont \textrm{Err}}\vert_{\houtg*} \le C_k\,\exp({-2}\,\rad) \qquad\forall\,k\in\N. \]
Furthermore, let $\{\outg[\atime]*\}_{\atime\in\interval*0*1}$ be the family of convex combinations of $\outg[0]*:=\AdSoutg*$ and $\outg[1]*:=\outg*$, \ie $\outg[\atime]*:=\AdSoutg*+\atime\,(\outg*-\AdSoutg*)$.
\end{assumptions}
We note that the idea to choose \emph{balanced coordinates} was already explained and used by Cederbaum-Cortier-Sakovich in \cite[Thm~3.9]{cederbaum2015center}---under their assumptions.

\begin{assumptions}[Existence and regularity intervals]
Let $\Hradius\gg0$ be a constant. Assume that $\Phi:\intervalI\times\sphere\to\outM$ is a $\Ck^1$-map such that
\begin{enumerate}[nolistsep,label={\normalfont($\intervalI$-\arabic{*})}]
\item $\intervalI\subseteq\interval*0*1$ is a interval with $0\in\intervalI$; \label{IntervalI_IT}
\item $\M[\atime]:=\Phi(\atime,\sphere)$ has constant mean curvature $\H[\atime]\equiv{-}2\,\frac{\cosh(\Hradius)}{\sinh(\Hradius)}$ with respect to $\outg[\atime]*:=\AdSoutg*+\atime\,(\outg*-\AdSoutg*)$;
\item $\M[0]=\sphere_{\rradius(\Hradius)}(0)$ for the specific radius $\rradius(\Hradius)$ for which $\H[0]<\rradius(\Hradius)>\equiv{-}2\,\frac{\cosh(\Hradius)}{\sinh(\Hradius)}$; \label{IntervalI_ID}
\item $\partial*_\atime\Phi$ is orthogonal to $\M[\atime]$ for every $\atime\in\intervalI$.
\end{enumerate}
Furthermore, let $\intervalI$ be maximal, \ie if $\Phi':\intervalI'\times\sphere\to\outM$ satisfies the above assumptions for the same $\Hradius$, then $\intervalI'\subseteq\intervalI$.\smallskip

Let $\c\ge0$, $\eta\in\interval0*4$, and $\zeta\in\interval*0{1-\frac5{5+2\outve}}$ be fixed constants and let $\intervalJ\subseteq\intervalI$ be the maximal interval containing~$0$ such that
\begin{enumerate}[nolistsep,label={\normalfont($\intervalJ$-\arabic{*})}]
\item $\eta>2$ or $\c^{-1}\exp(2\Hradius)\le\volume{\M}\le\c\,\exp(2\Hradius)$ and (in both cases) $\Rradius \le (1+\zeta)\rradius$; \label{IntervalJ-regularity}
\item $\M[\atime]$ has ${-}(4-\eta)\,\sinh(\Hradius)^{-2}$-con\-trolled instability. \label{IntervalJ-instability}.
\end{enumerate}
\end{assumptions}
If we choose $\Hradius$ and $\c$ sufficiently large then such a $\Phi$ exists for some $\intervalI\supseteq\intervalJ\supseteq\{0\}$, because $\AdSoutg*$ is the AdS-metric with mass $\mass\neq0$. Now, we first show that $\intervalI$ contains a neighborhood of $\intervalJ$ in $\interval*0*1$, then that $\intervalJ$ is open and closed in $\intervalI$, \ie $\intervalJ=\intervalI$. This implies that $\intervalI$ is open in $\interval*0*1$ and a simple convergence argument finishes the proof that $\intervalI=\interval*0*1$.
\begin{lemma}[\texorpdfstring{$\intervalI$}I is a neighborhood of \texorpdfstring{$\intervalJ$}J]\label{I_is_a_neighborhood_of_J}
There is a constant $\Hradius_0=\Cof{\Hradius_0}[\outve][\mass][\eta][\zeta][\zeta'][\c]$ such that $\intervalI$ contains a neighborhood of $\intervalJ$ in $\interval*0*1$ if $\Hradius>\Hradius_0$. Furthermore, $\Phi$ is uniquely defined on a neighborhood of $\intervalJ$.
\end{lemma}
\begin{proof}
This is implied by Prop.~\ref{Stability}. The full details of the proof are identical to the ones explained in \cite[Lemma~3.5]{nerz2015CMCfoliation}.
\end{proof}
As in \cite[Lemma~3.5]{nerz2015CMCfoliation}, we see furthermore that $\outx\circ\Phi$ is differentiable as a map from $\intervalI$ to $\Wkp^{2,p}(\sphere;\R^3)$. Thus, all quantities used in the definition of $\intervalJ$ depend continuously on $\atime\in\intervalJ$. As the assumptions of elements in $\intervalJ$ are closed ones, this implies that $\intervalJ$ is closed in $\intervalI$. Therefore, we only have to prove that $\intervalJ$ is open in $\intervalI$ to conclude $\intervalI=\intervalJ$. To do so, we need the following estimates on the lapse function of $\Phi$:
\begin{lemma}[First estimates on the lapse function]\label{First_estimates_on_lapse}
Let $\rnu[\atime]:=\outg[\atime]*(\partial*_\tau\Phi(\atime),\nu[\atime])$ denote the lapse function of $\Phi$, where $\nu[\atime]$ is the outer unit normal of $\M\hookrightarrow(\outM,\outg[\atime]*)$. For every $p\in\interval1\infty$, there is a constant $C=\Cof[\outve][\mass][\eta][\zeta][\c][p]$ independent of $\Hradius$ and $\atime$ such that for every $\atime\in\intervalJ$
\begin{align*}
	\Vert\deform{\rnu[\atime]}\Vert_{\Wkp^{2,p}(\M[\atime])}
		\le{}& C\,\exp((2+\frac2p-\decay)\,\Hradius),		&
	\Vert\rnu[\atime] \Vert_{\Wkp^{2,p}(\M[\atime])}
		\le{}& C\,\exp(\frac2p\,\Hradius).
	\labeleq{first_estimates_on_lapse}
\intertext{If the isometry $\hiso$ of Theorem~\ref{Regularity_Theorem} can be chosen such that $\hiso\circ\outx$ are balanced coordinates, $\mass(\hiso\circ\outx)=(\mass*,0,0,0)$, then}
	\Vert\deform{\rnu[\atime]}\Vert_{\Wkp^{2,p}(\M[\atime])}
		\le{}& C\,\exp((2+\frac2p-\decay)\,\Hradius),		&
	\Vert\rnu[\atime] \Vert_{\Wkp^{2,p}(\M[\atime])}
		\le{}& C\,\exp((\frac2p-\outve)\,\Hradius).
	\labeleq{second_estimates_on_lapse}
\end{align*}
\end{lemma}
\begin{proof}
Per definition, we know
\begin{equation*} \jacobiext*\rnu = {-}\partial@{\H[\atime](\M[\atime_0])} = {-}\outmomden*(\nu) + \div\,\outzFund_\nu - \trtr\zFund{\outzFund}, \labeleq{jacobiext_rnu}\end{equation*}
where we suppressed superindex $\atime$ and where the artificial quantities are defined by $2\,\outzFund*:=\AdSoutg*-\outg*$, $\outmomden*:=\outdiv(\outtr\,\outzFund\,\outg*-\outzFund*)$, and $\outzFund_\nu:=\outzFund*(\nu,\cdot)$,\pagebreak[1] see for example \cite[Prop.~3.7]{nerz2013timeevolutionofCMC} for $\uniM:=\interval*0*1\times\outM$ and $\unig:={-}\d\atime^2+\outg*$.\footnote{Actually, $\outzFund[\AdSa]*$ and $\outmomden[\AdSa]*$ are the second fundamental form and the momentum density of $\{\atime_0\}\times\outM\hookrightarrow(\uniM,\unig)$.}\pagebreak[1] In particular, we know
$\vert\hspace{.05em}\jacobiext*\rnu\vert \le C\,\exp({-}\decay\,\rad)$. Thus, Proposition~\ref{Stability} implies
\[ \Vert\deform\rnu\Vert_{\Wkp^{2,p}(\M)} \le C\,\exp((\frac2p-\frac12-\outve)\,\Hradius) \]
and therefore using Proposition~\ref{Stability} again, we get (in the notation of Definition~\ref{Canonical_partition})
\[ \vert\boost{\rnu} - \frac{\sinh^3(\Hradius)}{6\,\HmHaw}\,\sum_{i=1}^3 \eflap_i \int\eflap_i\,\jacobiext*\rnu\d\mug \vert
	\le C\,\exp((2-\decay)\,\Hradius).\]
Furthermore, Theorem~\ref{Regularity_Theorem} implies
\begin{equation*}\label{lapse_translating_part}
 \Vert\deform{\nu_i}\Vert_{\Wkp^{1,p}(\M)} \le C\,\exp((2-\decay+\frac2p)\,\Hradius),\qquad
	 \vert\Vert\nu_\oi\Vert_{\Lp^2(\M)}^2-\frac{\volume{\M}}{3}\vert \le C\,\exp((4-\decay)\,\Hradius).
\end{equation*}
Thus, we can replace $\{\eflap_i\}_{i=1}^3$ by $\{\frac{\sqrt 3}{2\,\sqrt\pi\,\sinh(\Hradius)}\,\nu_\oi\}_{\oi=1}^3$ and get
\[ \vert\boost{\rnu} - \frac{\delta^{\oi\oj}\;\nu_\oj}{8\,\pi\,\mass^0} \int\sinh(\vert\outx'\vert)\,\nu_\oi\;\jacobiext*\rnu\d\mug \vert
	\le C\,\exp({-}\outve\,\Hradius), \]
where $\outx'=\hiso\circ\outx$ are the new coordinates defined using the $\hyperbolicspace$-isometry of Theorem~\ref{Regularity_Theorem}. Thus, Definition~\ref{Definition_Mass}, Theorem~\ref{Regularity_Theorem}, and equation~\eqref{jacobiext_rnu}, imply
\[ \vert\boost{\rnu}\vert \le \vert\mass^0(\outx')\vert^{-1}\,\vert(\mass^\oi(\outx'))_{\oi=1}^3\vert_{\R^3} + C\,\exp({-}\outve\,\Hradius) \le 2, \]
where we used that $\mass$ is timelike.

If $\outx'$ are balanced coordinates, then we see that \eqref{lapse_translating_part} is (in highest order) identical to the definition of the center of mass, when we keep \eqref{jacobiext_rnu} and Theorem~\ref{Regularity_Theorem} in mind. Therefore, $\vert\boost\rnu\vert\le C\,\exp({-}\outve\Hradius)$.
\end{proof}

\begin{lemma}[\texorpdfstring{$\intervalJ=\intervalI$}{J=I}]\label{I=J}
There exist constants $\Hradius_0=\Cof{\Hradius_0}[\outve][\mass][\eta][\zeta][\c]$, ${\c}'=\Cof{{\c}'}[\outve][\mass][\eta][\zeta][\c]$, $\eta_0\in\interval*01$ such that $\intervalJ=\intervalI$ if $\Hradius>\Hradius_0$, $\c\ge{\c}'$, and $\eta=\eta_0$, \ie $\M[\atime]\in\regsphere{\c'}$ for every $\atime\in\intervalI$ if $\sigma>\sigma_0$.
\end{lemma}
\begin{proof}
Fix $\atime_0\in\intervalJ$. By Prop.~\ref{Stability}, $\M[\atime]$ has ${-}C\,\sinh(\Hradius)^{-3}$-con\-trolled instability if $\Hradius$ is sufficiently large. Now due to the continuity of $\Phi$ explained above, we can assume that there is a neighborhood of $\atime_0$ in $\intervalI$ such that $\M[\atime]\in\regsphere{2\,\c}$ and $\M[\atime]$ satisfies~\ref{IntervalJ-regularity}--\ref{IntervalJ-instability} for $2\,\c$ instead of $\c$ and $\eta=\frac12$. Thus, we only have to control the \lq$\atime$-derivative\rq\ of $\c[\atime]$ and ${}^{\atime\!}\zeta$ for surfaces $\M[\atime]\in\regsphere{\c[\atime]}$ to prove the claim. Fix an artificial time $\atime_0\in\intervalJ$ and let $C$ denote any constant depending on $\outve$, $\mass$, ${}^{\atime\!}\zeta$, and $\c[\atime]$ for $\atime\relax<\atime_0$.\smallskip

By Lemma~\ref{First_estimates_on_lapse}, we know
\[ \min_{\M<0>}\rad - 3\atime
		\le \min_{\M<\atime>} \rad
		\le \max_{\M<\atime>} \rad
		\le \max_{\M<0>} \rad + 3\atime, \qquad
	\vert\vphantom{\Big|}\volume{\M<\atime>}-\volume{\M<0>}\vert \le C\,\exp((2-\outve)\Hradius) \]
and therefore $\M<\atime>$ satisfies \ref{IntervalJ-regularity} and \ref{IntervalJ-instability} for $\c_*=\c+C\,\atime$ and \emph{every} $\zeta_*>\zeta$ instead of $\c$ and $\zeta$, respectively. Note that $\zeta_*$ and $C$ are chosen independently of $\atime$. Therefore, $\intervalJ$ is open in $\intervalI$ for fixed constants $\zeta_*$, $\c_*$, $\eta=\frac12$. As we already know that $\intervalJ$ is also closed in $\intervalI$ and that $\intervalI$ is connected (as it is an interval), this proves the claim.\qedhere
\end{proof}
Now, we can finally prove that $\intervalI=\interval*0*1$. In particular, there exists a surface $\M<\Hradius>$ with constant mean curvature $\H<\Hradius>\equiv{-}2\,\frac{\cosh(\Hradius)}{\sinh(\Hradius)}$ with respect to $\outg*$.
\begin{lemma}[\texorpdfstring{$\intervalI=\interval*0*1$}{I=[0,1]}]\label{I=[0;1]}
There is a constant $\Hradius_0=\Cof{\Hradius_0}[|\mass|][\decay][\scdecay][\oc]$ such that $\intervalI=\interval*0*1$ if $\Hradius>\Hradius_0$.
\end{lemma}
\begin{proof}
Analogous to the one of \cite[Lemma~3.7]{nerz2015CMCfoliation}.
\end{proof}
As we use the uniqueness and regularity of the CMC-leaves in order to prove that they foliate $\outM$, let us first prove the uniqueness and regularity of these surfaces.
\begin{proof}[Proof of Theorem~\ref{Uniqueness_Theorem_full}]
First, we note that by our assumptions on $\zeta$, we can apply Theorem~\ref{Regularity_Theorem} on $\M$ with $\decay'=\frac{\decay}{1+\zeta}>\frac52$. Therefore, our assumptions on $\zeta$ imply $\vert\outric\vert_{\houtg}\le C\,\exp({-}(\frac52+\delta)\Hradius)$ and $\vert\outsc\vert\le C\,\exp({-}(3+\delta)\Hradius)$ on $\M$ for some $\delta\in\interval0*\outve$. In particular, we can apply Proposition~\ref{Stability} on $\M$.

Thus, we can repeat everything done in the current section, but replace the assumptions \ref{IntervalI_IT} and \ref{IntervalI_ID} with \lq$\intervalI\subseteq\interval*0*1$ is a interval with $1\in\intervalI$\rq\ and \lq$\M[1]=\M$ for the given surface $\M$\rq, respectively.\footnote{Here, Lemma~\ref{I=J} a~priori is only true for $\eta_0\in\interval*0{\frac{\eta+\eta^*}2}$ instead of $\eta_0\in\interval*01$, where $\eta^*:=\min\{\frac{2\decay-5}{2\decay},\frac{\scdecay-3}3\}$.}
In particular, we get that there exists a $\Ck^1$-map $\Phi:\interval*0*1\times\sphere\to\outM$ such that $\Phi(\atime,\sphere)$ is a CMC-surface with respect to $\outg[\atime]=\atime\outg+(1-\atime)\AdSoutg$. Now we note that $\Phi(0,\sphere)$ is uniquely determined by $\H(\Phi(0,\sphere))\equiv\H(\M)$ as $\outg[0]=\AdSoutg$, see \cite{brendle2013constant}. By Lemma~\ref{I_is_a_neighborhood_of_J}, this implies that $\M$ is uniquely determined by $\H(\M)$---at least within the class given in Theorem~\ref{Uniqueness_Theorem_full}. This proves the claim.
\end{proof}%
\begin{proof}[Proof of Theorem~\ref{Regularity_Theorem_full}]
Without loss of generality, we can assume that $\outx$ are balanced coordinates, \ie $\mass=(\mass*,0,0,0)$. In particular, the second case of Lemma~\ref{First_estimates_on_lapse} is valid at the artificial time $\tau=0$. Arguing as in Lemma~\ref{I=J}, we see that the same is true for all $\tau\in\interval*0*1$. Thus, Theorem~\ref{Regularity_Theorem} proves the claim, where we get the additionally estimates on $\graphf<\Hradius>'$ by integrating $\rnu<\atime>$ along~$\atime$, where we recall that we have chosen coordinates with $\outcenterz=0$.
\end{proof}

\begin{proof}[Proof of Theorem~\ref{Existence_Theorem_full}]
By Lemma~\ref{I=[0;1]}, there is a constant $\Hradius_0=\Cof{\Hradius_0}[\outve][\mass][\eta][\zeta][\c]$ and a map $\Phi : \interval*0*1\times\interval{\Hradius_0}\infty\times\sphere\to\outM$ such that $\M[\atime]<\Hradius>:=\Phi(\atime,\Hradius,\sphere)$ has constant mean curvature $\H[\atime]<\Hradius>\equiv{-}2\,\frac{\cosh(\Hradius)}{\sinh(\Hradius)}$ with respect to $\outg[\atime]*$. Furthermore, there is a constant $\c'=\Cof{\c'}[\outve][\mass][\eta][\zeta][\c]$ such that $\M[\atime]<\sigma>\in\regsphere{\c'}$ for every $\atime\in\interval*0*1$ and $\sigma>\sigma_0$ due to Lemma~\ref{I=J}. In particular, the stability operator is invertible and an argument as in Lemma~\ref{I_is_a_neighborhood_of_J} ensures that we can choose $\Phi:\interval*0*1\times\interval{\Hradius_0}\infty\to\Wkp^{2,p}(\sphere;\R^3)$ to be continuously differentiable, when we keep the uniqueness (Theorem~\ref{Uniqueness_Theorem_full}) in mind.

The only claim left to prove is the foliation property of $\Phi[\atime]:=\Phi(\atime,\cdot,\cdot)$. Let $\rnu<\Hradius>:=\outg*(\spartial[\Hradius]@\Phi,\nu)$ denote the lapse function in $\Hradius$-direction. In particular, the foliation property holds if $\Vert \rnu<\Hradius>-1\Vert_{\Lp^\infty(\M)}\to 0$ for $\Hradius\to\infty$. As in the proof of Lemma~\ref{I=J}, we know that
\[ \jacobiext*(\rnu<\Hradius>-1) = \partial[\sigma]@{\H[\atime]<\Hradius>} - \outric*(\nu<\Hradius>,\nu<\Hradius>)-\vert\zFund<\Hradius>\vert_{\g<\Hradius>*}^2 = {-} \outric*(\nu<\Hradius>,\nu<\Hradius>) - 2 - \vert\zFundtrf<\Hradius>\vert_{\g<\Hradius>*}^2. \]
By Theorem~\ref{Regularity_Theorem}, this implies $\vert\jacobiext*(\rnu-1) + \outric*(\nu<\Hradius>,\nu<\Hradius>) + 2 \vert \le C\,\exp(-2\,\decay\,\Hradius)$. In particular, we have $\vert\jacobiext*(\rnu<\Hradius>-1)\vert\le C\,\exp({-}\decay\Hradius)$ implying $\vert\deform\rnu\vert\le C\,\exp((2-\decay)\Hradius)$. And therefore---again as in Lemma~\ref{I=J}---, we get
\[ \vert \rnu - 1 - \frac{\sinh(\Hradius)}{4\pi\,\mass*}\sum_{\oi=1}^3\int\widetilde G(\nu<\Hradius>,\partial*_r)\,\frac{\outx^\oi}\rad\d\mug\;\nu_\oi \vert \le C\,\exp({-}\outve\,\Hradius),  \]
where $\widetilde G:=\outric-(\frac12\outsc-1)\outg$. By Theorem~\ref{Regularity_Theorem_full}, this implies
\[ \vert \rnu-1\vert
		= \vert\rnu-1-\frac{\mass^\oi}{\mass*}\,\nu_\oi\vert
		\le \vert \rnu - 1 - \frac{\sinh(\Hradius)}{4\pi\,\HmHaw}\int\widetilde G(\nu<\Hradius>,\partial*_r)\,\frac{\outx^\oi}\rad\d\mug\,\nu_\oi \vert + \frac C{\exp(\outve\Hradius)}
		\le \frac C{\exp(\outve\Hradius)}, \]
where we used that we have chosen balanced coordinates and Herzlich's Ricci version of the mass vector, see \cite{herzlich2015computing} and Remark~ \ref{remarks_on_mass}.\ref{Ricci_version_of_mass}. As explained above, this proves the claim.
\end{proof}

\section{Stability of the CMC-foliation under perturbation of the metric}\label{Section_Stability}
Furthermore, we get the following stability result with respect to perturbations of the metric for the CMC-foliation.\footnote{Note that the corresponding result is true for the asymptotically flat setting. This can easily be seen by analyzing the authors proof of the existence of the CMC-foliation, \cite[Thm~3.1]{nerz2015CMCfoliation}.}\nopagebreak
\begin{theorem}\label{Stability_Theorem_full}
Let $(\outM[i],\outg[i])$ be two $\Ck^2_{\decay,\scdecay}$-asymptotically hyperbolic Riemannian manifolds with decay rates $\decay\ge\frac52+\outve>\frac52$ and $\scdecay\ge3+\outve$ and denote by $\outx[i]$ the corresponding balanced coordinates. If $\vartheta>0$ is such that
\[ \vert\pushforward{\outx[1]}{\outg[1]} - \pushforward{\outx[2]}{\outg[2]}\vert_{\houtg} + \vert\houtlevi*(\pushforward{\outx[1]}{\outg[1]} - \pushforward{\outx[2]}{\outg[2]})\vert_{\houtg} \le \vartheta\,\exp({-}(2+\outve)), \]
then there exist a constant $C=\Cof[\decay][\scdecay][\outc_i]$ and a family of functions $\{\graphf<\Hradius>\}_{\Hradius}$ with $\graphf<\Hradius>\in\Wkp^{2,p}(\M<\Hradius>[1])$ such that
\[ \outx[1](\graph\graphf<\Hradius>)=\outx[2](\M<\Hradius>[2]), \qquad
		\Vert\graphf<\Hradius>\Vert_{\Wkp^{2,p}(\M<\Hradius>[1])} \le C\,\vartheta\,\exp((\frac2p-\outve)\Hradius) \qquad\qquad\forall\,\Hradius>\Hradius_0, \]
where $\M<\Hradius>[i]$ denotes the surfaces constructed in Theorem~\ref{Existence_Theorem_full} with mean curvature $\H<\Hradius>[i]\equiv{-}2\:\frac{\cosh(\Hradius)}{\sinh(\Hradius)}$ with respect to $\outg[i]$. The functions $\graphf<\Hradius>$ depend continuously on $\Hradius$, \ie $\graphf:\interval{\Hradius_0}\infty\to\Wkp^{2,p}(\sphere):\Hradius\to\graphf<\Hradius>\circ\outPhi[1](\Hradius,\,{\cdot}\,)$ is continuous, where $\outPhi[1]$ is as in Theorem~\ref{Existence_Theorem_full}.
\end{theorem}
Note that this continuity result for the CMC-foliation corresponds directly to the continuity result for the mass of the metric proven by Dahl-Sakovich \cite{dahl2015density}.
\begin{proof}[Proof of Theorem~\ref{Stability_Theorem_full}]
We identify $\pushforward{\outx[i]}\outg[i]$ with $\outg[i]$. By the same arguments as in Section~\ref{proofmaintheorem}, we get that for every $\Hradius>0$, there exists a $\Ck^1$ map $\Phi<\Hradius>:\interval*0*1\times\sphere\to\R^3$ such that $\Phi<\Hradius>(\atime,\sphere)$ has constant mean curvature $\H\equiv{-}2\:\frac{\cosh(\Hradius)}{\sinh(\Hradius)}$ with respect to $\outg[\atime]$. Now, the estimates in Lemma~\ref{First_estimates_on_lapse} prove 
\[ \Vert\outg(\partial*_{\atime}\Phi<\Hradius>,\nu<\atime>)\Vert_{\Wkp^{2,p}(\Phi(\atime,\sphere))} \le C\,\exp((\frac2p-\outve)\Hradius)\qquad\forall\,\atime\in\interval*0*1.\qedhere \]
\end{proof}
\begin{remark}
Note that we need only that the metrics are (asymptotically) equal up to the first derivative. This is due to the fact that the lapse function, \ie \eqref{jacobiext_rnu}, depends only on the first derivatives of $\outg[1]-\outg[2]$. However, we still needed that the metrics are asymptotically hyperbolic up to the second derivative to ensure that the surfaces exist.
\end{remark}
\begin{corollary}\label{Stability_Corollary_full}
Let $(\outM[n],\outg[n])$, $(\outM,\outg)$ be $\Ck^2_{\decay,\scdecay}$-asymptotically hyperbolic Riemannian manifolds with decay rates $\decay>\frac52$ and $\scdecay>3$ and uniformly bounded constants $\outc[n]$. Denote by $\outx[n]$ and $\outx$ the corresponding balanced coordinates. If
\[ \limsup_{\rad\to\infty}\exp(2\rad)\,(\vert\pushforward{\outx[n]}{\outg[n]} - \pushforward{\outx}{\outg}\vert_{\houtg} + \vert\houtlevi*(\pushforward{\outx[n]}{\outg[n]} - \pushforward{\outx}{\outg})\vert_{\houtg}) \xrightarrow{n\to\infty} 0, \]
then $\outx[n](\M<\Hradius>[n])$ converge in the $\Wkp^{k,p}$-sense to $\outx(\M<\Hradius>)$ and this convergence is uniformly in $\Hradius$, where $\M<\Hradius>[n]$ and $\M<\Hradius>$ denote the surface constructed in Theorem~\ref{Existence_Theorem_full} with mean curvature $\H<\Hradius>[n]\equiv{-}2\:\frac{\cosh(\Hradius)}{\sinh(\Hradius)}$ and $\H<\Hradius>\equiv{-}2\:\frac{\cosh(\Hradius)}{\sinh(\Hradius)}$ with respect to $\outg[n]$ and $\outg$, respectively.
\end{corollary}
\begin{proof}
Exactly as the proof of Theorem~\ref{Stability_Theorem_full} for $\outg[n]$ and $\outg$.
\end{proof}

\begin{remark}[A comment on the evolution in time]\label{EvolutionInTime}
In \cite{nerz2013timeevolutionofCMC}, the author proved the following: If a family $(\outM,\outg[t],\outzFund[t])$ of \Ckae^2 initial data sets evolves in time according to the Einstein equations (with lapse function $\ralpha\approx1$ and shift vector $\vec{\bar\beta}\approx\vec0$), then their CMC-foliations evolve in time by translating by the quotient of the linear momentum $\impuls$ and the total mass $\mass$. In particular, the Euclidean coordinate center $\centerz[t]<\Hradius>$ (with respect to a suitable \Ckae coordinate system) of the CMC-leaf $\M[t]<\Hradius>$ of fixed mean curvature $\H[t]<\Hradius>\equiv{-}\frac2\Hradius$ with respect to $\outg[t]$ satisfies $\partial*_t{(\mass\'\centerz[t]<\Hradius>)}=\impuls+\mathcal O(\Hradius^{{-}\outve})$ as motivated by their Newtonian counterparts. By the time it was already proven, see \cite{szabados2006poincare}, that the same is true at infinity, \ie if the total center of mass $\outcenterz[t]$ is well-defined\footnote{We recall that this is a non-trivial assumption in the asymptotically Euclidean setting.}, then it satisfies $\partial*_t{(\mass*\,\outcenterz[t])}=\impuls$.

In \cite{cederbaum2015center}, Cederbaum-Cortier-Sakovich proved the analogous result (at infinity) for asymptotically hyperbolic manifolds: for a given family $(\outM,\outg[t],\outzFund[t])$ of \Ckah^2 initial data sets evolving in time accordinate to the Einstein equations (with lapse function $\ralpha=\cosh(\rad)+\Oof(\exp({-}(\frac12+\outve)\rad))$ and shift vector $\vec{\bar\beta}=\Oof(\exp({-}(3+\outve)\rad))$) the hyperbolic center of mass $\outcenterz[t]$ of $(\outM,\outg[t])$ satisfies \lq$\partial*_t{(\mass[t]\outcenterz[t])}=\impuls$\rq,\footnote{More precisely, they proved $\partial*_t[I(\mass[t]*\,\outcenterz[t])]=\impuls$ for the natural embedding $I:\hyperbolicspace\to\R^{3,1}$, see Definition~\ref{center}} where $\impuls$ denotes the linear momentum defined in the cited article.\footnote{Actually, the quantity used in \cite{cederbaum2015center} as linear momentum was previously used in \cite{chrusciel2006rigid} but interpreted differently. We refer to \cite[Sect.~4\&5]{cederbaum2015center} for the details.} . Put differently, we can say that the linear momentum uniquely determines boosts $\hiso[t]:\hyperbolicspace\to\hyperbolicspace$ such that that the center of mass of $(\outM,\outg[t])$ with respect to $\hiso[t]\circ\outx$ is time-independent. Applying Theorem~\ref{Stability_Theorem_full}, this implies that the evolution of the CMC-leaf $\M[t]<\Hradius>$ of $(\outM,\outg[t])$ is (up to an error of order $\exp({-}\outve\Hradius)$) given by this boost, \ie by the linear momentum.
\end{remark}

\appendix
\section{A hyperbolic \texorpdfstring{$\Wkp^{2,p}$}{W2p}-regularity}\label{Section-W2pSurfaceRegularity}
We recall that the second fundamental form of a hypersurface in the Euclidean (or hyperbolic) space can interpreted as the derivative of the Gau\ss\ map. Therefore, it is quite obvious that a sufficient control of the second fundamental form gives a control of the shape of the surface. In \cite{DeLellisMueller_OptimalRigidityEstimates}, De~Lellis-M\"uller proved exactly this in the Euclidean space: 
\begin{theorem}[{\cite{DeLellisMueller_OptimalRigidityEstimates}}]
If $\M\hookrightarrow\R^3$ is a smooth, compact, connected hypersurface without boundary, then
\[ \Vert \eukzFund - \frac1r\,\eukg\Vert_{\Lp^2(\M)} \le c_U\,\Vert\eukzFundtrf\Vert_{\Lp^2(\M)}, \]
where $r:=\eukAradius:=\sqrt{(4\pi)^{-1}\,\eukvolume{\M}}$ and where $c_U$ is an universal constant. If additionally $\Vert\eukzFundtrf\Vert_{\Lp^2(\M)}\le8\pi$, then there exists a conformal parametrization $\psi:\euksphere_r(p)\to\M$ such that
\[ \Vert\psi-\id\Vert_{\Hk^2(\euksphere_r(p))} \le c_U\eukvolume{\M}\Vert\eukzFundtrf\Vert_{\Lp^2(\M)}, \]
where $p\in\R^3$ is some point and where $\id$ is the natural embedding of $\euksphere_r(p)\hookrightarrow\R^3$.
\end{theorem}
By giving explicit counterexamples, De~Lellis-M\"uller furthermore proved that this theorem is optimal, \ie their assumptions cannot be weakened. When we assume stricter assumption on $\zFundtrf$ (\eg pointwise smallness), then it is relatively easy to use De~Lellis-M\"uller's result to conclude that $\M$ is even an Euclidean graph over the Euclidean sphere $\euksphere_r(p)$.\smallskip

Neves-Tian used the Poincar\'e ball model $\outy:\hyperbolicspace\to\eukball^3_1(0)$ and the conformal invariance of $\Vert\zFundtrf\Vert_{\Lp^2(\M)}$ to apply De~Lellis-M\"uller's theorem for a surface $\M$ with small \emph{hyperbolic} tracefree second fundamental form $\hzFundtrf$, \ie to apply it in the hyperbolic setting. However, they needed to additionally prove an estimate on the \emph{Euclidean} (with respect to the Poincar\'e model) surface area $\eukvolume{\outy(\M)}$ of $\M$ where they had to use additional assumptions on $\M$ (besides smallness of $\hzFundtrf$). As mentioned above, they concluded that $\outy(\M)$ is in fact a \emph{Euclidean} graph over an \emph{Euclidean} sphere and got strict estimates on the graph function---both by using pointwise smallness of $\hzFundtrf$. But to conclude that it is also a \emph{hyperbolic} graph over an \emph{hyperbolic} sphere and to get estimates on the \emph{hyperbolic} graph function, they had again to use additional assumptions on $\M$ being very strict assumption on the minimal $\rradius$ and maximal $\Rradius$ (hyperbolic) geodesic distance to the origin. More precisely, $\Rradius-\rradius=\Oof(\rradius^0)$ was assumed to be very small (compared to $\rradius\gg1$).

In this section, we choose a different approach to prove that a given hypersurface is a hyperbolic graph over a hyperbolic sphere, where we directly use the strong assumptions satisfied by the surfaces $\M$ which we study: they are large $\volume{\M}\gg1$, have (pointwise) almost constant mean curvature $\hH(\M)\approx \hH(\hsphere_\Hradius(0))$, are (pointwise) almost umbilic $\hzFundtrf\approx0$, and satisfy the estimate $\frac{3-\eta'}3\Rradius\le\Hradius\le(1+\eta')\rradius$ on the minimal $\rradius$ and maximal $\rradius$ (hyperbolic) geodesic distance to the origin as well as the mean curvature radius $\Hradius$ defined using the mean curvature, where $\eta'\in\interval*01$. The latter can be expressed without using coordinates by the assumption $\M\subseteq \hball^2_{(1+\zeta)\Hradius}(p_0)\setminus \hball^2_{(1-\zeta)\Hradius}(p_0)\subseteq\hyperbolicspace$, where $\zeta\in\interval*0{\frac12}$ is a constant and $\hball^2_r(p_0)$ denotes the hyperbolic geodesic ball of radius $r$ around some point $p_0$.

Obviously, this hyperbolic result is by far not as strong as the result by De~Lellis-M\"uller as we only look at large spheres, assume pointwise inequalities as well as additional inequalities on $\H$ and on the minimal and maximal distance to some point. In particular in contrast to De Lellis-M\"uller's result, the theorem presented here is \emph{not} optimal in the sense that it should also hold if we remove some of the assumptions. 

\begin{theorem}\label{W2pSurfaceRegularity}
For all constants $\outve\in\interval0*1$, $\c>0$, $\zeta\in\interval*0{\frac12}$, and $p\in\interval1\infty$, there exist constants $\Hradius_0=\Cof{\Hradius_0}[\outve][\zeta][\c][p]$ and $C=\Cof[\outve][\zeta][\c][p]$ with the following property:

Let $\Hradius>\Hradius_0$ be a constant, $\M\hookrightarrow\R^3$ be a closed hypersurface in the three-dimen\-sional hyperbolic space, and $p_0\in\R^3$ be a point in the inside of $\M$. If
\begin{equation*}
 \Vert\frac{\hzFund}2 + \frac{\cosh(\Hradius)}{\sinh(\Hradius)}\hg\Vert_{\Lp^\infty(\M)}
	\le \c\,\exp({-}(1+\outve)\Hradius), \ \;
 \Vert\frac\hH2+ \frac{\cosh(\Hradius)}{\sinh(\Hradius)}\Vert_{\Lp^p(\M)}	\le C\,\hvolume{\M}^{\frac1p}\exp({-}(2+\outve)\Hradius),
\end{equation*}
and $\M\subseteq\hball^2_{(1+\zeta)\Hradius}(p_0)\setminus\hball^2_{(1-\zeta)\Hradius}(p_0)$, then there exists a point $p_1$ and a function $\graphf\in\Wkp^{2,p}(\sphere^{\hyperbolich}_\Hradius(p_1))$ with $\M=\hgraph\graphf$ and
\[ \Vert\graphf\Vert_{\Wkp^{2,p}(\sphere^{\hyperbolich}_\Hradius(p_1),\sphg)}
		+ \exp((2-\frac2p)\Hradius)\,\Vert\hzFundtrf\Vert_{\Lp^p(\M)}
		+ \exp(2\Hradius)\,\Vert\vphantom{\big|} \houtlevi\,\vphantom{d}^\hyperbolich\hspace{-.08em}\bar d_{p_1}-\hnu\Vert_{\Lp^\infty(\M)}^2 
	\le C\,\exp({-}\outve\Hradius), \]
where $\hnu$ denotes the hyperbolic outer unit normal of $\M$ and where $\sphg=\sinh(\Hradius)^{-2}\hg$ denotes the rescaled metric on $\sphere_\Hradius(p_1)$.
\end{theorem}
Here, $\hball^2_r(p)$ and $\hsphere_r(p)$ denote the $\houtg*$-geodesic ball and the $\houtg*$-geodesic sphere of radius $r$ around $p\in\hyperbolicspace$, respectively. Furthermore, the $\houtg*$-geodesic graph of a function $\graphf\in\Ck^0(\sphere_r(p))$ is defined by $\hgraph\graphf := \lbrace\vphantom{e}^{\hyperbolich}\hspace{-.05em}\text{\normalfont exp}_p(\graphf(p)\,\hnu<r>_p) : p\in\hsphere_r(p)\rbrace$, where $\hnu<r>$ is the outer unit normal of $\hsphere_r(p)$ and $\vphantom{e}^{\hyperbolich}\hspace{-.05em}\text{\normalfont exp}$ denotes the exponential map---both with respect to $\houtg*$. Note that the assumption on $\zFundtrf$ implies that the assumption on $\hH$ is also satisfied pointwise, \ie for $p=\infty$, but only for the reduce the decay rate, \ie for $C\,\exp({-}(1+\outve))$ instead of $C\,\exp({-}(2+\outve))$.

\begin{remark}[Characterizing the isometry]\label{CharacterizingTheIsometry}
A direct analyze of the proof shows that the hyperbolic distance of $p_0$ and $p_1$ is at most of order $\exp({-}2\Hradius)\,\max_i\vert\int_{\Omega}\frac{\outx^i}{\rad}\d\houtmug\vert$. Furthermore, $p_1$ can be chosen such that if $\hiso:\R^3\to\R^3$ is an isometry of the hyperbolic space with $\hiso(0)=p_1$, then
\begin{equation*}\labeleq{pseudo_center_prenote}
 \frac12\int_{\sphere}({\vphantom{\Big|}}{\sinh}(\graphf)\cosh(\graphf)-\graphf)p^i\d\sphmug = \int_\Omega \frac{\hiso^i\circ\outx}{\vert\hiso\circ\outx\vert}\d\houtmug = 0, \]
where $\Omega$ is the interior of $\M$.
\end{remark}
\begin{remark}[On the proof]
The proof of this theorem contains three steps. The first step is 
to prove that $\M$ is a graph over the coordinate sphere and to obtain first (weak) inequalities for the graph function. This step is a modification of the author's previous proof of a similar result in the Euclidean setting, \cite[Cor.~E.1]{nerz2015CMCfoliation}. As several details have to be changed, we nevertheless demonstrate this step in full detail.

In the second step, we prove that the graph function $\graphf<\Hradius>$ of $\M<\Hradius>$ satisfies $\sphlaplace(\graphf<\Hradius>-\Hradius)\approx 1-\exp(2(\graphf<\Hradius>-\Hradius))$ (in $\Lp^p$). This means that the metric $\exp(2(\graphf<\Hradius>-\Hradius))\,\sphg*$ has ($\Lp^p$-)approx\-imately constant Gau\ss\ curvature. Note that an analogous result was proven by Neves-Tian in \cite[Thm~6.1]{NevesTianExistenceCMC_I} in their setting.
It is well-known that the space of solutions to
\begin{equation*} \tag{\ref{const-sc-eq}} \sphlaplace f=1-\exp(2f) \end{equation*}
is non-compact and therefore we cannot get the needed estimates on $\solutionf<\Hradius>\approx\graphf<\Hradius>-\Hradius$ solely using this partial differential equation. Neves-Tian solved that problem by proving that the functions $\solutionf<\Hradius>\approx\outx-\Hradius$ on the CMC-surfaces $\M<\Hradius>$ can not concentrate at one point if we choose balanced coordinates, \ie they showed
\[ \limsup_{\Hradius\to\infty}\int_{\sphball_r^2(p)}\exp({-}2\,\solutionf<\Hradius>)\d\sphmug =: C(r,p) \xrightarrow{r\to\infty} 0 \qquad\forall\,p\in\sphere, \]
if $\mass=(\mass*,0,0,0)$, where we suppressed the parametrizations $\varphi<\Hradius>:\sphere\to\M<\Hradius>$ and where $\sphball_r^2(p)$ denotes the geodesic ball of radius $r$ in $\sphere$ around $p\in\sphere$. Thus $\solutionf<\Hradius>$ is bounded in $\Hk^2$ independently of $\Hradius$ as every $\Hk^2$-unbounded sequence $\solutionf<\Hradius>$ of solutions of \eqref{const-sc-eq} would \lq concentrate\rq\ at one point, \ie the above $C(r,p)$ would satisfy $C(r,p)\ge\delta>0$ independently of $r$ (for some $p\in\sphere$), see \cite[Thm~1]{chen1991}.
However, Neves-Tian's argument crucially uses the fact that $\outg$ is of the form they assumed it to be, \ie that it satisfies \eqref{Ass_NevesTian_I}---more precisely that $\outric$ is of the form \eqref{Ass_Ricci} with non-vanishing $\mass*$-term.

As we look at the hyperbolic space\footnote{or at spaces asymptotic to the hyperbolic space with insufficient decay rate in order to apply a version of Neves-Tian's proof}, \ie the mass is vanishing, we cannot use Neves-Tian's method to prove that $\exp({-}2\graphf)$ does not concentrate at one point---as this is possible. Therefore, we have to change the argument to conclude an estimate for $f\approx\graphf-\Hradius$. We recall that the metrics $\exp(2f)\,\sphg*$ mentioned above have constant Gau\ss\ curvature. Thus, they are pullbacks of the standard metric of the Euclidean sphere $\sphg$ along conformal diffeomorphisms. Thus, each solution $f$ of \eqref{const-sc-eq} is characterized by a conformal diffeomorphism $\varphi(f)$ of the Euclidean sphere which again is the action at infinity of an isometry $\hiso(\varphi(f))$ of the hyperbolic space, see \cite{chrusciel2003mass}. Thus, we should be able to choose an $\hyperbolicspace$-isometry $\hiso(\varphi(f'))$ such that $\hiso(\M)$ has a graph function $\graphf'$ such that its \lq Gau\ss ian part\rq\ $f'\approx\graphf'-\Hradius$ (solving \eqref{const-sc-eq}) vanishes. As the corresponding conformal diffeomorphism $\varphi(\hiso)$---which we are looking for---is only characterized by the action of the hyperbolic isometry \emph{at infinity}, we do not directly choose the necessary $\hyperbolicspace$-isometry, but instead step by step choose a family of isometries $\{\Phi_t\}_{t\in\interval*0*1}$ of the hyperbolic space such that
\begin{enumerate}[label=(\alph{*}),nosep]
\item all $\Phi_t(\M)$ satisfy the same assumptions as $\M$ (with respect to the \emph{same} center point $p_0$),
\item the Gau\ss ian parts ${}_tf\approx\graphf<t>-\Hradius$ (solving \eqref{const-sc-eq}) of the graph functions $\graphf<t>$ of $\Phi_t(\M)$ gets smaller as $t$ gets larger and vanishes for $t=1$.
\end{enumerate}\smallskip

Here, the author thanks Katharina Radermacher for the idea to interpret the artificial quantity $\int_{\M}\exp(2\graphf)\,\frac{\outx_i}\rad\d\mug$ as \lq pseudo-center\rq\ of the interior of $\M$ which is a crucial idea in order to choose these isometries, see Remark~\ref{Pseudo-center}.
\end{remark}
\begin{proof}[Proof of Theorem~\ref{W2pSurfaceRegularity}]
As all assumptions are geometric ones, we can apply an isometry $\hiso$ of the hyperbolic space such that $\hiso(p_0)=0$, \ie without loss of generality $p_0=0$.\smallskip

\emph{Step 1 {\normalfont(}$\M$ is a Graph{\normalfont):}}
Let $X$ be the tangential projection of the radial direction $\eta=\houtlevi*\rad=\eukoutlevi*\rad$ and $\fX$ be its norm square, \ie$X:=\eta^T:=\eta-\houtg*(\eta,\hnu)\,\hnu$ and $\fX:=\vert X\vert_{\houtg*}^2$. Furthermore, let $\gamma:\R\to\M$ be the integral curve to $X$ through an arbitrary point $p\in\M$---as $\M$ is compact and without boundary, $\solX$ is well-defined on the entire $\R$. We see that $\fXsolX:=\fX\circ\solX$ satisfies
\[ \fXsolX = \houtg*(X\circ\solX,\eta) = (D_X\,\rad)\circ\solX = \partial[t]@{\rad@\solX} =: \rad@\solX', \]
which by the compactness of $\M$ implies
\begin{equation*}
 \int_s^t \fXsolX(u)\d u = \int_s^t \rad@\solX'(u) \d u = d(\vphantom{\big|}\solX(t)) - d(\vphantom{\big|}\solX(s)) \le \max_{\M}\rad - \min_{\M}\rad < \infty. \labeleq{zFund->Graph_Section_step_1}
\end{equation*}
In particular, $\fXsolX(t_n)\to0$ for some sequence $t_n\to{-}\infty$. Furthermore, we see
\begin{align*}
 \hzFund(X,X)
 ={}& \houtg(\houtlevi*\!_X(\eta-\houtg(\eta,\hnu)\hnu),\hnu)
 = 		\houtg(\houtlevi*\!_X\eta,\hnu) - D_X\houtg(\eta,\hnu) \\
 ={}& (\houtHess*\,\rad)(X,\hnu) - D_X\houtg(\eta,\hnu) \labeleq{k_XX__1}
\end{align*}
and 
\begin{equation*}
 {-}\frac12\fXsolX' = \frac12D_X(1-\fX)
 = \frac12D_X(\houtg(\eta,\hnu)^2) = \houtg(\eta,\hnu)\,D_X\houtg(\eta,\hnu). \labeleq{k_XX__2}
\end{equation*}
A direct calculation gives $\houtHess\rad = \frac{\cosh(\rad)}{\sinh(\rad)}(\houtg-\d\rad\otimes\d\rad)$ and combining this with \eqref{k_XX__1} and \eqref{k_XX__2} gives
\begin{align*}
 \houtg(\eta,\hnu)\,\hzFund(X,X)
 ={}& {-}\frac{\cosh(\rad@\solX)}{\sinh(\rad@\solX)} \houtg*(\eta,X)\houtg*(\eta,\hnu)^2 + \frac12\,\fXsolX' \\
 ={}& {-}\frac{\cosh(\rad@\solX)}{\sinh(\rad@\solX)} \fXsolX(1-\fXsolX) + \frac12\,\fXsolX'.
\end{align*}
In $t\in\R$ with $\houtg(\eta,\hnu)_{\solX(t)}\ge\frac12$, this can be written as
\begin{align*}
 \frac{\fXsolX'}2
	={}& \houtg(\eta,\hnu)\,\hzFundtrf(X,X) + \fXsolX(\frac\hH2+\frac{\cosh(\rad@\solX)}{\sinh(\rad@\solX)}+\hH\frac{\houtg(\eta,\hnu)-1}2 - \frac{\cosh(\rad@\solX)}{\sinh(\rad@\solX)} \fXsolX) \\
	\le{}& \vert\hzFundtrf(X,X)\vert
			+ \fXsolX(\frac{\hH}2+\frac{\cosh(\rad@\solX)}{\sinh(\rad@\solX)})
				- \frac{\fXsolX^2}2(\frac{\hH}{1+\houtg(\eta,\hnu)} + 2\frac{\cosh(\rad@\solX)}{\sinh(\rad@\solX)}) \\
	\le{}& (2\,\vert\hzFundtrf\vert + \sinh(\rad@\solX)^{{-}2}\,(\vphantom{\big|}\Hradius-\rad@\solX))\,\fXsolX
			- \frac14\fXsolX^2 \\
	\le{}& (2\,\vert\hzFundtrf\vert+C\exp(-2(1-\zeta)\Hradius)\Hradius)\,\fXsolX - \frac14\fXsolX^2,
\end{align*}
if $\Hradius$ is sufficiently large. As $\fXsolX(t_n)\to0$, we can assume that---for some subsequence---
\[ 1 \ge \houtg(\eta,\hnu)_{\solX(t_n)} \xrightarrow{n\to\infty} 1. \]
In particular, there exists a $s\in\R$ such that $\houtg(\eta,\hnu)_{\solX(s)}\ge\frac12$ and $\fXsolX'(s)\ge0$, where $k\in\N$ is arbitrary. Using the above, we have for such a $s\in\R$
\begin{equation*}\labeleq{graph__control_on_normal__2}
 \fXsolX(s) \le 8\,\vert\hzFundtrf\vert + C\,\exp({-}(1+\outve')\Hradius) \le C\,\exp({-}(1+\outve')\Hradius) \ll\frac14,
\end{equation*}
where $\outve':=\min\{1-2\zeta,\outve\}>0$. This implies $\houtg(\eta,\hnu)_{\solX(s)}>\frac12$. Thus, we know that the set $[\houtg(\eta,\hnu)_{\solX}\ge\frac12]\cap[\fXsolX'\ge0]$ is open and closed within $[\fXsolX'\ge0]$. With $\liminf_{t\to{-}\infty}\fXsolX(t)=0$, this proves
\begin{equation*}\labeleq{graph__control_on_normal}
 \sup_{t\in\R}\fXsolX(t) \le C\,\exp({-}(1+\outve')\Hradius), \qquad
		\houtg(\eta,\hnu)_{\solX}\ge 1 - C\,\exp({-}(1+\outve')\Hradius).
\end{equation*}
As $\solX$ was the integral curve to~$X$ through an arbitrary point~$p$, we get the same inequality for $\sup_{\M}\fX$ instead of $\sup_\R \fXsolX$. In particular, $\M$ is a graph over the concentric sphere. Re-examining \eqref{graph__control_on_normal__2}, we see that we already proved the claimed inequality on $\houtlevi\,\vphantom{d}^\hyperbolich\hspace{-.05em}\bar d_{p_1}-\hnu$ if we proved the other claims of this theorem (as those imply $\zeta=0$). \smallskip\pagebreak[2]

\noindent\emph{Step 2 {\normalfont(}controlling the graph function {\normalfont I):}}
By the above, there is a graph function with $\graphf\in\Ck(\hsphere_\Hradius(0))$ such that $\M=\{\text{exp}_{\Hradius p}(\graphf(\Hradius p)\hnu)=(\graphft(p),p)\,:\,p\in\sphere_1\}$, where $\graphft(p):=\graphf(\Hradius p)+\Hradius\in\Ck(\sphere)$ and where we used spherical coordinates. \pagebreak[1] Now, we prove
\begin{equation*}\labeleq{solutionf} \exists\,\solutionf\in\Ck^\infty(\sphere):\quad\Vert\graphft-\Hradius-\solutionf\Vert_{\Wkp^{2,p}(\sphere)} \le C\,\exp({-}\outve'\Hradius), \qquad\sphlaplace \solutionf = 1 - \exp(2\solutionf).\nopagebreak \end{equation*}\nopagebreak
where again $\outve'=\min\{1-2\zeta,\outve\}$.\pagebreak[2]

Let $(\partial*_r,\partial*_2,\partial*_3)$ denote the standard coordinate frame of spherical coordinates of $\R^3$. In this frame and for $\ii\in\{2,3\}$, we have
\begin{equation*}
 \vert\partial*_{\ii}\graphft\vert^2
 =		\vert\houtg*(\partial*_{\ii}\graphFt,X)\vert^2
 \le \vert\partial*_{\ii}\graphFt\vert_{\houtg*}^2\,\fX
 =		(\vert\partial*_{\ii}\graphft\vert^2 + \sinh(\graphft)^2\sphnormof{\partial*_{\ii}}^2)\fX,
\end{equation*}
where $\graphFt(p):=(\graphft(p).p)$. This implies
\begin{equation*}\labeleq{graph-levih_1}
 \vert\sphlevi*h\vert_{\sphg}^2 \le \frac\fX{1-\fX}\sinh(\graphft)^2 \le 2\fX\sinh(\graphft)^2, \quad
 \vert\graphft-\Hradius\vert \le \max_{\M}\{\rad-\Hradius\} \le \zeta\Hradius
\end{equation*}

Per definition of the mean curvature, we have
\def\sh{\text{sh}}\def\ch{\text{ch}}
\begin{equation*} \left\{\ \begin{aligned}\hspace{3em}&\hspace{-3em}
 \sphlaplace\graphft - (\frac{\ch}{\sh}\vert\sphlevi*\graphft\vert_{\sphg}^2 - \sh^{{-}2}\,\sphHess\,\graphft(\sphlevi*\graphft,\sphlevi*\graphft))(1+\sh^{{-}2}{\vert\sphlevi*\graphft\vert_{\sphg}^2})^{-1} \\
	={}& 2\,\sh\,\ch + \hH\,\sh^2\,(1+\sh^{{-}2}\vert\sphlevi*\graphft\vert_{\sphg}^2)^{\frac12}
	\end{aligned}\right.\labeleq{LaplaceGraph},
\end{equation*}
where $\sh:=\sinh(\graphft)$, $\ch:=\cosh(\graphft)$, and therefore \eqref{graph__control_on_normal} and the first inequality in \eqref{graph-levih_1} imply
\begin{equation*}
 \Vert\sphlaplace\graphft - 1 + \frac{\sinh(\graphft)^2}{\sinh(\Hradius)^2}\Vert_{\Lp^p(\sphere)}
	\le C\,\exp({-}(1+\outve)\Hradius)\,\Vert\sphHess\,\graphft\Vert_{\Lp^p(\sphere)} + C\,\exp({-}\outve'\Hradius)
\end{equation*}
Using the regularity of the Laplace operator and the second inequality in \eqref{graph-levih_1}, we get $\Vert\sphHess\,\graphft\Vert_{\Lp^p(\sphere)}\le C\,\exp({-}\outve'\Hradius)$. Including this in the last approximation, we get
\begin{equation*}\labeleq{LaplaceGraph_approx}
 \Vert\sphlaplace(\graphft-\Hradius) - 1 + \exp(2(\graphft-\Hradius))\Vert_{\Lp^p(\sphere)}
	\le C\,\exp({-}\outve'\Hradius). \labeleq{laplace-graphf}
\end{equation*}
In other words, the scalar curvature $\sc'$ of the metric\footnote{Note that this metric is \emph{not} the metric of $\M$.} $\g':=\exp(2(\graphft-\Hradius))\sphg*$ satisfies $\Vert\sc'-1\Vert_{\Lp^p(\sphere)} \le C\,\exp(-\outve'\Hradius)$. By \cite[Thm~A.1]{nerz2015GeometricCharac}, there is a conformal diffeomorphism $\varphi:\sphere\to\sphere$ such that the conformal function $v\in\Wkp^{2,p}(\sphere)$ with $\pullback\varphi(\g')=\exp(2v)\sphg*$ satisfies $\Vert v\Vert_{\Wkp^{2,p}(\sphere)}\le C\,\exp({-}\outve'\Hradius)$ or in other words \eqref{solutionf} holds.

\noindent\emph{Step 3 {\normalfont(}controlling the graph function {\normalfont II} -- choosing an isometry{\normalfont):}}
Let us first introduce the \lq pseudo-center\rq\ $Z(\Omega)=(Z^i(\Omega))_{i=1}^3$ defined for any compact region $\Omega\subseteq\R^3$ by
\[ Z^i(\Omega) := \int_\Omega \frac{\outx^i}\rad\d\houtmug. \]
Now, let $\Omega\subseteq\R^3$ denote the interior of $\M\subseteq\R^3$. We see
\begin{equation*}
 Z := Z(\Omega)
 = \int_{\sphere}\int_0^{\graphft(p)}\frac{\outx^i(r\,p)}r\sinh(r)^2 \d r\,\d\sphmug
 = \frac12\int_{\sphere}(\sinh(\graphft)\,\cosh(\graphft)-\graphft)p^i \d\sphmug,
\end{equation*}
Using $\vert\sinh(\graphft)\cosh(\graphft)-\frac14\exp(2\graphft)\vert\le\exp({-}2\graphft)$, $\sphlaplace p^i={-}2p^i$, and the controls of $\graphft$, we get
\begin{equation*}
 \vert 4Z - \frac{\exp(2\Hradius)}2\int\exp(2\solutionf)p^i\d\mug\vert
  \le \vert 4Z - \int(\frac{\exp(2\graphft)}2+\sphlaplace\graphft)p^i\d\sphmug\vert + C\,\exp((2-\outve')\Hradius)
	\le C\,\exp((2-\outve')\Hradius),
\end{equation*}
where $\solutionf\in\Ck^\infty(\sphere^2)$ is as in~\eqref{solutionf}.
By Lemma~\ref{lambda_control_by_ux_i_control}, we know
\[ \theta(\text{exp}(\Vert\solutionf\Vert_{\Lp^\infty(\M)})) = (\sum_{i=1}^3(\fint_{\sphere}\exp(2\solutionf)\,p^i \d\sphmug)^2)^{\frac12}, \quad
	\theta(\lambda):=\vert\frac{1+4\ln(\lambda)\lambda^2-\lambda^4}{\lambda^4-2\lambda^2+1}\vert \]
and therefore
\[ \sup_{\M}\vert\vphantom{\big|}\rad-\Hradius\vert
		\le \Vert\solutionf\Vert_{\Lp^\infty(\sphere)} + C\,\exp({-}\outve'\Hradius)
		\le \ln(\theta^{-1}(\frac{8\,\vert Z\vert}{\exp(2\Hradius)} + C\,\exp({-}\outve'\Hradius)))+C\,\exp({-}\outve'\Hradius). \]
Now, let us assume that $\vert Z(\Omega)\vert\neq0$. By Lemma~\ref{pseudo-center}, there exists a smooth one parameter family $\{\hiso_t\}_{t\in\interval*0*1}$ of isometries $\hiso_t:\R^3\to\R^3$ of the hyperbolic space with $\hiso_0=\id$, $\vert Z(\hiso_t(\Omega))\vert$ is strictly monotone decreasing, and $Z(\hiso_1(\Omega))=0$. For arbitrary $\zeta'\in\interval\zeta{\frac12}$, there exists $t_0>0$ such that $\hiso_t(\M)$ satisfies the assumptions of this theorem for every $t\le t_0$ and $p_0=0$ if we replace $\zeta$ by $\zeta'$, because these assumptions depend continuously on $t$ as $\hiso_t$ does so. However, as $\vert Z(\hiso_t(\Omega))\vert$ is strictly smaller than $\vert Z(\Omega)\vert$, the above implies that $\hiso_t(\M)$ satisfies the assumptions of this theorem for $\zeta'\le\zeta$. In particular, we can apply the theorem on every $\hiso_t(\M)$ without changing $\zeta$ along $t$, when we replace $\zeta\Hradius$ by $\zeta\Hradius+C\,t\,\exp({-}\outve\Hradius)$. As $Z(\hiso_1(\Omega))=0$, we can without loss of generality assume that $Z(\Omega)=0$ if we replace $p_0=0$ by $p_1:=\hiso_1^{-1}(0)$.

Thus using the regularity of the Laplace operator, the above proves 
\[ \Vert\vphantom{\big|}\graphf\Vert_{\Wkp^{2,p}(\hsphere_\Hradius(p_1))} \le C\,\exp({-}\outve'\Hradius) \]
for the graph function of $\M$ above $\hsphere_\Hradius(p_1)$ which is the same as the one of $\hiso_1(\M)$ above $\hsphere_\Hradius(0)$.
In particular, $\vert\rad-\Hradius\vert\ll C$ and we can therefore assume $\zeta=0$ implying $\outve'=\outve$. This proves all the claims.
\end{proof}

\begin{lemma}\label{pseudo-center}
Let $\Omega\subseteq\hyperbolicspace$ be any compact region with well-defined \lq pseudo-center\rq\ $Z(\Omega):=(Z^i(\Omega))_{i=1}^3$, where
\[ Z^i(\Omega) := \int_\Omega \frac{\outx^i}\rad\d\houtmug \qquad\forall\,i\in\{1,2,3\}. \]
Either $Z(\Omega)=0$ or there is an one-parameter family $\{\Phi_t:\R^3\to\R^3\}_{t\in\interval*0*1}$ of isometries of the hyperbolic space which depends smoothly on $t\in\interval*0*1$ with $\Phi_0=\id$ such that $t\mapsto \vert Z(\Phi_t(\Omega))\vert^2$ is strictly monotone decreasing and vanishes at $t=1$.
\end{lemma}
\begin{remark}[The pseudo-center]\label{Pseudo-center}
In Theorem~\ref{W2pSurfaceRegularity}, we have to control
\[ \int_{\sphere_\rradius} \frac{\outx^i}{\rad}\,\exp(2\graphf) \d\mug
		\approx 4\int_{\sphere_\rradius}\frac{\outx^i}\rad\,(\sinh(\graphf)\,\cosh(\graphf)-\graphf)\d\sphmug \]
in order to get the estimates on $\graphf$ we aime for. Here, $\graphf$ is the graph function of our surface $\M$. By a simple integration, we know
\[  \frac12\int_{\sphere_\rradius(0)}\frac{\outx^i}\rad\,(\sinh(\graphf)\,\cosh(\graphf)-\graphf)\d\sphmug
		= \int_0^\rradius\int_{\sphere}\frac{\outx^i}\rad\,\sinh(r)^2 \d\mug\d r
		= \int_{\Omega}\frac{\outx^i}\rad\d\outmug \]
for the interior $\Omega$ of $\M$. Up to multiplication with the constant $\volume{\Omega}^{-1}$, this is the mean value of the direction (from the origin) which is quite \emph{similar}---thus, \lq pseudo-\rq---to the mean value $\fint_{\Omega}\outx^i\d\houtmug$ of the coordinates. The latter is often called \emph{coordinate center} of \lq the object $\Omega$\rq (here, a subset of $\R^3$)---although this name is misleading in the hyperbolic setting, see \cite{cederbaum2015center}. However, if it is some kind of center, then it should behave nicely under the change of coordinates. Thus, we have found a way to understand the originally artificial integral $\int\frac{\outx^i}\rad\exp(2\graphf)\d\mug$ giving as a way to get ride of it---by changing coordinates. This is based on the idea to interpret $\int\frac{\outx^i}\rad\exp(2\graphf)\d\mug$ as pseudo-center by Katharina Radermacher.
\end{remark}
\begin{proof}$\,\!$%
\nopagebreak\par\nopagebreak
\emph{Step 1 (Local existence):}
Without loss of generality $Z(\Omega)\neq0$ and by applying a rotation, we can assume $Z=(\vert Z\vert,0,0)$. Let $\varphi:\R^3\to\eukball^3_1(0)$ be the rotation-free isometry between the hyperbolic space $(\R^3,\houtg)$ and its Poincar\'e ball model with $\varphi(0)=0$. In particular,
\[ Z^i(\Omega) = 4\int_{\varphi(\Omega)}\frac{y^i}{\vert y\vert}\,(1-\vert y\vert^2)^{{-}2} \d y = \int_{\varphi(\Omega)}\frac{y^i}{\vert y\vert}\,\d\houtmug \]
and therefore we can suppress $\varphi$ in the following.
Let furthermore $\hiso_t:B_1(0)\to B_1(0)$ be the rotation-free isometry of the Poincar\'e ball with $\hiso_t(0)=t\frac{Z(\Omega)}{\vert Z(\Omega)\vert}=(t,0,0)=te_1$, \ie
\[ \hiso_t(rp) = \frac{t(2rt\,p^1+r^2+1)e_1 + (1-t^2)rp}{r^2\,t^2+2rtp^1+1} \quad\forall\,r\in\interval*0*1,\,t\in\interval-11,\,p\in\sphere_1(0). \]
In particular,
\begin{equation*}\labeleq{del_t__integrand}
 \partial[t]<t=0>(\frac{\hiso_t(rp)^i}{\vert\hiso_t(rp)\vert})
	= \frac{1+r^2}r(\delta^{1i}-p^1\,p^i)
			\qquad\forall\,p\in\sphere,\,r\in\interval01
\end{equation*}
implying
\begin{align*}
 \left.\frac\d{\d t}\right|_{t=0}\vert Z(\hiso_t(\Omega))\vert^2
	={}& \eukoutg(Z(\Omega),\partial[t]<t=0>Z(\hiso_t(\Omega)))
	= 2\vert Z(\Omega)\vert\,\left.\frac\d{\d t}\right|_{t=0} Z^1(\hiso_t(\Omega)) \\
	={}& 2 \vert Z(\Omega)\vert \int_\Omega\frac{1+\rad@\outy^2}{\rad@\outy}(1-\frac{\vert\outy^1\vert^2}{\rad@\outy^2})\d\houtmug.
\end{align*}
This vanishes if and only if $Z(\Omega)=0$ or $\Omega\subseteq[\outy^1=1]$. As $\houtmug([\outy^1=1])=0$, we know $Z(\Omega)=0$ or $\frac\d{\d t}\vert Z(\hiso_t(\Omega))\vert^2>0$.

\emph{Step 2 (global existence):} Now, denote by $\HISO:\eukball^3_1(0)\to\text{iso}(\hyperbolicspace):v\mapsto\HISO_v$ the smooth map from the Euclidean ball to the space of isometries of $\hyperbolicspace$ such that $\HISO_v(0)=v$, \ie for $v\in\eukball_1(0)$
\[ \HISO_v(rp) = \frac{(2r\,\eukoutg(v,p)+r^2+1)v + (1-\vert v\vert_{\eukoutg}^2)rp}{r^2\,\vert v\vert_{\eukoutg}^2+2r\eukoutg(v,p)+1} \qquad\forall\,r\in\interval*0*1,\,p\in\sphere_1(0). \]
Furthermore, let $\alpha:\intervalI\to \eukball^3_1(0)$ be the maximal continuous, piecewise smooth curve such that $\HISO_{\alpha(t)}$ is the isometry with
\[ \partial[t]@{(\HISO_{\alpha(t)}(rp))}
	= C_t\partial[\tau]<\tau=0>(\frac{\tau(1+2r\tau(p\cdot Z_t)+r^2)Z_t+(1-\tau^2)rp}{1+r^2\tau^2+2r\tau(p\cdot Z_t)+1}) \]
and $0\in\intervalI\subseteq\interval*0*1$, where $Z_t:=Z(\HISO_{\alpha(t)}(\M))$ and where $C_t>0$ is some constant such that $\vert Z_t\vert = (1-t)\vert Z(\M)\vert$. By the first step, we can choose $\intervalI$ to be an open subinterval of $\interval*0*1$. Left to prove is $1\in\intervalI$ which is the case if $\intervalI$ is closed.

First, we note that \eqref{del_t__integrand} implies 
\[ \vert\frac{\d(Z_t^i)}{\d t}\vert
		\le C_t\int_{\Omega}\frac{1+\rad@\outy^2}{\rad@\outy}\vert Z_t^i-Z_t^jp_jp^i\vert\d\outmug
		\le C(\M)\,C_t \]
for some constant $C(\M)$ depending on $\M$, but not on $t$. This implies $\vert\nicefrac{\d C_t}{\d t}\vert\le C(\M)$. In particular, $\alpha(t)\to v_*\in\overline{\textstyle\eukball_1^3(0)}$ as $t\to\sup\intervalI$ and equally calculating the $Z_t$, we see that $\alpha$ is in fact smooth on $\intervalI$.

Now, we prove $v_*\in \eukball_1^3(0)$, \ie $\sup\vert\alpha(t)\vert<1$. Assume that this is not true, \ie $\lim_t\alpha(t)=\vert v_*\vert_{\eukoutg}=1$. Let $r>0$ be such that $\Omega\subseteq\hball_r^3(0)$. Per definition of $\HISO$, we know $\HISO_{\alpha(t)}(\Omega)\subseteq\hball_r^3(\alpha(t))$ and directly see that this implies
\[ \sup\lbrace\vert\outy-\alpha(t)\vert_{\eukoutg}\ \middle|\ \outy\in\HISO_{\alpha(t)}(\Omega)\rbrace \xrightarrow{t\to\sup \intervalI} 0. \]
Thus, we get $\volume{\M}>\vert Z(\HISO_{\alpha(t)}(\Omega))\vert\to\volume{\M}$ which contradicts the fact that $\vert Z(\HISO_{\alpha(t)}(\M))\vert$ is decreasing.

Thus, we know that $\gamma(t)$ converges in $\eukball_1^3(0)$ as $t\to\sup\intervalI$ and therefore we can extend $\gamma$ continuously to $\sup\intervalI$. Combining this with the first step, we have proven $\intervalI=\interval*0*1$ which proves the claim.
\end{proof}
\begin{remark}[A less constructive proof]
You can prove this theorem in a less constructive way: Using the same argument as in the step in which we prove $\vert v_*\vert<1$, we see that the map $h:\eukball^3_1(0)\to\eukball^3_1(0):p\mapsto \hvolume{\Omega}^{-1}\,Z(\HISO_p(\Omega))$ can continuously extended to $\bar h:\overline{\textstyle\eukball^3_1(0)}\to\overline{\textstyle\eukball^3_1(0)}$ by defining $\bar h(p)=p$ for every $p\in\bar X:=\overline{\textstyle\eukball^3_1(0)}$ with $\vert p\vert=1$. In particular, $h$ is a continuous map from $\bar X$ to itself with $h|_{\border\bar X}=\id|_{\border\bar X}$ and therefore $h(\bar X)=\bar X$. In particular, there is a continuous curve $\alpha:\interval*0*1\to\eukball^3_1(0)$ such that $\vert h(\alpha(t))\vert=t\,\vert h(0)\vert$. This proves the claim.
\end{remark}

\begin{lemma}\label{lambda_control_by_ux_i_control}
If $u\in\Ck^\infty(\sphere)$ is a solution of $\sphlaplace u=1-\exp(2u)$, then $\text{exp}(\Vert u\Vert_{\Lp^\infty(\M)}) = f^{-1}(K(u))$, where $f^{-1}$ is the inverse function of the diffeomorphism $f:\interval*1\infty\to\interval*01$, and where $f$ and $K$ are defined by
\[ f(\lambda) := \frac{\lambda^4-\ln(\lambda)\lambda^2-1}{\lambda^4-2\lambda^2+1}\ \forall\,\lambda>1, \quad
	 f(1) := 0, \quad
	K(v) = (\sum_{i=1}^3\vert\fint_{\sphere} \exp(2u)\,x_i\d\mug\vert^2)^{\frac12}. \]
\end{lemma}
\begin{proof}
Per definition of the stereographic projection $\psi:\sphere\to\R^2$, we know
\[ (\outx^1\circ\psi^{-1})(y) = \frac{\vert y\vert^2-1}{\vert y\vert^2+1}, \qquad\quad
	 (\outx^{\ii}\circ\psi^{-1})(y) = \frac{2\,y^{\ii}}{\vert y\vert^2+1}\quad \forall\,\ii\in\{2,3\}. \]
We chose $v\in\Ck^\infty(\R^2)$ with $\pullback\psi(\exp(2v)\delta)=\exp(2u)\sphg*$ and know
\[ v(y) = \ln(\frac{2\lambda}{\lambda^2+\vert y-y_0\vert^2}) \]
for some $\lambda>0$ and $y_0\in\R^3$, \cite[Thm~1]{chen1991}. Applying a suitable rotation $\varphi$ of $\sphere$, we can without loss of generality assume $y_0=0$. We note that by applying a inversion at the equator, $\lambda$ is replaced by $\lambda^{-1}$ while $K(v)$ does not change. Therefore, we can assume $\lambda\ge1$. 
In particular,
\begin{equation*}
 \int \outx^{\ii}\,\exp(2v)\d\sphmug
	= \int (\frac{2\lambda}{\lambda^2+\vert y\vert^2})^2\frac{2y^{\ii}}{\vert y\vert^2+1} \d y = 0
\end{equation*}
for $\ii\in\{2,3\}$ and therefore
\begin{align*}
 K(v)^2
	={}& \vert\int \outx^1\,\exp(2v)\d\sphmug\vert^2
	= \vert\int \outx^1\d\mug[\exp(2v)\sphg]\vert^2
	= \left|\int_{\R^2} (\frac{2\lambda}{\lambda^2+\vert y\vert^2})^2\frac{\vert y\vert^2-1}{\vert y\vert^2+1} \d y\right|^2 \\
	={}& \vert 4\pi\,\frac{1+4\ln(\lambda) \lambda^2-\lambda^4}{\lambda^4-2\,\lambda^2+1}\vert^2 = (4\pi f(\lambda))^2.
\end{align*}
A direct analysis shows that $f(\lambda)$ is strictly monotone increasing in $\lambda$ within $\interval*1*\infty$. Furthermore, we see
\[ u = (v - \ln(\frac2{1+|y|^2}))\circ\psi \]
and therefore $\sup\vert u\vert=\ln\lambda$.
\end{proof}

\section{A triviality on the second fundamental form}
The following lemma is well-known, but the author did not find a correct citation for it in the non-minimal setting and therefore included it for the readers convenience.
\begin{lemma}[A trivial generalization of {\cite[(1.28)]{schoen1975curvature}}]\def\vzFundtrf{\left|\zFundtrf\right|_{\hspace{-.05em}\g}}\label{Schoen-Yau75}
Let $(\M,\g*)$ be a hypersurface within a $n{+}1$-dimensional Riemannian manifold $(\outM,\outg*)$. If $\M$ has constant mean curvature, then
\[ \forall\,\delta>0\ \,\exists\,C=\Cof[n][\delta^{-1}]\!:\quad\trtr{\levi\zFundtrf}{\levi\zFundtrf}\,\vzFundtrf^2 \ge \frac{n+2-\delta}n\vzFundtrf^2\vert\levi\vzFundtrf\vert_{\g}^2 - C\vert\trzd{\outric_{\nu}}\zFundtrf\vert_{\g}^2. \]
\end{lemma}
\begin{proof}\def\vzFundtrf{\left|\zFundtrf\right|_{\hspace{-.05em}\g}}
Fix normal coordinates around an arbitrary point $p\in\M$ such that $\zFund(p)$ is diagonalized. Doing the same calculation as in \cite{schoen1975curvature,metzger2004dissertation}, we see
\begin{align*}
 (\trtr{\levi\zFundtrf}{\levi\zFundtrf}-\trtr{\levi\vzFundtrf}{\levi\vzFundtrf})\,\vzFundtrf^2
 &{}=	\vzFundtrf^2\,\trtr{\levi\zFundtrf}{\levi\zFundtrf}-\frac14\trtr{\levi\vzFundtrf^2}{\levi\vzFundtrf^2} \\
 &{}= \frac12\Big(\zFundtrf_{ij}\,\levi*_k\,\zFundtrf_{lm} - \zFundtrf_{lm}\,\levi*_k\,\zFundtrf_{ij}\Big)(\zFundtrf^{ij}\,\levi*^k\,\zFundtrf^{lm} - \zFundtrf^{lm}\,\levi*^k\,\zFundtrf^{ij}) \\
 &{}=: \frac12B_{ijklm}\,B^{ijklm}
\end{align*}
in $p$. Now, we optimize the cited calculations by noting $B_{ijklm}=B_{jiklm}=B_{ijkml}$ and $B_{ijkij}=0$ for every $i,j,k\in\{1,\dots,n\}$ and that the index sets
\begin{align*}
 \boldsymbol B_1 :={}& \{ (i,j,i,j,k) \}_{ijk=1}^n, & 
 \boldsymbol B_2 :={}& \{ (i,j,k,j,k) \}_{ijk=1}^n, \\
 \boldsymbol B_3 :={}& \{ (i,j,k,k,i) \}_{ijk=1}^n, &
 \boldsymbol B_4 :={}& \{ (i,j,j,k,i) \}_{ijk=1}^n
\end{align*}
pairwise intersect in a set $\boldsymbol B$ of indices with $B_{ijklm}=0$ for $(i,j,k,l,m)\in\boldsymbol B$. In particular, we see in $p$ for every $\boldsymbol i\in\{1,\dots,4\}$---\emph{without} using the Einstein summation convention---
\begin{align*}
 \sum_{I\in\boldsymbol B_{\boldsymbol i}} B_IB^I
 ={}&		\sum_{I\in\boldsymbol B_1} B_IB^I
 =			\sum_{i,j,k=1}^n (\zFundtrf_{ij}\,\levi*_i\,\zFundtrf_{jk} - \zFundtrf_{jk}\,\levi*_i\,\zFundtrf_{ij})^2 \\
 ={}&		\sum_{i,k=1}^n (\zFundtrf_{ii}\,\levi*_i\,\zFundtrf_{ik} - \zFundtrf_{ik}\,\levi*_i\,\zFundtrf_{ii})^2
 \ge		\frac1n\sum_{k=1}^n (\sum_{i}^n(\zFundtrf_{ii}\,\levi*_i\,\zFundtrf_{ik} - \zFundtrf_{ik}\,\levi*_i\,\zFundtrf_{ii}))^2 \\
 \ge{}&	\frac1n\sum_{k=1}^n (\sum_{i,j=1}^n(\zFundtrf_{ij}\,\levi*_i\,\zFundtrf_{jk} - \zFundtrf_{jk}\,\levi*_i\,\zFundtrf_{ij}))^2.
\end{align*}
Let $\delta>0$ be arbitrary and denote any constant depending on $\delta$ with $C=\Cof[n][\delta]$. Now, we rejoin the cited literature by using the Codazzi equation (and $\H\equiv\text{const}$) which implies---again \emph{not} using the Einstein summation convention---
\begin{align*}
 (\trtr{\levi\zFundtrf}{\levi\zFundtrf}-\trtr{\levi\vzFundtrf}{\levi\vzFundtrf})\,\vzFundtrf^2
 \ge{}&	\frac2n\sum_{k=1}^n(\sum_{i,j=1}^n(\frac12\levi*_k\vert\zFundtrf_{ij}\vert_{\g}^2 + \zFundtrf_{jk}\,\outrc_{\nu iji}))^2 \\
 \ge{}&	\frac2n\sum_{k=1}^n(\frac12\levi*_k\!\trtr\zFundtrf\zFundtrf - (\trzd{\outric_{\nu}}\zFundtrf)_k)^2 \\
 \ge{}&	\frac{1-\delta}{2n}\sum_{k=1}^n\vert\levi*_k\!\trtr\zFundtrf\zFundtrf\vert^2 - C\,\sum_{k=1}^n\vert\trzd{\outric_{\nu}}\zFundtrf\vert_k^2.
\end{align*}
This proves the claim.
\end{proof}
\bigskip

\bibliography{bib}
\makeatletter
\def\bibindent{10em}
\let\old@biblabel\@biblabel
\def\@biblabel#1{\old@biblabel{#1}\kern\bibindent}
\makeatother																						
\bibliographystyle{alpha}\vfill
\end{document}

%% file: header__arxiv.tex
\usepackage[implicit,naturalnames,hyperfootnotes=false,final]{hyperref}
\usepackage[pdftex]{color}
\usepackage{lmodern}
\usepackage{xcolor}
\usepackage{enumitem}
\usepackage{amsmath,amssymb,amsthm}
\usepackage{bbm,amsfonts,mathrsfs,calligra,dsfont}
\usepackage{nicefrac}
\usepackage{esint}
\usepackage{mathtools}
\usepackage{ifthen}
\usepackage{xparse}
\usepackage{xspace}
\usepackage{calc}
\usepackage[nostandardtensors]{simpletensors}
\usepackage{partial}
\usepackage{bracket-hack}
\usepackage{auto-eqlabel}

\newenvironment{itemizeequation}%
 {\par\noindent\minipage[r]{\textwidth-3.5em}\csname equation*\endcsname}%
 {\vspace{.5em}\csname endequation*\endcsname\endminipage\par\noindent}

\newtheoremstyle{mytheorem}{3pt}{}{\itshape}{}{\bfseries}{\nopagebreak\newline}{.5em}{}%
\newtheoremstyle{mydefinition}{3pt}{}{}{}{\bfseries}{\nopagebreak\newline}{.5em}{}%
\gdef\mytheorem{\theoremstyle{mytheorem}}
\gdef\mydefinition{\theoremstyle{mydefinition}}
\gdef\mytheoremcounter{theorem}
\mytheorem
\newtheorem{theorem}{Theorem}[section]
\newtheorem{corollary}[\mytheoremcounter]{Corollary}
\newtheorem{lemma}[\mytheoremcounter]{Lemma}
\newtheorem{proposition}[\mytheoremcounter]{Proposition}
\newtheorem{assumptions}[\mytheoremcounter]{Assumptions}
\mydefinition
\newtheorem{definition}[\mytheoremcounter]{Definition}
\newtheorem{notation}[\mytheoremcounter]{Notation}
\theoremstyle{remark}
\newtheorem{remark}[\mytheoremcounter]{Remark}

\makeatletter
\NewDocumentCommand\indexsmall{m}{\@ifnextchar\egroup{#1}{\indexsmallsnd{#1}}}
\makeatother
\NewDocumentCommand\indexsmallsnd{mg}{\IfValueTF{#2}{#1{#2}}{\indexsmallthd{#1}}}
\NewDocumentCommand\indexsmallthd{mm}{\expandafter\indexsmallfth\expandafter#1#2}
\NewDocumentCommand\indexsmallfth{mg}{\IfValueTF{#2}{#1{#2}}{\indexsmallfith{#1}}}
\NewDocumentCommand\indexsmallfith{mm}{#1\ifx\indexsmall#2\hspace{-.075em}\fi#2}
\gdef\ii{\indexsmall{I}}\gdef\ij{\indexsmall{J}}\gdef\ik{\indexsmall{K}}\gdef\il{\indexsmall{L}}
\gdef\oi{i}\gdef\oj{j}
\gdef\ui{\alpha}

\NewDocumentCommand\outsymbolsmall{om}{\bar{\IfValueTF{#1}{#1{#2}}{#2}}}
\NewDocumentCommand\outsymbol{om}{\overline{\IfValueTF{#1}{#1{#2}}{#2}}}
\NewDocumentCommand\unisymbol{om}{\widehat{\IfValueTF{#1}{#1{#2}}{#2}}}

\NewDocumentCommand\outtensor{d()}{\IndexSymbol(\outsymbol[#1])}
\makeatletter
\newdimen\middle@width
\newcommand*\phantomas[3][c]{\ifmmode\makebox[\widthof{$#2$}][#1]{$#3$}\else\makebox[\widthof{#2}][#1]{#3}\fi}
\NewDocumentCommand\tracefree{m}%
{\setlength\middle@width{\widthof{\ensuremath{#1}}}%
 #1\hskip-\middle@width%
 \phantomas[r]{#1}{\vphantom{\ensuremath{#1}}^{\,\!^\circ}}}%
\NewDocumentCommand\mean{om}%
{\IfValueTF{#1}%
 {\setlength\middle@width{\widthof{\ensuremath{#1#2}}}%
  \phantomas{#1#2}{{#1\hspace{-.1em}\textbf{\small$\boldsymbol\backslash$}\!}}%
  \hskip-\middle@width#1#2}%
 {\mathchoice{\mean[\displaystyle]{#2}}{\mean[\textstyle]{#2}}{\mean[\scriptstyle]{#2}}{\mean[\scriptscriptstyle]{#2}}}}
\gdef\meanH{\hspace{-.08em}\mean{\hspace{.08em}\H}}%
\makeatother

\NewTensor*\mass{\MakeSymbol[\outsymbol m]<{\outsymbol m}>{\vvecmass{\outsymbol m}}}
\NewTensor[\MakeSymbol<\outsymbol>\vec]*\outcenterz[\!]Z
\def\totalmass{\mass*}

\let\oldPhi\Phi\RenewTensor\Phi\oldPhi
\NewTensor[\outsymbol]\outPhi\oldPhi
\let\oldvarphi\varphi\RenewTensor*\varphi\oldvarphi
\let\oldPsi\Psi\RenewTensor*\Psi<\!>\oldPsi
\NewTensor[\outsymbol]*\outPsi\oldPsi
\NewTensor[\outsymbol]*\outx x
\NewTensor[\outsymbol]*\outy y
\NewTensor[\outsymbol]*\outp p

\NewDocumentCommand\intervalI{s}{\IfBooleanTF{#1}I{\ensuremath{\intervalI*}\xspace}}
\NewDocumentCommand\intervalJ{s}{\IfBooleanTF{#1}J{\ensuremath{\intervalJ*}\xspace}}

\gdef\hyperbolich{{\mathcal h}}
\let\AdSa\antidesittera
\gdef\referencer{r}

\gdef\houtg{\outg[\hyperbolich]}
\gdef\hcenterz{\centerz[\hyperbolich]}
\gdef\houtcenterz{\outcenterz[\hyperbolich]}

\gdef\houtsc{\outsc[\hyperbolich]}

\gdef\hg{\g[\hyperbolich\,]}

\gdef\houttr{\outtr[\hyperbolich\,]}
\gdef\htr{\tr[\hyperbolich]}
\gdef\hH{\H[\hyperbolich]}
\gdef\hzFund{\zFund[\hyperbolich]}
\gdef\hzFundtrf{\zFundtrf[\hyperbolich\hspace{.1em}]}
\gdef\houtric{\outric[\hyperbolich]}
\gdef\houtlevi{\outlevi[\hyperbolich]}
\gdef\houtdiv{\outdiv[\hyperbolich]}

\gdef\houtHess{\outHess[\hyperbolich\,]}

\gdef\hlevi{\levi[\hyperbolich]}

\gdef\sinh{\sinhsnd{sinh}}\gdef\cosh{\sinhsnd{cosh}}

\NewDocumentCommand\sinhsnd{mg}{\IfValueTF{#2}{\operatorname{\text{\normalfont #1}(#2)}}{\sinhthd{#1}}}
\NewDocumentCommand\sinhthd{md()}{\IfValueTF{#2}{\sinhsnd{#1}{#2}}{\operatorname{\text{\normalfont #1}}}}

\gdef\hmug{\mug[\hyperbolich]}
\gdef\houtmug{\outmug[\hyperbolich\hspace{.05em}]}
\gdef\hnu{\nu[\hyperbolich]}

\gdef\refoutg{\outg[\referencer]}

\gdef\refoutsc{\outsc[\referencer]}

\gdef\AdSoutg{\outg[\totalmass]}
\gdef\AdSoutric{\outric[\totalmass]}
\gdef\AdSoutsc{\outsc[\totalmass]}

\gdef\sphtr{\tr[\sphg*\hspace{.05em}]}
\gdef\sphmug{\mug[\hspace{.05em}\sphg*\hspace{.05em}]}
\gdef\euclideane{e}

\gdef\eukoutg{\outg[\euclideane]}
\gdef\eukoutlevi{\outlevi[\euclideane]}

\gdef\eukg{\g[\euclideane\hspace{.05em}]}

\gdef\eukH{\H[\euclideane]}
\gdef\eukzFund{\zFund[\euclideane]}
\gdef\eukzFundtrf{\zFundtrf[\euclideane\hspace{.08em}]}
\NewDocumentCommand\rad{t@}{\IfBooleanTF{#1}\radsnd{\radsnd\outx}}
\gdef\radsnd#1{|#1|}
\gdef\atime{\tau}

\NewDocumentCommand\volume{od<>om}
{\IfValueTF{#3}{\volume[#3]<#2>{#4}}%
 {\IfValueTF{#1}%
  {\vphantom{\left|#1\right|}^{#1}\IfValueTF{#2}{_{#2}}\relax\hspace{-.25em}}%
	{\IfValueTF{#2}{\vphantom{\left|#1\right|}_{#2}\hspace{-.25em}}\relax}\left|#4\right|}}%
\gdef\hvolume{\volume[\hyperbolich]}
\gdef\eukvolume{\volume[\euclideane]}
\gdef\lieD#1#2{\mathfrak L_{#1}#2}

\undef\c\NewDocumentCommand\c{sG{c}}{\IfBooleanTF{#1}{#2}{\IndexSymbol*{#2}}}
\gdef\cSob{\c{\c_S}}
\NewDocumentCommand\Cof{G{C}d<>d()}
{\IfValueTF{#3}{\Cof{#1}<\IfValueTF{#2}{#2,}\relax{#3}>}{\CofSnd{#1\IfValueTF{#2}{[#2]}\relax}}}
\NewDocumentCommand\CofSnd{G{C}d<>o}
{\IfValueTF{#3}{\CofSnd{#1}<\IfValueTF{#2}{#2,}\relax{#3}>}{\mathop{#1\IfValueTF{#2}{(#2)}\relax}}}

\NewTensor\M[\hspace{-.05em}]<\hspace{-.05em}>{{\normalfont\Sigma}}
\NewTensor[\outsymbol]\outM[\hspace{-.025em}]{{\normalfont\textrm M}}

\NewTensor[\unisymbol]\uniM[\hspace{-.025em}]{{\textrm M}}

\NewTensor[\outsymbol]\outexp{\exp*}
\global\let\graphfsymbol\varsigma
\NewTensor*\graphf[\hspace{-.05em}]\graphfsymbol
\let\graphft s
\NewTensor*\solutionf f
\gdef\fX{f_{\hspace{-.08em}X}}
\gdef\fXsolX{{\fX^\solX}}
\let\solX\alpha
\NewTensor\graphF[\!]S
\let\graphFt S
\NewTensor\graphpushforward[\!]{\pushforward{S}\hspace{-.05em}}
\NewTensor\conformalf v
\NewTensor\sphg{\Omega\vphantom{\Omega}}
\newcommand\sphlaplace{\laplace[\sphg*]}
\newcommand\sphlevi{\levi[\sphg*]}

\gdef\sphHess{\Hess[\sphg\,]}

\NewTensor*\rnu[\!]u
\NewTensor*\lapse[\!]u
\NewTensor\rbeta[\!]\beta
\NewTensor[\MakeSymbol<>\vec]*\centerz[\!]<\hspace{-.05em}>z
\undef\metric\NewTensor[\newmathcal]\metric[\hspace{-.15em}]<\hspace{.05em}>g
\let\g\metric
\NewTensor[\newmathcal]\metrictwo[\hspace{-.05em}]<\hspace{-.05em}>h

\NewTensor[\outsymbol[\newmathcal]]\outmetric<\hspace{.05em}>g
\NewTensor[\outsymbol[\newmathcal]]\outfg<\hspace{.05em}>{e\hspace{-.05em}r}
\NewTensor[\outsymbol]\ralpha<\hspace{-.05em}>\alpha
\let\outg\outmetric
\NewTensor[\mathfrak]\massaspect m
\gdef\trmassaspect{\sphtr\hspace{.08em}\massaspect}
\NewTensor[\unisymbol[\newmathcal]]\unimetric[\!\!]<\hspace{.05em}>g
\let\unig\unimetric
\NewTensor[\newmathcal]\riem R
\let\rc\riem
\NewTensor[\outsymbol[\newmathcal]]\outriem R
\NewTensor[\unisymbol[\newmathcal]]\uniriem R
\let\outrc\outriem
\NewTensor[\normalfont\textrm]\ricci{Ric}
\let\ric\ricci
\let\Ric\ricci
\NewTensor[\outsymbol[\normalfont\textrm]]\outricci{Ric}
\let\outric\outricci
\NewTensor[\unisymbol[\normalfont\textrm]]\uniricci{Ric}

\NewTensor[\newmathcal]\scalar[\!]S
\NewTensor[\outsymbol[\newmathcal]]\outscalar S
\let\outsc\outscalar
\NewTensor[\unisymbol[\newmathcal]]\uniscalar S

\NewTensor[\newmathcal]\gauss G
\NewTensor[\textrm]\II k

\let\zFund\k
\NewTensor[\tracefree]\IItrf[\!]{\textrm k}

\let\zFundtrf\ktrf
\NewTensor[\outsymbol[\textrm]]\outII[\hspace{-.05em}]k

\let\outzFund\outk
\NewDocumentCommand\troutzFund{D<>{}O{}}{\tr<#1>[#2]\hspace{.05em}\outzFund[#2]}
\NewDocumentCommand\outzFundnu{D<>{}O{}}{\mathop{{\outzFund[#2]_{\nu<#1>[#2]}}}}
\NewDocumentCommand\outzFundnunu{D<>{}O{}}{\mathop{{\outzFund[#2]_{\nu<#1>[#2]\nu<#1>[#2]}}}}
\NewTensor[\mathcal]\mc[\!]<\hspace{-.05em}> H
\let\H\mc
\NewTensor[\outsymbol[\mathcal]]\outmc[\!]<\hspace{-.05em}>H

\let\oldnu\nu
\NewTensor*\normal[\!]\oldnu
\let\nu\normal
\NewTensor[\outsymbol]\timenormal[\!]\vartheta

\NewTensor*\laplace[\!]\Delta
\NewTensor[\outsymbol]*\outlaplace[\!]\Delta
\NewTensor\Hess[\!]{\text{Hess}}
\NewTensor\Hesstrf[\!]{\text H\tracefree{\text{es}}\text s}
\NewTensor[\outsymbol]\outHess[\!]{\text{Hess}}
\undef\div\NewTensor\div[\hspace{-.05em}]{\text{div}}
\NewTensor[\outsymbol]\outdiv[\!]{\text{div}}
\NewTensor[\unisymbol]\unidiv[\!]{\text{div}}
\NewTensor*\tr[\hspace{0em}]{\text{tr}}
\NewTensor[\outsymbol]*\outtr{\text{tr}}
\NewTensor[\unisymbol]*\unitr[\!]{\text{tr}}
\NewTensor\levi<\MakeSymbol<>{\!}>{\MakeSymbol<\Gamma>\nabla}
\NewTensor[\outsymbol]\outlevi<\MakeSymbol<>{\!}>{\MakeSymbol<\Gamma>\nabla}
\NewTensor[\unisymbol]\unilevi<\MakeSymbol<>{\!}>{\MakeSymbol<\Gamma>\nabla}
\NewTensor*\vol{\text{vol}}

\NewTensor*\mug[\!]\mu
\NewTensor[\outsymbol]*\outmug[\!]\mu
\NewTensor[\outsymbol]\momentum[\hspace{-.05em}]\Pi
\NewTensor[\outsymbol]\einstein{{\normalfont\textrm G}}
\NewTensor[\outsymbol]\outmomden[\!]J
\NewTensor[\outsymbol]\outenden[\hspace{-.05em}]\varrho

\NewTensor[\textrm]*\jacobiext[\hspace{-.05em}]<\hspace{-.05em}> L

\NewDocumentCommand\trzd{smm}{{#2}\odot{#3}}
\NewDocumentCommand\htrzd{smm}{{#2}{}\hspace{.15em}^\hyperbolich\hspace{-.32em}\odot{#3}}
\NewDocumentCommand\trtr{smm}
{\ifthenelse{\equal{#2}{#3}}{\left|#2\right|_{\hspace{-.05em}\g*}^2}{\tr(\trzd{#2}{#3})}}
\NewDocumentCommand\htrtr{smm}
{\ifthenelse{\equal{#2}{#3}}{\left|#2\right|_{\hg*}^2}{\htr(\htrzd{#2}{#3})}}

\NewDocumentCommand\outc{G{c}}{\IndexSymbol(\outsymbol)*{#1}}\let\oc\outc

\NewTensor\decay\varpi
\NewTensor\scdecay\upsilon
\NewTensor\outve\ve
\NewDocumentCommand\outck{d<>o}{\outc<#1>[#2]_{\outzFund*}\vphantom{\outc_{\outzFund*}}}
\gdef\outtr{\IndexSymbol(\outsymbol)[\!]{\text{tr}}}
\NewDocumentCommand\outtrzd{D<>{}O{}mm}
{\mathop{{#3\IndexSymbol(\outsymbol)[#2]<#1>\odot{#4}}}}
\NewDocumentCommand\outtrtr{D<>{}O{}mm}
{\mathop{{\ifthenelse{\equal{#3}{#4}}%
 {{\left|#3\right|_{\outg<#1>[#2]}^2}}%
 {\mathop{{\outtr<#1>[#2](\outtrzd<#1>[#2]{#3}{#4})}}}}}}%
\NewDocumentCommand\houttrzd{mm}
{\mathop{#1\mathop{\hspace{.05em}^{\hyperbolich}\hspace{-.05em}\outsymbol\odot}#2}}
\NewDocumentCommand\houttrtr{mm}
{\ifthenelse{\equal{#1}{#2}}%
 {\left|#1\right|_{\houtg*}^2}%
 {\mathop{\houttr(\houttrzd{#1}{#2})}}}%
\NewTensor*\eflap[\!]f
\NewTensor*\ewlap[\!]\lambda
\NewTensor*\efjac[\!]{\mathcal f}
\NewTensor*\ewjac[\hspace{-.05em}]<\!>\kappa
\NewTensor*\funcg[\!]g
\NewTensor*\funch[\!]h
\makeatletter
\gdef\boostLp^#1(#2){\Lp^{#1}(#2)\hspace{-.05em}\boost{\vphantom{#2}}}
\gdef\deformLp^#1(#2){\Lp^{#1}(#2)\hspace{-.05em}\deform{\vphantom{#2}}}

\NewDocumentCommand\normof{d<>od<>m}
{\IfValueTF{#1}%
 {\IfValueTF{#2}%
  {\IfValueTF{#3}{\normofleft<#1#3>[#2]{#4}}{\normofleft<#1>[#2]{#4}}}%
  {\IfValueTF{#3}{\normofleft<#1#3>{#4}}{\normofleft<#1>{#4}}}}%
 {\IfValueTF{#2}%
  {\IfValueTF{#3}{\normofleft<#3>[#2]{#4}}{\normofleft[#2]{#4}}}%
  {\IfValueTF{#3}{\normofleft<#3>{#4}}\relax}}%
 \lvert#4\rvert}
\NewDocumentCommand\normofleft{d<>om}%
{\vphantom{\lvert#3\rvert}\IfValueTF{#1}{_{#1}}\relax\IfValueTF{#2}{^{#2}}\relax}
\gdef\sphnormof{\normof[\sphg*]}

\let\Hradius\sigma
\let\Aradius R
\gdef\eukAradius{{\vphantom{R}^{\euclideane\hspace{-.2em}}R}}
\def\rradius{{\underline r}}
\def\Rradius{{\overline r}}

\NewTensor\mHaw{{\normalfont m_{\text{H}}}}
\NewTensor\HmHaw{{\normalfont m_{\text{H}}^{\hspace{-.05em}\text h}}}
\NewTensor\HmInt{{\normalfont m_{\text{I}}^{\text h}}}

\NewTensor[\MakeSymbol<\outsymbol>\vec]\impuls P

\undef\d
\NewDocumentCommand\d{s}{\IfBooleanTF{#1}\relax{\mathop{}\!}\mathrm d}

\newcommand\dr{{\d r}}

\newcommand\pullback[1]{{#1}^*}
\newcommand\pushforward[1]{{#1}_*}

\DeclareMathOperator\id{id}

\DeclareMathOperator\graph{graph}%
\DeclareMathOperator\hgraph{\vphantom{h}^{\hyperbolich\hspace{-.05em}}graph}

\newcommand\R{\mathds{R}}
\newcommand\N{\mathds{N}}
\newcommand\X{\mathfrak{X}}
\newcommand\Lp{{\normalfont\textrm L}}
\newcommand\Wkp{{\normalfont\textrm W}}
\newcommand\Hk{{\normalfont\textrm H}}
\newcommand\Ck{\mathcal C}

\NewDocumentCommand\sphere{t^}{\mathds S\IfBooleanTF{#1}^{^2}}
\gdef\hsphere{\vphantom{S}^\hyperbolich\sphere}
\gdef\euksphere{\vphantom{S}^{\euclideane\hspace{-.05em}}\sphere}

\gdef\sphball{\vphantom{B}^{\sphg*}\!B}
\gdef\eukball{\vphantom{B}^{\euclideane}\!B}
\gdef\hball{\vphantom{B}^{\hyperbolich}\!B}

\NewDocumentCommand\hyperbolicspace{t^}{\IfBooleanTF{#1}{\mathbbm H^}{\hyperbolicspace^3}}%
\newcommand\hyperbolicisometry{\boldsymbol\gamma}
\NewTensor*\hiso<\!>\hyperbolicisometry
\NewDocumentCommand\HISO{t_}{\boldsymbol\Gamma\IfBooleanTF{#1}{\hspace*{-.2em}_}\relax}
\let\ve\varepsilon

\usepackage{urwchancal}
\DeclareMathAlphabet{\mathcal}{OT1}{pzc}{m}{n}
\let\newmathcal\mathcal

\makeatletter
\def\my@overarrow@#1#2#3{\vbox{\ialign{##\crcr#1#2\crcr\noalign{\kern-2.75\p@\nointerlineskip}$\m@th\hfil#2#3\hfil$\crcr}}}
\gdef\vvecmass{%
 \mathpalette{\my@overarrow@\vectfillb@}}
\newcommand{\vecbar}{%
  \scalebox{1}{$\relbar$}}
\def\vectfillb@{\arrowfill@\vecbar\vecbar{\raisebox{-3.65pt}[\p@][\p@]{$\mathord\mathchar"017E$}}}
\makeatother

\NewDocumentCommand\Cka{mt^}{\IfBooleanTF{#2}{\Ckasnd{#1}^}{\Ckasnd{#1}\relax\relax}}
\NewDocumentCommand\Ckasnd{mmmt_}{\IfBooleanTF{#4}{\Ckathd{#1}#2{#3}_}{\Ckathd{#1}#2{#3}\relax\relax}}
\NewDocumentCommand\Ckathd{mmmmmt^}{\IfBooleanTF{#6}{\Ckafth{#1}#2{#3}#4{#5}^}{\Ckafth{#1}#2{#3}#4{#5}\relax\relax}}
\NewDocumentCommand\Ckafth{mmmmmmm}{\short{$\Ck#2{#3}\vphantom{\Ck}#4{#5}\vphantom{\Ck}#6{#7}$-asymp\-to\-tic\-ally #1}}
\gdef\Ckah{\Cka{hyperbolic}}
\gdef\Ckae{\Cka{Euclidean}}

\NewDocumentCommand\short{ms}{#1\IfBooleanTF{#2}\relax\xspace}

\NewDocumentCommand\functionpart{mod()om}%
{{#5}^{\IfValueTF{#2}{\hspace{-.05em}\IndexSymbol[#2]<#3>{#1}\hspace{-.05em}}{\IfValueTF{#4}{\hspace{-.05em}}\relax\IndexSymbol[#4]<#3>{#1}\IfValueTF{#4}{\hspace{-.05em}}\relax}}}
\gdef\boost{\functionpart{\text{\normalfont b}}}
\gdef\deform{\functionpart{\text{\normalfont d}}}

\expandafter\let\expandafter\dddd\csname'\endcsname
\expandafter\gdef\csname'\endcsname{\ifmmode\hspace{.05em}\else\expandafter\dddd\fi}